\documentclass[a4paper,12pt]{amsart}
\usepackage{import
}

\usepackage[utf8]{inputenc}
\usepackage[english]{babel}             
\usepackage{geometry}
\usepackage{amsfonts,amsmath,amssymb,amsthm}    
\usepackage{mathtools}
\usepackage{bbm}
\usepackage{esint}
\usepackage{lmodern}                    
\usepackage{graphicx}                   
\usepackage{xcolor}                     
\usepackage[backend=biber,style=alphabetic,bibencoding=utf8,isbn=false,url=false,sorting=nty,doi=false,giveninits=true,maxbibnames=10,maxalphanames=10]{biblatex}  
\AtEveryBibitem{\clearfield{month}}
\usepackage{csquotes}                   
\usepackage{enumitem}                   
\usepackage{parskip}
\usepackage[margin=1cm,labelfont=bf]{caption}   
\usepackage{hyperref}                   
\usepackage{float}

\savegeometry{paper}                    

\graphicspath{{./files/figures/}}       

\AtBeginBibliography{\small}            

\allowdisplaybreaks                     

\setlist[enumerate,1]{label=(\roman*)}  

\numberwithin{equation}{section}

\DeclareFieldFormat{extraalpha}{#1}     
\DeclareLabelalphaTemplate{
  \labelelement{
    \field[final]{shorthand}
    \field{label}
    \field[strwidth=2 , strside=left,ifnames=1]{labelname}
    \field[strside=left,ifnames=2-,varwidthlist=true]{labelname}
  }
}
\DeclareLabelalphaNameTemplate{
  \namepart[use=true, pre=true, strwidth=1, compound]{prefix}
  \namepart{family} 
}

\renewbibmacro{in:}{}                   
\DeclareFieldFormat[article,periodical]{volume}{\mkbibbold{#1}}
\DeclareFieldFormat[article,periodical]{journaltitle}{#1\isdot}
\DeclareFieldFormat[article,periodical]{title}{#1}

\newcommand{\mbbS}{\mathbb{S}}

\newcommand{\mcalA}{\mathcal{A}}
\newcommand{\mcalB}{\mathcal{B}}

\newcommand{\mcalF}{\mathcal{F}}

\newcommand{\mcalV}{\mathcal{V}}

\newcommand{\naturals}{\mathbb{N}}
\newcommand{\N}{\naturals}

\newcommand{\reals}{\mathbb{R}}
\newcommand{\R}{\reals}
\newcommand{\complex}{\mathbb{C}}
\newcommand{\C}{\complex}

\newcommand{\field}{\mathbb{F}}

\newcommand{\eps}{\varepsilon}
\newcommand{\fourier}[1]{\widehat{#1}}

\newcommand{\gs}{>}
\newcommand{\ls}{<}
\newcommand{\conj}[1]{\overline{#1}}
\newcommand{\norm}[1]{\left\|  #1  \right\|}
\newcommand{\abs}[1]{\left| #1 \right|}
\newcommand{\vol}[1]{\abs{#1}}
\newcommand{\of}[1]{\left( #1 \right)}

\newcommand{\inv}[1]{\frac{1}{#1}}

\newcommand{\inset}[1]{\left \{ #1 \right \}}
\newcommand{\angles}[1]{\left\langle #1 \right\rangle}
\newcommand{\inner}[1]{\angles{#1}}

\newcommand{\charf}[1]{\mathbbm{1}_{#1}}

\DeclareMathOperator{\tRe}{Re}

\DeclareMathOperator{\diff}{d\!}

\newcommand{\theoremname}{Theorem}
\newcommand{\lemmaname}{Lemma}
\newcommand{\propositionname}{Proposition}
\newcommand{\definitionname}{Definition}
\newcommand{\corollaryname}{Corollary}
\newcommand{\claimname}{Claim}

\newcommand{\examplename}{Example}
\newcommand{\remarkname}{Remark}
\newcommand{\commentsname}{Comments}
\newcommand{\notename}{Note}
\newcommand{\notationname}{Notation}

\newcounter{allcounter}[section]

\newtheorem{numthm}[allcounter]{\theoremname}
\newtheorem{theorem}{\theoremname}

\newtheorem*{lemma}{\lemmaname}                     
\newtheorem{numlemma}[allcounter]{\lemmaname}

\newtheorem{numprop}[allcounter]{\propositionname}

\newtheorem{numcor}[allcounter]{\corollaryname}

\newtheorem*{obs}{\remarkname}                      
\newtheorem{numremark}[allcounter]{\remarkname}

\newcommand\ee{v^*}

\newcommand{\Fock}{\mcalF^2}

\newcommand{\Fnorm}[1]{\norm{#1}_{\Fock}}
\newcommand{\Om}{\Omega}
\newcommand{\om}{\omega}
\newcommand{\dso}{\delta_{s_0}}
\newcommand{\STFT}{\mcalV}
\newcommand{\Barg}{\mcalB}
\newcommand{\bol}{\boldsymbol{\omega}}
\newcommand{\BergD}{\mcalA_{\alpha}(D)}
\newcommand{\BergC}{\mcalA_{\alpha}(\C^+)}

\newcommand{\K}{\mathcal K}
\newcommand{\e}{\varepsilon}
\newcommand{\p}{\partial}
\newcommand{\mc}{\mathcal}
\newcommand{\mb}{\mathbb}
\DeclareMathOperator*{\ddiv}{div}
\renewcommand{\Re}{\operatorname{Re}}
\renewcommand{\Im}{\operatorname{Im}}

\title[Stability of Faber-Krahn for the STFT]{Stability of the Faber-Krahn inequality for the Short-time Fourier Transform}

\author{Jaime Gómez}
\address{Mathematics section, EPFL, Lausanne, Switzerland.}
\email{jaime.gomezramirez@epfl.ch}
\author{Andr\'e Guerra}
\address{Institute for Theoretical Studies, ETH Zürich, Zürich, Switzerland.}
\email{andre.guerra@eth-its.ethz.ch}
\author{Jo\~ao P. G. Ramos}
\address{Department of Mathematics, ETH Zürich, Zürich, Switzerland.}
\email{joao.ramos@math.ethz.ch}
\author{Paolo Tilli}
\address{Dipartimento di Scienze Matematiche, Politecnico di Torino, Corso Duca degli Abruzzi 24, 10129 Torino, Italy.}
\email{paolo.tilli@polito.it}

\begin{document}
\maketitle

\begin{abstract}
We prove a sharp quantitative version of the Faber--Krahn inequality for the short-time Fourier transform (STFT). 
To do so, we consider a deficit $\delta(f;\Om)$ which measures by how much the STFT of a function $f\in L^2(\R)$ fails to be optimally concentrated on an arbitrary set $\Omega\subset \R^2$ of positive, finite measure. We then show that an optimal power of the deficit $\delta(f;\Om)$ controls both the $L^2$-distance of $f$ to an appropriate class of Gaussians and the distance of $\Omega$ to a ball, through the Fraenkel asymmetry of $\Omega$. Our proof is completely quantitative and hence all constants are explicit. We also establish suitable generalizations of this result in the higher-dimensional context. 

\end{abstract}
\vspace{0.7cm} 

\section{Introduction}

\subsection{Main results}
Given a function $g\in L^2(\R)$ (called the \emph{window}), the
\emph{short-time Fourier transform (STFT)}
of a function $f \in L^2(\R)$ is
usually defined as
\begin{equation}
    \label{defSTFT}
    V_gf(x,\omega) = \int_{\R} e^{- 2 \pi i t \omega} f(t) \overline{g(x-t)} \diff t.
\end{equation}
This transform plays a distinguished role in different areas of mathematics, including time-frequency analysis \cite{Grochenig} and signal processing \cite{Mallat2009}, mathematical physics \cite{Lieb2001}, where it is also known as the \textit{coherent state transform}, and semiclassical and microlocal analysis \cite{Hormander2003,Tataru}. 

From the point of view of time-frequency analysis, the STFT is a
measure of the ``instantaneous frequency'' of the signal $f$ at each point, in analogy to what a music score does. 
As the notion of ``instantaneous frequency'' is not well-defined for generic signals, due to the uncertainty principle, the STFT can only concentrate a limited amount of its $L^2$-norm on any
set $\Omega\subset \R^2$ with finite Lebesgue measure $|\Omega|$,
and finding explicit bounds in terms of $|\Omega|$  
is an important issue in time-frequency analysis. For a general window $g$, this appears to be  extremely challenging and only suboptimal bounds have been obtained: we refer the reader to the work of E.\ Lieb \cite{Lieb} for what is, to our knowledge, the current best result at this level of generality.

For very regular windows, however, the situation improves. In particular,
in the relevant case
(extensively studied in the literature also in connection with
the spectrum of localization operators in the radially symmetric case, 
see e.g.\ \cite{Abreu2012,Daubechies,Galbis,RamosLocalization})
where
$g=\varphi$ is the \textit{Gaussian} window
\begin{equation}\label{eqn:GaussianWindow}
    \varphi(x) = 2^{1/4} e^{-\pi x^2}, \quad x \in \R,
\end{equation}
a complete solution to this concentration problem 
has recently been 
given in \cite{NicolaTilli},
thus proving a conjecture from \cite{Abreu2021} (see also \cite{Donoho1989}).
Denoting by
$\STFT f := V_{\varphi}f$ the STFT with the Gaussian window $\varphi$ defined in \eqref{eqn:GaussianWindow},
the main result of \cite{NicolaTilli} can be stated as follows:
\begin{theorem}[\cite{NicolaTilli}; Faber-Krahn inequality for the STFT]\label{thm:FKforSTFT}
    If $\Omega\subset \R^2$ is a measurable set with finite Lebesgue measure $|\Omega|>0$, 
    and $f\in L^2(\R)\setminus\{0\}$ is an arbitrary function,
    then
    \begin{equation}\label{eqn:FKInequality}
    \frac{\int_{\Omega} |\STFT f(x,\omega)|^2 \diff x \diff \omega}
    {\Vert f\Vert_{L^2(\R)}^2}
    \leq 1-e^{- \vol{\Om}}.
    \end{equation}
    Moreover, 
    equality is attained  if and only if $\Om$ coincides (up to a set of measure zero) with a ball
    centered at some $z_0=(x_0, \om_0)\in\R^2$  and, at the same time, 
$f$ is a function of the kind
    \begin{equation}\label{eqn:FKEqualityFunctions}
        f(x) = c\, \varphi_{z_0}(x) , \qquad  \varphi_{z_0}(x) := e^{2 \pi i \om_0 x} \varphi(x-x_0),
    \end{equation}
    for some $c\in \C\setminus\{0\}$.
\end{theorem}

Note that the optimal functions in \eqref{eqn:FKEqualityFunctions} are scalar multiples of the Gaussian window defined in \eqref{eqn:GaussianWindow}, translated and modulated according to the center of the ball $\Omega$.


This result, which improves upon Lieb's uncertainty principle \cite{Lieb}, has inspired 
other subsequent works: \cite{RamosTilli}, where a similar
result is extended to the case of Wavelet transforms; \cite{Kulikov}, where Kulikov used techniques 
inspired by those of \cite{NicolaTilli} to prove some contractivity conjectures; and \cite{Frank}, where R. Frank uses the same circle of ideas to generalize a series of entropy-like inequalities. We also refer the reader to \cite{KNOT, KalajRamos, Kalaj1, Kalaj2, KalajMelentijevic, Melentijevic, NicolaTilli2} and the references therein for further closely related work.

In the present paper we investigate the \emph{stability} of Theorem \ref{thm:FKforSTFT}: given $\Om \subset \R^2$ and $f\in L^2(\R)$ which are almost optimal, in the sense that they almost saturate  inequality \eqref{eqn:FKInequality}, can we infer (and to what extent) that $\Omega$ is close to a ball and that $f$ is close to a function of the form \eqref{eqn:FKEqualityFunctions}? To
formulate this question precisely,
a crucial point  is choosing how to measure \textit{almost optimality} as well as \textit{closeness}. 

To measure almost optimality in \eqref{eqn:FKInequality} for a pair $(f,\Omega)$, we will consider the \emph{combined deficit}
\begin{equation}\label{eqn:Deficit}
    \delta(f;\Om) \coloneqq 1- \frac{\displaystyle \int_{\Omega} |\mathcal{V}f(x,\omega)|^2 \, \diff x  \diff \omega}{\displaystyle (1-e^{-\vol{\Om}}) \|f\|_{L^2(\R)}^2},
\end{equation}
while we will use the \emph{Fraenkel asymmetry} of $\Om \subset \R^2$ to measure its distance to a ball:
\begin{equation}\label{eqn:FraenkelAsymmetry}
    \mcalA (\Om) \coloneqq \inf \inset{ \frac{\vol{\Om \triangle B(x,r)}}{\vol{\Om}} \colon \vol{B(x,r)} = \vol{\Om} , r \gs 0, x \in \R^2 } .
\end{equation}
The Fraenkel asymmetry is a natural notion of asymmetry and it is often used to formulate the stability of geometric and functional inequalities, such as the isoperimetric inequality \cite{CicaleseLeonardi,FigalliMaggiPratelli,Fuglede1989,FuscoMaggiPratelli} or the Faber--Krahn inequality for the Dirichlet Laplacian \cite{AllenKriventsovNeumayer,BrascoDePhilippisVelichkov}.

Our main result reads as follows:

\begin{numthm}[Stability of the Faber-Krahn inequality for the STFT]\label{thm:Stability}
    There is an explicitly computable constant $C>0$ such that, for all  measurable sets $\Om \subset \R^2$ with finite measure
    $|\Omega|>0$ and all functions $f \in L^2(\R)\backslash\{0\}$, we have 
    \begin{equation}\label{eqn:StabilityFunction}
        \min_{z_0\in \C, |c|=\|f\|_{2}} \frac{\|f - c\, \varphi_{z_0} \|_2}{\|f\|_2} \leq C\big( e^{|\Omega|} \delta(f;\Om)\big)^{1/2}.
    \end{equation}
    Moreover, for some explicit constant $K(|\Omega|)$ we also have
    \begin{equation}\label{eqn:StabilitySet}
        \mcalA(\Om) \leq K(|\Omega|) \delta(f;\Om)^{1/2} .
    \end{equation}
\end{numthm}

\begin{numremark}[Sharpness]\label{rmk:sharpness}
In Theorem \ref{thm:Stability} the factor $\delta(f;\Om)^{1/2}$ in \eqref{eqn:StabilityFunction} and \eqref{eqn:StabilitySet} cannot be replaced by $\delta(f;\Om)^{\beta},$ for any $\beta > 1/2$. Similarly, the dependence on $|\Omega|$ in \eqref{eqn:StabilityFunction} 
 is also sharp,
in the sense that factor $e^{|\Omega|/2}$ cannot be replaced
by $e^{\beta |\Omega|}$
for any $\beta<1/2$. We refer to Section \ref{sec:sharpness} for proofs of these claims.
\end{numremark}

\begin{numremark}[Higher dimensions]\label{rmk:higherd}
    There is a natural generalization of the $STFT$ to functions $f\in L^2(\R^d)$, for any $d\geq 1$. In Section \ref{sec:Generalize} we show that a more general version of Theorem \ref{thm:Stability} holds in all dimensions. It is worth noting that, although $\delta(f;\Omega)^{1/2}$ still controls the distance of $f$ to the set of optimizers, there is a dimensional dependence of this estimate on $|\Omega|$.
\end{numremark}

As observed in \cite{NicolaTilli}, if the set $\Omega$ is fixed and has finite measure, Theorem \ref{thm:FKforSTFT} (and consequently
also Theorem  \ref{thm:Stability}) can be interpreted in terms of the 
well-known \emph{localization operator}  \cite{Berezin1971,Daubechies} defined,
in terms of the STFT operator $\STFT \colon L^2(\R)\to L^2(\R^2)$ with Gaussian window, by 
$$L_{\Omega } := \STFT^*\, 1_{\Omega}\, \STFT,\qquad L_\Omega\colon  L^2(\R)\to L^2(\R).$$
This is a positive trace-class operator, hence its norm coincides with its largest eigenvalue
\begin{equation}
    \label{eqn:FKforSTFTdefnPhi}
    \lambda_1(\Omega):=
    \max_{f\in L^2(\R)\backslash \{0\}}
\frac {\langle L_\Omega \, f,f\rangle}{\Vert f\Vert_{L^2(\R)}^2}    
    = \max_{f\in L^2(\R)\backslash \{0\}} \frac{\int_{\Omega} |\STFT f(x,\omega)|^2 \diff x \diff \omega}{\int_{\R^2} |\STFT f(x,\omega)|^2 \diff x \diff \omega}.
\end{equation}
In particular, due to the arbitrariness of $f$,  \eqref{eqn:FKInequality} entails
that 
\begin{equation}
\label{eq:FK2}    
\lambda_1(\Omega)\leq 1-e^{-|\Omega|},
\end{equation}
with equality if and only if $\Omega$ is a ball, and so we call \eqref{eq:FK2} a Faber--Krahn inequality, by analogy with the Dirichlet Laplacian.
Clearly, for any fixed $\Omega$, the functions $f_\Omega$ that achieve the
maximum in \eqref{eqn:FKforSTFTdefnPhi}
(i.e.\ the eigenfunctions of $L_{\Omega}$ associated with its first eigenvalue $\lambda_1(\Omega)$) are those functions whose STFT
is optimally concentrated in that particular set $\Omega$.
When $\Omega$ is a ball, these eigenfunctions are the functions
described in \eqref{eqn:FKEqualityFunctions} and appearing also in
\eqref{eqn:StabilityFunction}: therefore, 
specifying Theorem \ref{thm:Stability} to the case where $f = f_\Om$ is the first eigenfunction of $L_{\Omega}$, normalized so that $\|f_\Om\|_{L^2}=1$, we obtain the following 
stability result for the first eigenvalue and eigenfunction of localization
operators:
\begin{numcor}\label{cor:eigenvalue} Let $\Omega \subset \R^2$ be a measurable set of finite Lebesgue measure, and let $\lambda_1(\Omega)$ be the first eigenvalue of the localization operator
$L_\Omega$ as in \eqref{eqn:FKforSTFTdefnPhi}, with 
unit-norm  eigenfunction $f_\Omega$. Then \eqref{eq:FK2} holds true,
and
\begin{equation}
\label{eqn:stabfOmega}
    \min_{z_0\in \C, |c|=1} \|f_{\Om} - c\, \varphi_{z_0} \|_2 \leq C e^{|\Omega|/2} \left(1-\frac{\lambda_1(\Omega)}{1-e^{-|\Omega|}}\right)^{1/2},
\end{equation}
for some universal (explicitly computable) constant $C$. Moreover,
for some explicit constant $K(|\Omega|)$ we also have
    \begin{equation*}
        \mcalA(\Om) \leq K(|\Omega|)  \left(1-\frac{\lambda_1(\Omega)}{1-e^{-|\Omega|}}\right)^{1/2} .
    \end{equation*}
\end{numcor} 
This result is the analogue of
the  stability results for the Faber--Krahn inequality for the Dirichlet Laplacian \cite{AllenKriventsovNeumayer,BrascoDePhilippisVelichkov,FTV}.
Note, however, that the stability estimate \eqref{eqn:StabilityFunction}
is more general than \eqref{eqn:stabfOmega}, because
it holds for arbitrary functions $f \in L^2(\R)$ which are not assumed to be eigenfunctions of the localization operator $L_{\Om}$. 
Indeed, the
results of Theorem \ref{thm:Stability} are stronger than the available stability results for the Faber--Krahn inequality for the Dirichlet Laplacian 
also in that, contrarily to  \cite{AllenKriventsovNeumayer,BrascoDePhilippisVelichkov}, our proof of Theorem \ref{thm:Stability} is \textit{quantitative} and does not rely on compactness arguments, as in the penalization method  \cite{CicaleseLeonardi}. It is for this reason that the constants in estimates \eqref{eqn:StabilityFunction}--\eqref{eqn:StabilitySet} can
be made \textit{explicit}. Note, moreover, that the set $\Omega$ in Theorem \ref{thm:Stability} is not assumed to be smooth; in fact, 
since $\STFT f$ is essentially an entire function via
the Bargmann transform, we can replace $\Omega$ with a suitable super-level set of a holomorphic function, which in Section \ref{subsec:GeometryOfSuperLevelSets} we prove to be very well-behaved (we then use the rigidity of the problem to come back from super-level sets of holomorphic functions to the original set $\Omega$).

We saw in Remark \ref{rmk:sharpness} that \eqref{eqn:StabilityFunction} is sharp, but whether Corollary \ref{cor:eigenvalue} is sharp as well is a more delicate question. To answer it, one would need to either (i) compute the first eigenfunctions of $L_{\Omega}$ for domains $\Omega$ close to a ball,  or (ii) given a function $f$ close to the Gaussian $\varphi$, construct a domain $\Omega_f \subset \R^2$ such that $f$ is the first eigenfunction of $L_{\Omega}$. Strategy (i) appears rather difficult: to the best of our knowledge, the eigenfunctions of $L_{\Omega}$ are not known even in the simple case where $\Omega$ is an ellipse of small eccentricity; see \cite{Abreu2012,Daubechies}.  Implementing strategy (ii) involves tools essentially disjoint from those of this manuscript and so we decided not to address the question of optimality of Corollary \ref{cor:eigenvalue} here; instead, this is one of the main goals of an upcoming work by the third author \cite{RamosLocalization}.

To discuss the main ideas behind the proofs of our results, we now briefly recall some facts
and background notions from \cite{NicolaTilli}, which we shall
use throughout the paper. We point out, however, that the proof
of Theorem \ref{thm:FKforSTFT} in \cite{NicolaTilli} cannot
be readily adapted to yield quantitative results such as
\eqref{eqn:StabilityFunction} or \eqref{eqn:StabilitySet}. Instead, the proof of these inequalities requires a set of new geometric ideas 
and estimates in the Fock space, which are the core of the present paper
and which (often being of a general character, such as Lemma
\ref{lemma:super-level-new} or the results in Section \ref{subsec:GeometryOfSuperLevelSets}), are of interest on their own. 

\subsection{Proof strategy in the Bargmann--Fock space}
\label{subsec:12}

As shown in \cite{NicolaTilli}, 
energy concentration problems for the STFT can be
very cleanly formulated (and dealt with) in terms
of the Fock space
\cite{Zhu2012}, i.e.\ the Hilbert space $\Fock(\C)$ of all holomorphic functions $F \colon \C \to \C$ for which
\[ \Fnorm{F} \coloneqq \of{\int_{\C} \abs{F(z)}^2 e^{- \pi \abs{z}^2} \diff z}^{1/2} \ls \infty, \]
endowed with the natural scalar product
\[
\langle F,G\rangle_{\Fock}=
\int_{\C} F(z)\overline{G(z)}\, e^{- \pi \abs{z}^2} \diff z.
\]
Here and throughout, 
$z=x+i y$ and $\diff z=\diff x\diff y$ denotes Lebesgue measure on $\C$, 
always identified with $\R^2$.
This Hilbert space is closely connected to the STFT through the Bargmann transform $\Barg\colon L^2(\R)\to \Fock(\C)$, defined for $f \in L^2(\R)$ as
\begin{equation}\label{defB}
\Barg f(z) = \int_{\R} f(t) e^{2 \pi t z -\pi t^2 - \frac{\pi}{2}z^2} \diff t , \quad z \in \C ,
    \end{equation}
see e.g.\ \cite[Section~3.4]{Grochenig}.
The Bargmann transform is a unitary isomorphism which maps the orthonormal basis of Hermite functions on $\R$ onto the orthonormal basis of $\Fock(\C)$ given by the normalized monomials
\begin{equation}
    \label{monomials}
 e_k(z) = \of{\frac{\pi^k}{k!}}^{1/2} z^k , \quad k=0,1,2,\ldots .
 \end{equation}
More importantly for us, the definition of $\Barg$ encodes the 
crucial property that
\begin{equation*}
    \STFT f (x,-\om) = e^{\pi i x \om} \Barg f(z) e^{- \pi \abs{z}^2/2} , \quad z = x+i\om,
\end{equation*}
which allows us to express the energy concentration in the time-frequency plane in terms of functions in the Fock space, since
\begin{equation}
    \label{concB}
 \frac{\int_{\Om} \abs{\STFT f(x,\om)}^2 \diff x \diff \om}{\norm{f}_{L^2}^2} = \frac{\int_{\Om '} \abs{\Barg f(z)}^2 e^{- \pi \abs{z}^2} \diff z}{\Fnorm{\Barg f}^2} ,
 \end{equation}
where $\Om ' = \inset{(x, \om): (x,-\om) \in \Om}$. 
In this new setting, the image via $\Barg$ of the functions $\varphi_{z_0}$ defined in \eqref{eqn:FKEqualityFunctions} 
takes the form
\begin{equation}\label{eqn:DefnFz0}
    \Barg \varphi_{z_0} = F_{z_0}, \qquad F_{z_0}(z)= e^{-\frac{\pi}{2} \abs{z_0}^2} e^{\pi z \conj{z_0}},
\end{equation}
and therefore 
Theorem~\ref{thm:FKforSTFT} can be rephrased in terms of the Fock space as follows, cf.\ \cite[Theorem~3.1]{NicolaTilli}:
\begin{theorem}\label{thm:FKforSTFTFock}
  If $\Omega\subset \R^2$ is a measurable set with positive and finite Lebesgue measure, 
    and if
    $F \in \Fock(\C) \setminus \inset{0}$
is an arbitrary function,
    then
    \begin{equation}\label{eqn:FKInequalityFock}
        \frac{\int_{\Om} \abs{F(z)}^2 e^{-\pi \abs{z}^2} \diff z}{\Fnorm{F}^2} \leq 1-e^{- \vol{\Om}} .
    \end{equation}
    Moreover, 
    equality is attained  if and only if $\Om$ coincides (up to a set of measure zero) with a ball
    centered at some $z_0\in\C $  and, at the same time, 
$F = c F_{z_0}$
    for some $c\in \C\setminus\{0\}$.
\end{theorem}
Similarly, we 
can rephrase Theorem \ref{thm:Stability} 
over the Bargmann--Fock space,
as follows:
\begin{numthm}[Fock space version of Theorem \ref{thm:Stability}]\label{thm:StabilityFockSpace}
    There is an explicitly computable constant $C>0$ such that, for all  measurable sets $\Om \subset \R^2$ with finite measure and all functions $F\in \Fock(\C)\backslash\{0\}$, we have 
    \begin{equation}\label{eqn:StabilityFunctionFock}
        \min_{\substack{|c|=\|F\|_{\Fock},\\ z_0\in \C}} \frac{\Fnorm{F-cF_{z_0}}}{\Fnorm{F}} \leq C
        \left(e^{\vol{\Om}} \delta(F;\Om)\right)^{1/2},
    \end{equation}
where
\begin{equation}
\label{defdeltaFO}
\delta(F;\Omega) \coloneqq 1-\frac{\int_\Omega |F(z)|^2 e^{-\pi |z|^2}\diff z}
{(1-e^{-|\Omega|})\Vert F\Vert_{\Fock}^2}.
\end{equation}
Moreover, for some universal explicit constant $K(|\Omega|)$ we also have
    \begin{equation}\label{eqn:StabilitySet2}
        \mcalA(\Om) \leq K(|\Omega|) \delta(F;\Om)^{1/2} .
    \end{equation}
\end{numthm}
We will provide two different proofs of this theorem, 
based on a careful study of the real analytic function
\begin{equation}\label{eqn:Defnu}
    u_F(z) = u(z) \coloneqq \abs{F(z)}^2 e^{- \pi \abs{z}^2}
\end{equation}
and the properties of its super-level sets
\begin{equation}\label{eqn:DefnSuperLevelSetsOfu}
    A_t \coloneqq \inset{u \gs t} = \inset{z \in \C \colon u(z) \gs t},
\end{equation}
where $F$ is an arbitrary function in  $\Fock(\C)\setminus\{0\}$.
This study was initiated in \cite{NicolaTilli}, where it was proved
that the distribution function
\begin{equation}
    \label{def:mut}
    \mu_F(t)=\mu(t) \coloneqq |A_t|,\quad t\geq 0
\end{equation}
is locally absolutely continuous on $(0,\infty)$ and satisfies
\begin{equation}
    \label{eqn:diffineq} \mu'(t)\leq - \,\frac 1 t\quad
    \text{for a.e. $t\in (0,T),\qquad T \coloneqq \max_{z\in\C} u(z)$,}
\end{equation}
from which one readily obtains that
\begin{equation}
    \label{estmut1} \mu(t) \geq \log_+ \frac T t\quad\text{for all }
    t>0,\qquad\text{where  $\log_+ x \coloneqq \max\{0,\log x\}$}.
\end{equation}
Notice that, when $F=c F_{z_0}$ as in the last part
of Theorem \ref{thm:FKforSTFTFock}, then 
$T=|c|$ and $\mu(t)=\log_+ T/t$. In \cite{NicolaTilli},
\eqref{eqn:diffineq} can be found in the equivalent form
\begin{equation}\label{eqn:DiffInequ}
    u^*(s) + (u^*)'(s) \geq 0, \text{ for almost every } s \ge 0,
\end{equation}
where
$u^*\colon \R^+\to (0,T]$ is the decreasing rearrangement of $u$, usually
defined as
\begin{equation}\label{eqn:defu*}
    u^*(s) \coloneqq \sup \inset{t \geq 0 \colon \mu(t) \gs s},\qquad
    s\geq 0.
\end{equation}
The function $u^*$ is proved to be invertible, with $\mu|_{(0,T]}$ as inverse function (see \cite{Kulikov} for a direct
usage of \eqref{eqn:diffineq} in this form).
This fact
enables one to find, for any
number $s\geq 0$,
a unique super-level set $A_{t}=A_{u^*(s)}$ of measure $s$, which is the set
where $u$ is most concentrated among all sets of measure $s$, namely
\begin{equation}\label{eqn:DefnI}
    I(s) \coloneqq \int_{\inset{u \gs u^*(s)}} u(z) \diff z
    \geq \int_\Omega u(z) \diff z,\quad\text{whenever $|\Omega|=s$.}
\end{equation}
Based on \eqref{eqn:DiffInequ}, it was proved in \cite{NicolaTilli}
that the function
 $G(\sigma) \coloneqq  I(- \log \sigma)$ is
\emph{convex} on $[0,1]$. Since
$$G(0) = \lim_{s \to \infty} I(s) = \int_{\C} u(z) \diff z = \|F\|_{\mathcal{F}^2}^2,\qquad G(1) = I(0) = 0,$$ 
the convexity of $G$ yields the upper bound 
$G(\sigma) \le \|F\|_{\mathcal{F}^2}^2(1-\sigma)$ or, equivalently,
\begin{equation}
    \label{eq:boundI}
    I(s) \le \|F\|_{\mathcal{F}^2}^2(1-e^{-s}),
\end{equation}
which, combined with \eqref{eqn:DefnI},
proves
\eqref{eqn:FKInequalityFock}. 

It was then observed in
\cite{NicolaTilli} that, if equality holds in \eqref{eqn:FKInequalityFock}, then by convexity 
we must have $G(\sigma)\equiv \|F\|_{\Fock}^2 (1-\sigma)$ on
$[0,1]$ or, equivalently, 
$I(s)=\|F\|_{\Fock}^2 ( 1- e^{-s})$ for every $s\geq 0$, and
in particular
\begin{equation}
\label{eq:Iprimezero}    
I'(0)=\Vert F\Vert_{\Fock}^2.
\end{equation}
But since $\mathcal{F}^2(\C)$ is a Hilbert space with reproducing kernel $K_{w}(z) = e^{\frac{\pi}{2} \abs{w}^2} F_{w}(z)$, we have
\begin{equation}
    \label{eq:repkernel}
    |F(z)|^2 e^{-\pi|z|^2} \le \|F\|_{\mathcal{F}^2}^2
\end{equation}
for all $F\in \Fock(\C)$,
with equality at some $z=z_0$ if and only if $F = c F_{z_0}$ for some $c \in \C$ (see e.g.\ \cite[Proposition 2.1]{NicolaTilli}). Since in
any case $I'(0)=T:=\max_{z \in \C} |F(z)|^2 e^{-\pi|z|^2}$,
\eqref{eq:Iprimezero} shows that equality in \eqref{eqn:FKInequalityFock} forces equality (for at least one $z$) also in \eqref{eq:repkernel},
and this proves the last part of Theorem \ref{thm:FKforSTFTFock}. 

Without loss of generality, we always assume the normalization condition ${\|F\|_{\Fock}=1}$.
A simple but fundamental observation to both our proofs of Theorem \ref{thm:StabilityFockSpace} is that equality in \eqref{eq:repkernel} can be precisely quantified: indeed,  \begin{equation}
    \label{eq:kernintro}
    \min_{\substack{z_0\in\C \\ |c|=1}} \Vert F-cF_{z_0}\Vert_{\Fock}^2 = 2(1-\sqrt{T})\leq 2(1-T),
\end{equation}
cf.\ Lemma \ref{lemma:kern} below. Thus, to prove estimate \eqref{eqn:StabilityFunctionFock} in Theorem \ref{thm:StabilityFockSpace}, we need to show that the deficit controls $(1-T)$.

Our first proof of Theorem \ref{thm:StabilityFockSpace} is based on a careful study of the area between the graphs of $s\mapsto u^*(s)$ and $s\mapsto e^{-s}$. Consider a parameter $s^*>0$, defined to be a solution of the equation
$$u^*(s^*)=e^{-s^*}.$$
Such a solution always exists and, as soon as $T<1$, it is unique. An argument relying on the convexity inequality \eqref{eqn:DiffInequ} yields
\begin{equation}
    \label{eq:quantmaxintro}
    \int_0^{s^*} \left(e^{-s} - u^*(s)\right) \diff s \leq e^{|\Om|} \delta, \qquad \delta:=\delta(F;\{u>u^*(|\Om|)\}),
\end{equation}
cf.\ Lemma \ref{lemma:QuantitativeMaxuNew}.
Thus, to prove the desired stability estimate \eqref{eqn:StabilityFunctionFock}, by \eqref{eq:kernintro} and \eqref{eq:quantmaxintro} it is enough to show that the integral above controls $(1-T)$. In fact, it is not difficult to see that this integral controls $(1-T)$ to a \textit{suboptimal} power, as we have
\begin{equation}
    \label{eq:suboptimal}
    \frac{(1-T)^2}{2} =  \int_0^{s_*} (1-s-T)_+ \diff s\leq  \int_0^{s_*} \left(e^{-s} - u^*(s)\right) \diff s \leq  e^{|\Om|} \delta.
\end{equation}
Thus, by \eqref{eq:kernintro}, \eqref{eq:suboptimal} already yields a suboptimal form of stability.

To upgrade \eqref{eq:suboptimal} to an \textit{optimal} estimate, we need to estimate the integral in \eqref{eq:quantmaxintro} much more precisely, and our approach is to give a precise quantification of the equality cases in  \eqref{estmut1}.
By passing to the inverse functions we have
\begin{equation}
    \label{eq:betterlowerbound}
    \int_{e^{-s^*}}^T \left(\log \frac 1 t - \mu(t)\right) \diff t \leq \int_0^{s_*} \left(e^{-s} - u^*(s)\right) \diff s,
\end{equation}
cf.\ \eqref{eq1002} below, and our proof proceeds by establishing a sharp estimate for the distribution function $\mu(t)$: precisely, there is a universal constant $C>0$ such that
\begin{equation}
    \label{eq:sharplevelsets}
    \mu(t) \leq (1+C(1-T)) \log \frac T t,
\end{equation}
provided that $t$ and $T$ are sufficiently close to 1 (see Lemma \ref{lemma:super-level-new}). Note that, by the suboptimal estimate \eqref{eq:suboptimal}, this restriction on $t,T$ does not restrict generality. Establishing \eqref{eq:sharplevelsets} is the most delicate part of the whole argument, as this estimate relies on a cancellation effect due to analyticity of $F$. The desired estimate \eqref{eqn:StabilityFunctionFock} then follows by an elementary analysis, after plugging in \eqref{eq:sharplevelsets} into \eqref{eq:betterlowerbound} and using again \eqref{eq:kernintro} and \eqref{eq:quantmaxintro}.

Concerning the stability of the set in \eqref{eqn:StabilitySet2}, we note that it is not clear how to quantify inequality \eqref{eqn:DefnI} used in the proof of Theorem \ref{thm:FKforSTFTFock} described above, at least for general sets $\Om$. Nonetheless, since we already have estimate \eqref{eqn:StabilityFunctionFock}, we know that $u$ is close to a Gaussian. This allows us to first compare $\Om$ with $A_{u^*(|\Om|)}$, and then compare $A_{u^*(|\Om|)}$ with a ball.

The described strategy also works to show the stability of a similar Faber-Krahn inequality for \emph{wavelet transforms} (see \cite{RamosTilli}), after adapting the current arguments. We plan to address this in a future work. 



\subsection{The geometry of super-level sets and a variational approach}\label{sec:variationalintro}

As mentioned above, we will give two different proofs of Theorem \ref{thm:StabilityFockSpace}, the first one having been described in the previous subsection. We now describe our second proof, which is variational in nature and based on the following result, which is of independent interest:

\begin{numprop}\label{prop:levelsets}
 There are small explicit constants $\delta_0,c>0$ such that the following holds: for all $F\in \Fock(\C)$ such that 
 $$e^{s} \delta(F;A_{u^*(s)}) \leq \delta_0$$
 and for all $s<c\log(1/\delta_0)$, the super-level set
 $$A_{u^*(s)}=\{z\in \C:u(z)>u^*(s)\}$$
 has smooth boundary and convex closure.
\end{numprop}

Proposition \ref{prop:levelsets} shows in particular that level sets of $u$ sufficiently close to its maximum can be seen as smooth graphs over a circle, thus they can be deformed to a circle through an appropriate flow. This observation, in turn, allows us to give a variational approach to Theorem \ref{thm:StabilityFockSpace}, in the spirit of Fuglede's computation \cite{Fuglede1989} for the quantitative isoperimetric inequality. We refer the reader to \cite{Henrot2005} for a detailed introduction to variational methods in shape-optimization problems. 

To be precise, and comparing with \eqref{eq:boundI}, for some fixed $s>0$ we consider the functional 
\begin{equation*}\label{eq:def-variation-and-u} 
\K \colon F \mapsto \frac{I_F(s)}{\|F\|_{\Fock}^2}.
\end{equation*}
We study perturbations of $F_0\equiv 1$, i.e.\ we consider $F=1+\e G$ for some small $\e>0$. Taking $\Om=A_{u^*(s)}$, we note that by a formal Taylor expansion we have
$$ \delta(1+\e G;\Om) 
= \frac{1}{1-e^{-|\Om|}}\big(\K[1]-\K[1+\e G]\big) 
= \frac{1}{1-e^{-|\Om|}} \Big(\frac{\e^2}{2} \nabla^2 \K[1](G,G) + o(\e^2)\Big),$$
since $\nabla \K[1](G)=0$ for all $G\in \Fock(\C)$ satisfying the orthogonality conditions 
$$\langle 1,G\rangle_{\Fock} = \langle z, G\rangle_{\Fock} = 0,$$ according to Theorem \ref{thm:FKforSTFTFock}. Thus, once the Taylor expansion above has been justified (and this is achieved in Appendix \ref{sec:appendix}), we see that for small perturbations of $F_0\equiv 1$ the deficit is governed by the second variation of $\K$. For stability to hold, this variation ought to be uniformly negative definite, since $\e$ is essentially the right-hand side in \eqref{eqn:StabilityFunctionFock}. In Proposition \ref{prop:negdef} we show that, under the above orthogonality conditions, we have
\begin{equation}
    \label{eq:secondvarintro}
     \nabla^2 \K[1](G,G) \leq - s e^{-s}\|G\|^2_{\Fock(\C)}.
\end{equation}
This inequality is interesting for several reasons. Firstly, it is sharp, as highlighted by taking $G$ to be any polynomial of degree 1. Secondly, by the suboptimal stability result \eqref{eq:suboptimal}, to prove \eqref{eqn:StabilityFunctionFock} it is enough to consider functions with small deficit. Therefore, the above Taylor expansion, combined with \eqref{eq:secondvarintro}, easily yields the stability estimate \eqref{eqn:StabilityFunctionFock}, although with a \textit{suboptimal} dependence of the constant on $|\Omega|$.
Finally, the non-degeneracy of $\nabla^2 \K$ provided by \eqref{eq:secondvarintro}, combined once again with the above Taylor expansion, shows that the deficit behaves quadratically near $F_0\equiv 1$, which leads to a direct proof of the optimality of our estimates, as claimed in Remark \ref{rmk:sharpness}.

\subsection*{Outline}
In Section \ref{sec:firstproof} we give a first proof of \eqref{eqn:StabilityFunctionFock}, following the strategy described in Section \ref{subsec:12} above. In Section \ref{subsec:GeometryOfSuperLevelSets} we study the geometry of the super-level sets of functions with small deficit and, in particular, we prove Proposition \ref{prop:levelsets} above. In Section \ref{sec:set-stability} we prove the set stability estimate \eqref{eqn:StabilitySet}. Section \ref{sec:alternative-proof} contains the variational proof described in Section \ref{sec:variationalintro} and in particular the proof of \eqref{eq:secondvarintro}. In Section \ref{sec:sharpness} we prove the claims from Remark \ref{rmk:sharpness}. Finally, in Section \ref{sec:Generalize} we extend our results to the higher-dimensional setting, as claimed in Remark \ref{rmk:higherd}.
 
\subsection*{Acknowledgements} A.G. was supported by Dr.\ Max R\"ossler, the Walter Haefner Foundation and the ETH Z\"urich Foundation. J.P.G.R. acknowledges financial support the European Research Council under the Grant Agreement No. 721675 “Regularity and Stability in Partial Differential Equations (RSPDE)''.




\section{First proof of the function stability part}\label{sec:firstproof}

\newcommand\donotshow[1]{}

The goal of this section is to prove \eqref{eqn:StabilityFunctionFock}, by combining a series
of new results (potentially of independent interest)
valid for arbitrary functions $F\in\Fock$, which for convenience
will be assumed to be normalized by
\begin{equation}
    \label{eq:normalized}
    \Vert F\Vert_{\Fock}=1.
\end{equation}
In these statements, we will make extensive use of the notation
and the background results recalled in Subsection \ref{subsec:12}, 
concerning the functions $u(z)$, $\mu(t)$ and $u^*(s)$ 
that can be associated with a given $F\in\Fock$. In particular,
as in \eqref{eqn:diffineq}, in our statements we will let
\begin{equation}
    \label{def:T}
    T \coloneqq \max_{z\in\C} u(z)=u^*(0)=\max_{z\in\C} |F(z)|^2
    e^{-\pi |z|^2},
\end{equation}
recalling that $T\in [0,1]$ whenever \eqref{eq:normalized}
is assumed. 

We also note that, since $u^*$ is (by its definition) 
equimeasurable with $u$ and decreasing, there holds
\begin{equation}
    \label{equimeas}
    \int_{\{u>u^*(s_0)\}} u(z)\diff z =\int_0^{s_0} u^*(s)\diff s
    \quad\forall s_0\geq 0.
\end{equation}
Moreover, as recalled in Subsection \ref{subsec:12}, when
$F=c F_{z_0}$ (with $|c|=1$) is one of the optimal functions described in
Theorem \ref{thm:FKforSTFTFock}, one has $\mu(t)=\log_+ \frac 1 t$
or, equivalently, $u^*(s)=e^{-s}$. For this reason, 
a careful comparison between $e^{-s}$ and $u^*(s)$
(for an arbitrary $F$ satisfying \eqref{eq:normalized})
will be the core of the results of this section. 
Since, when \eqref{eq:normalized} holds, letting
$s_0\to \infty$ in \eqref{equimeas} we have
\begin{equation}
    \label{massone}
    1=\int_0^\infty u^*(s)\diff s =\int_0^\infty e^{-s}\diff s,
\end{equation}
as noted in \cite{KNOT} there exists at least
one value $s^*>0$ for which
\begin{equation}
    \label{defs*}
    u^*(s)\quad \begin{cases}
        \quad\leq e^{-s} & \text{if $s \in [0,s^*]$}\\
        \quad\geq e^{-s} & \text{if $s\geq s^*$}
    \end{cases}
\end{equation}
or, equivalently, in terms of the inverse functions, a
value $t^* \in (0,T)$ for which
\begin{equation}
    \label{deft*}
    \mu(t) \quad\begin{cases}
        \quad\geq \log\frac 1 t & \text{if $t \in (0,t^*]$}\\
        \quad\leq \log\frac 1 t & \text{if $t\in [t^*,T]$}
    \end{cases}
\end{equation}
(note $\mu(t)=0$ for $t\geq T$). When $T=1$ (or, equivalently,
if $F$ is one of the optimal functions described in Theorem 
\ref{thm:FKforSTFTFock}, see \cite[Proposition 2.1]{NicolaTilli})
and hence $u^*(s)=e^{-s}$, clearly all values of $s^*$
(or $t^*$) have this property, but when $T<1$
we will prove in Corollary \ref{cor:nondegen} that $s^*$
and $t^*$ are in fact \emph{unique}, with an unexpected
universal upper
bound on $t^*$ (or lower bound on $s^*$).

With this background, we are now ready to state and prove
the results of this section, starting with
a sharp estimate for $\mu(t)$, which shows that
\eqref{estmut1} becomes almost an equality when $T$ is close to $1$.

\donotshow{

Let us begin with a simple remark in the flavor of the introduction to \cite{Maggi} about working with the small deficit assumption. We may always be able to assume that $F$ and $\Om$ have a small deficit $\delta(F;\Om)$, as small as required. Indeed, if we look for a result of the type
\[ \Fnorm{F - c F_{z_0}} \leq K \Fnorm{F} \delta(F;\Om)^{\alpha} \]
for some $K, \alpha \gs 0$, and we have already found a constant $C$ for which it holds assuming that $\delta(F;\Om)$ is smaller than some $\delta_0$, then since $\Fnorm{F-cF_{z_0}} \leq 2\Fnorm{F}$, we can take the constant $K$ to be
\[ K = \max \inset{C, 2\delta_0^{- \alpha}} .\]
The same can be said about \eqref{eqn:StabilitySet}, since $\mcalA(\Om) \leq 2$ as well. For this reason, we will often look for an ``asymptotic'' result, that is, we will assume that $\delta(F;\Om)$ is smaller than some given $\delta_0$. By homogeneity, we can and will assume that
\[ \|F\|_{\Fock}=1. \]

Furthermore, as Theorem \ref{thm:StabilityFockSpace} directly implies estimate \eqref{eqn:StabilityFunction} in Theorem~\ref{thm:Stability}, by use of the Bargmann transform, we will restrict ourselves to proving the former.

Here and henceforth, we will use the notation $u_F$ as in \eqref{eq:def-variation-and-u} and the notation $u$ as in \eqref{eqn:Defnu} interchangeably. By our normalization assumption, we have $u^*(0)=1$.

Let us fix $s_0 \gs 0$ as the measure of $\Om$, thus $s_0=|\Om|$. In what follows, we will be interested in a particular case of the problem, that is, when we take $\Om$ to be a super-level set of $u$. For this reason, given $s \gs 0$, we define
\begin{equation}\label{defn:LevelSetDeficit}
    \delta_{s}(F) \coloneqq \delta(F ; \inset{u \gs u^*(s)}) = \frac{1-e^{-s} - I(s)}{1-e^{-s}}.
\end{equation}
Thus $\delta_s(F)$ is the deficit measuring the failure of $F$ in achieving the equality in \eqref{eqn:FKInequalityFock} over its super-level set of measure $s$.

\begin{figure}[H]\label{fig:LemmaMaxuSetting}
    \centering
    \def\svgwidth{0.6\textwidth}
    \import{files/figures/}{LemmaMaxuSetting.pdf_tex}
\end{figure}

Notice that 
\[ \int_0^{\infty} e^{-s} \diff s = 1 = \lim_{s \to \infty} I(s) = \int_{0}^{\infty} I'(s) \diff s = \int_0^{\infty} u^*(s) \diff s .\]
Since $1-\max u = u^*(0) \leq 1$, by the continuity of $u^*$ there exists $s^* \gs 0$ for which $u^*(s^*) = e^{-s^*}$. We can take advantage of this in proving a quantitative version of the inequality $\max u \leq 1$, as shown in the following lemma.

}

\begin{numlemma}\label{lemma:super-level-new} 
For every $t_0\in (0,1)$, there exists a threshold $T_0\in (t_0,1)$ and a constant
$C_0>0$ with the following property.
If $F\in \Fock$ is such that $\Vert F\Vert_{\Fock}=1$ and
$T \geq T_0$,
then
\begin{equation}
    \label{newestmu}
    \mu(t)\leq \left( 1+{C_0 (1-T)}\right)\log \frac T t\quad\forall t\in [t_0,T].
\end{equation}
\end{numlemma}
We note, before proving such a result, that the proof presented below shows that one can choose $C_0=C/t_0^3$, where $C$ is some universal constant.

\newcommand\fock{{\mathcal F}^2}
\newcommand{\expz}{e^{-\pi  |z|^2}}
\newcommand{\expzz}{e^{-\pi \frac {|z|^2}2}}
\newcommand\rz{r_0}
\newcommand\ru{r_1}
\newcommand\rs{r_{\!\sigma}}
\newcommand\desqr{\delta^2}

\begin{proof}

Given  $t_0\in (0,1)$ and $F$ as in the statement,
we split the proof into several steps.

\noindent\textsc{Step I. } We may assume that $u(z)$ achieves its absolute maximum $T$ at $z=0$ and that $F(0)$ is a real number, so that
$F(0)=\sqrt T$. Expanding $F$ with respect to the orthonormal basis of monomials \eqref{monomials},
 we have
\begin{equation}
  \label{series1}
\frac {F(z)}{\sqrt{T}}=1 +R(z),\quad z\in\C,
\end{equation}
where $R(z)$ is the entire function
\begin{equation}
  \label{defR}
R(z) \coloneqq \sum_{n=2}^\infty \frac{a_n}{\sqrt{T}}\, \frac {\pi^{n/2} z^n}{\sqrt{n!}},\quad z\in\C
\end{equation}
The fact that $a_1=0$, i.e. $F'(0)=0$, follows easily from our assumption that $u(z)$ has
 a critical point at $z=0$, which by \eqref{eqn:Defnu} forces a critical point for
$|F(z)|^2$ and ultimately for $F(z)$.
The assumption that $1=\Vert F\Vert_{\fock}^2$ takes the form $1=T+\sum_{n=2}^\infty |a_n|^2$,
which we record in the form
\begin{equation}
\label{estcoda}
\sum_{n=2}^\infty \frac{|a_n|^2}T=\frac {1-T}T =:\desqr,
\end{equation}
hereby defining $\delta$. In the sequel we will often tacitly assume that $\delta$ is small enough, depending only on $t_0$; in
the end, the required smallness of $\delta$ will determine the threshold $T_0$ in the statement of Lemma \ref{lemma:super-level-new}.

From \eqref{defR}, Cauchy-Schwarz and \eqref{estcoda} we obtain
\begin{equation}
\label{eq99}
|R(z)|^2\leq \left(\sum_{n=2}^\infty \frac{|a_n|^2}T\right)
\left(\sum_{n=2}^\infty \frac {\pi^n |z|^{2n}}{n!}\right)
= \desqr \left(e^{\pi |z|^2}-1-\pi |z|^2\right).
\end{equation}
In particular, $|R(z)|^2\leq \desqr \left(e^{\pi |z|^2}-1\right)$, hence squaring \eqref{series1}
we have
\begin{equation}
  \label{estmodsq}
  \frac {|F(z)|^2}T \leq 1 + \desqr \left(e^{\pi |z|^2}-1\right) + h(z),
\end{equation}
where $h(z)$  is the
real valued harmonic function
\begin{equation}
  \label{defh}
h(z):=2\mathop{\textrm{Re}} R(z),\qquad z\in\C.
\end{equation}

\noindent\textsc{Step II: } Estimates for $h$.
Since $|h(z)|\leq 2 |R(z)|$, the elementary inequality $e^x-1-x\leq \frac {x^2}2 e^x$, written with $x=\pi|z|^2$ and combined
with \eqref{eq99}, implies
\begin{equation}
  \label{esth2}
  |h(z)|\leq \sqrt{2}\pi \,\delta |z|^2 e^{\frac{\pi |z|^2}2},\qquad\forall z\in\C.
\end{equation}
Differentiating \eqref{defR}, 
and then using Cauchy-Schwarz and \eqref{estcoda} as in \eqref{eq99}, we have
\begin{equation}
\label{Rprime}
|R'(z)|\leq
\sum_{n=2}^\infty  \frac {|a_n|}{\sqrt T}
\frac {n \pi^{n/2} |z|^{n-1}}{\sqrt {n!}}
\leq  \delta\left( \sum_{n=2}^\infty
\frac {n^2 \pi^{n} |z|^{2(n-1)}}{n!}
\right)^{\frac 1 2}\leq
\delta \sqrt{2}  \pi |z| e^{\frac {\pi |z|^2}2} ,
\end{equation}
having used the inequality $\frac {n^2}{n!}\leq \frac 2 {(n-2)!}$ in the last passage.
Similarly, differentiating \eqref{defR} twice, using Cauchy-Schwarz and estimating the resulting power series, we find
\begin{equation}
\label{Rsecond}
|R''(z)|\leq 
\delta\left( \sum_{n=2}^\infty
\frac {n^2(n-1)^2 \pi^{n} |z|^{2(n-2)}}{n!}
\right)^{\frac 1 2}
\leq
\delta C (1+|z|^2) e^{\frac {\pi |z|^2}2}.
\end{equation}
By \eqref{defh} and the Cauchy-Riemann equations $|\nabla h(z)|=2 |R'(z)|$
and $|D^2 h(z)|=2\sqrt 2 |R''(z)|$, we obtain the following  
uniform estimates with respect to the angular variable $\theta$ for
the first and second
radial derivatives of $h(r e^{i\theta})$:
\begin{equation}
\label{esthr}
\begin{aligned}
\left\vert \frac {\partial h(r e^{i\theta})}{\partial r} \right\vert\leq
\delta 2\sqrt{2}  \pi r e^{\frac {\pi r^2}2},
\end{aligned}
\end{equation}
and
\begin{equation}
\label{esthrr}
\left\vert \frac {\partial^2 h(r e^{i\theta})}{\partial r^2} \right\vert
\leq 
\delta C (1+r^2) e^{\frac {\pi r^2}2}.
\end{equation}

\noindent\textsc{Step III: } Definition of  $E_\sigma$ and $\rs(\theta)$.
Assuming $T>t_0$, we consider any $t\in [t_0,T)$ and
any complex number $z=r e^{i\theta}$ ($r\geq 0$),
and we observe that
\[
u(r e^{i\theta})>t\quad\iff\quad
\frac  {t e^{\pi r^2}}T < \frac {|F(r e^{i\theta})|^2}T.
\]
Hence, by virtue of \eqref{estmodsq}, we obtain the implication
\begin{equation}
\label{implication}
u(r e^{i\theta})>t\quad\Longrightarrow\quad
g_\theta(r,1)
<1,
\end{equation}
where, for every fixed $\theta\in [0,2\pi]$, $g_\theta$ is defined as
\begin{equation}
\label{defg}
  g_\theta(r,\sigma):=  e^{\pi r^2}\left(\frac t T -\desqr\right)+\desqr  - \sigma \,h(r e^{i\theta}),\quad
  r\geq 0,\quad \sigma\in [0,1].
\end{equation}
The variable $\sigma\in [0,1]$ plays the role of a parameter that defines the
family of planar sets
\[
E_\sigma:=\left\{ r e^{i\theta}\in\C\,|\,\, g_\theta(r,\sigma)<1\right\},\quad
\sigma\in [0,1].
\]
Since \eqref{implication} is equivalent to the set inclusion $\{u>t\}\subseteq E_1$, \eqref{newestmu} will be proved if we
show that
\begin{equation}
  \label{eqtesi2}
  |E_1|\leq \left(1+\frac {C\desqr}{t_0^3}\right)\log\frac T t.
\end{equation}
The advantage of the parameter $\sigma$ is that we can easily prove the analogous estimate for $E_0$ -- which is
a circle, since $g_\theta(r,0)$ is independent of $\theta$ --, and then show that this estimate is inherited
by every $E_\sigma$ (including $E_1$),
by exploiting a cancellation effect due to the harmonicity of $h$.

We first show that each set $E_\sigma$ is \emph{star-shaped} with respect to the origin, by showing that
$g_\theta(r,\sigma)$ is  increasing in $r$ (for fixed $\theta$ and $\sigma$).
Using \eqref{esthr} and assuming e.g. that $\desqr+\sqrt 2\delta \leq t_0/2$, we have from \eqref{defg}
\begin{equation}
\label{estgr}
\begin{aligned}
  \frac {\partial g_\theta(r,\sigma)}{\partial r} & =
 2\pi r e^{\pi r^2}\left(\frac t T -\desqr\right)-\sigma \frac {\partial h(r e^{i\theta})}{\partial r} \\
 &\geq 2\pi r e^{\pi r^2}\left(t_0  -\desqr\right) -  \delta 2\sqrt{2}  \pi r e^{\frac {\pi r^2}2} \\
 &\geq 2\pi r e^{\pi r^2}\left(t_0  -\desqr-\sqrt 2 \delta\right) \geq \pi t_0 r e^{\pi r^2}>0.
  \end{aligned}
\end{equation}
Since $g_\theta(0,\sigma)=t/T\geq t_0$, integrating the previous bound  we also obtain that
\begin{equation}
\label{growthg}
g_\theta(r,\sigma)\geq \frac {t_0}2\bigl(1+ e^{\pi r^2}\bigr),\quad \forall r\geq 0,\quad \forall \sigma\in [0,1],
\end{equation}
and hence, since $g_\theta(0)=t/T< 1$, for every $\sigma\in [0,1]$ the equation in $r>0$
\begin{equation}
\label{defrs}
g_\theta(r,\sigma)=1
\end{equation}
has a unique solution $\rs> 0$, which we shall also denote by $\rs(\theta)$
 when the dependence of $\rs$ on the angle $\theta$ is to be stressed, as in \eqref{deffs} below.
 Since $E_\sigma$ is star-shaped, using polar coordinates
we can compute its area $|E_\sigma|$ in terms of $\rs$,
  as
\begin{equation}
\label{deffs}
f(\sigma):= |E_\sigma|=\frac 1 2 \int_0^{2\pi}\rs(\theta)^2\diff \theta,\quad \sigma\in [0,1].
\end{equation}
Notice that  $f(1)$ is the area of $E_1$ that we want to estimate
as in \eqref{eqtesi2},  while
 \begin{equation}
   \label{fzez}
 f(0)=\frac 1 2\int_0^{2\pi} \rz(\theta)^2\diff \theta=\pi \rz^2,
\end{equation}
since when $\sigma=0$, equation  \eqref{defrs} simplifies to
\begin{equation}
  \label{defrz}
  e^{\pi \rz^2}\left(\frac t T -\desqr\right)+\desqr
  =1,
\end{equation}
so $\rz$ is independent of $\theta$ and $E_0$ is a ball of radius $\rz$. Note that the sets $E_\sigma$ are
uniformly bounded, since \eqref{growthg} and \eqref{defrs} entail that
\begin{equation}
  \label{boundrs}
   \pi\rs^2 \leq \log\frac 2 {t_0},\quad\forall \sigma\in [0,1].
\end{equation}

\noindent\textsc{Step IV: } Estimates for $\rs'$ and $\rs''$.
By \eqref{estgr} and the implicit function theorem,
 $\rs$ is, for every fixed value of $\theta\in  [0,2\pi]$, a smooth, bounded function of the
parameter $\sigma\in [0,1]$.
Denoting for simplicity by $\rs'$ its derivative with respect to $\sigma$,
we have
\begin{equation}
  \label{rprime1}
\rs'=\frac{\partial \rs(\theta)}{\partial \sigma}=-\,\,
\frac {\frac {\partial g_\theta}{\partial\sigma}(\rs,\sigma)}
{\frac{\partial g_\theta}{\partial r}(\rs,\sigma)}=
\frac {h(\rs e^{i\theta})}{\frac {\partial g_\theta}{\partial r}(\rs,\sigma)},\quad
\sigma\in [0,1],
\end{equation}
and using \eqref{esth2}
and \eqref{estgr} we find the bound
\begin{equation}
\label{boundrsp}
|\rs'|
\leq
\frac {\sqrt 2 \,\delta\rs}{t_0}e^{-\frac {\pi \rs^2}2 }.
\end{equation}
In particular, this implies that
\begin{equation}
  \label{cfrrsrz}
  \rs^2 \leq 2 \rz^2\quad\forall \sigma\in [0,1],
\end{equation}
since by \eqref{boundrsp} $|\rs'|/\rs\leq  \sqrt 2\,\delta/t_0 \leq \log\sqrt{2}$ provided $\delta$
is small enough, we have for every $\sigma\in [0,1]$
\[
\log \rs =\log \rz + \int_0^\sigma \frac {r_s'}{r_s}\diff s\leq
\log \rz +\sigma\log\sqrt{2}\leq\log\rz+\log\sqrt{2},
\]
and \eqref{cfrrsrz} follows.

Differentiating \eqref{rprime1} with respect to $\sigma$, we have
\begin{equation}
  \label{rsecond}
\rs''=\frac
{\frac {\partial h(r e^{i\theta})}{\partial r} \rs'}{ \frac {\partial g_\theta}{\partial r}}
-\frac
{h(r e^{i\theta})\left( \frac {\partial^2 g_\theta}{\partial\sigma\partial r}+
  \frac {\partial^2 g_\theta}{\partial r^2}\rs'\right)}
{\left(\frac{\partial g_\theta}{\partial r}\right)^2}.
\end{equation}
Since by \eqref{defg} and \eqref{esthr}
\begin{align*}
  \left\vert  \frac {\partial^2 g_\theta}{\partial\sigma\partial r}  \right\vert
= \left\vert  \frac {\partial h(r e^{i\theta})}{\partial r}  \right\vert
\leq \delta C  r e^{\frac{\pi r^2}2},
  \end{align*}
while by \eqref{defg} and \eqref{esthrr}
\begin{align*}
  \left\vert  \frac {\partial^2 g_\theta}{\partial r^2}  \right\vert
  \leq
  (2\pi +4\pi^2 r^2)e^{\pi r^2}\left(\frac t T-\desqr\right)+
  \sigma \left\vert \frac {\partial^2 h(r e^{i\theta})}{\partial r^2}  \right\vert
  \\
  \leq \frac t T (2\pi +4\pi^2 r^2)e^{\pi r^2}+
  \delta C (1+ r^2)e^{\frac {\pi r^2}2}
  \leq
C(1+ r^2)e^{\pi r^2},
\end{align*}
from \eqref{rsecond} and \eqref{esth2} we see that
\begin{align*}
|\rs''|&\leq \frac
{\delta 2\sqrt 2\pi \rs e^{\frac{\pi \rs^2}2}}
{\pi t_0 \rs e^{\pi \rs^2}} |\rs'|  +
\frac
{\sqrt 2\pi\delta \rs^2 e^{\frac{\pi\rs^2}2 }}  {(\pi t_0 \rs e^{\pi \rs^2})^2}
\left(\delta C  \rs e^{\frac{\pi \rs^2}2}+  C(1+\rs^2)e^{\pi \rs^2} |\rs'|
\right)\\
&\leq \frac{\delta C e^{-\frac {\pi \rs^2} 2}}{ t_0} |\rs'|+\frac{\delta C e^{-\frac {3 \pi \rs^2} 2}}{ t_0^2}
\left(\delta  \rs e^{\frac{\pi \rs^2}2}+  (1+ \rs^2)e^{\pi \rs^2} |\rs'|
\right).
\end{align*}
Combining with \eqref{boundrsp},
\begin{equation}
 \label{boundrsecond}
|\rs''|\leq \frac{ \desqr  C \rs e^{-\pi \rs^2}}{ t_0^2 }
+\frac{\delta C e^{-\frac {3 \pi \rs^2} 2}}{ t_0^2}
\left(\delta  \rs e^{\frac{\pi \rs^2}2}+  (1+ \rs^2)\frac {\delta C\rs}{t_0}e^{\frac{\pi \rs^2}2}
\right)
\leq \frac {\desqr C\rs}{t_0^3}.
\end{equation}

\noindent\textsc{Step V: } Proof of \eqref{eqtesi2}.
Now, recalling the bounds \eqref{boundrs} and \eqref{boundrsp}, one can
differentiate under the integral in \eqref{deffs}, obtaining
\begin{equation}
\label{fprime}
f'(\sigma)=\int_0^{2\pi} \rs(\theta)\frac{\partial \rs(\theta)}{\partial\sigma}\diff \theta.
\end{equation}
Differentiating \eqref{fprime}
again, and then using \eqref{boundrsecond} and \eqref{boundrsp},  we obtain the estimate
\begin{align*}
|f''(\sigma)|\leq\int_0^{2\pi} \left( |\rs'|^2 +|\rs\rs''|\right)\diff \theta\leq
\frac {C\desqr}{t_0^3}\int_0^{2\pi} \rs^2 \diff \theta,\quad \sigma\in [0,1].
\end{align*}
This, combined with \eqref{cfrrsrz} and recalling \eqref{fzez}, gives
\begin{equation}
  \label{estfsecond}
  |f''(\sigma)|\leq \frac{C\desqr}{t_0^3} f(0), \quad\forall\sigma\in [0,1].
\end{equation}

We now claim that $f'(0)=0$, which is the crucial step of the proof. Indeed, when $\sigma=0$,
we see from \eqref{defg} and \eqref{estgr} that $\rz$ and $\partial g_\theta/\partial r$ are
independent of $\theta$, and therefore, by \eqref{rprime1}, when $\sigma=0$ we may write
\[
\left.\rs(\theta)\frac{\partial \rs(\theta)}{\partial\sigma}\right\vert_{\sigma=0}=
\rz \, \frac {h(\rz \, e^{i\theta})}{\displaystyle\frac {\partial g_\theta}{\partial r}(\rz,0)}=
\phi(\rz) h(\rz \, e^{i\theta}),
\]
where $\phi(r_0)\not=0$ depends on $\rz$ but is independent of $\theta$. Therefore, from \eqref{fprime},
\begin{align*}
  f'(0)=\int_0^{2\pi} \rs(\theta)\frac{\partial \rs(\theta)}{\partial\sigma}\Big\vert_{\sigma=0}\diff \theta
  =\phi(\rz)\int_0^{2\pi} h(\rz \, e^{i\theta})\diff \theta.
\end{align*}
On the other hand, the last integral vanishes, since the mean value theorem applied to the harmonic function $h$ gives
\[
\frac 1 {2\pi \rz}\int_0^{2\pi} h(\rz e^{i\theta})\diff \theta=h(0)=0.
\]
Hence, as $f'(0)=0$, we may write, through Taylor's formula,
\[
f(s)=f(0)+ \frac{f''(\sigma)}{2}s^2\quad \text{for some $\sigma\in (0,s)$},
\]
and taking  $s=1$ and using \eqref{estfsecond} gives
\begin{equation}
\label{estfs}
|E_1|=
f(1)\leq f(0)+\frac {C\desqr}{t_0^3} f(0)=\left(1+\frac {C\desqr}{t_0^3}\right)f(0).
\end{equation}
Now, as we may assume that $2\desqr\leq t_0$, we claim that
\begin{equation}
\label{estrz}
\pi r_0^2\leq \left(1+\frac{2\desqr}{t_0}\right)\log\frac T t,
\end{equation}
which, according to \eqref{defrz}, is equivalent to
\begin{equation}
\label{estrz2}
e^{\pi\rz^2}=\frac {1-\desqr}{\frac t T-\desqr}\leq \left(\frac T t\right)^{1+\frac {2\desqr}{t_0}}.
\end{equation}
Setting for convenience $\kappa=1/t_0$ and defining the function
\begin{equation}
\label{defpsi}
\psi(\tau):= \desqr+ \left(1-\desqr \tau \right)
\tau^{2\desqr\kappa },\quad\tau\in [1,\kappa],
\end{equation}
we observe that \eqref{estrz2} is equivalent to $\psi(T/t)\geq 1$. Since
$\psi(1)=1$, $T/t\in [1,\kappa]$
and $\psi$ is concave (note that $2\desqr\kappa\leq 1$ by assumption), it suffices to
prove that $\psi(\kappa)\geq 1$. Indeed, we have
\begin{align*}
\psi(\kappa) &=  \desqr+ \left(1-\desqr \kappa\right)e^{2\desqr \kappa \log \kappa}
  \geq \desqr+ \left(1-\desqr \kappa\right)(1+2\desqr \kappa \log \kappa)\\
  &=1+\desqr\left(1 + 2\kappa (1 -\desqr \kappa)  \log \kappa - \kappa \right)\geq
  1+\desqr\left(1 + \kappa  \log \kappa - \kappa \right),
\end{align*}
having used $1-\desqr\kappa\geq \frac 1 2$ in the last passage. This shows that $\psi(\kappa)\geq 1$,
hence \eqref{estrz} is established.

Thus, \eqref{eqtesi2} follows by combining \eqref{estfs} with \eqref{fzez} and \eqref{estrz}.\qedhere

\donotshow{

We may assume that $u(z)$ achieves its absolute maximum $T$ at $z=0$ and that $F(0)$ is real valued, so that
$F(0)=\sqrt T$. Expanding $F$ w.r.to the orthonormal basis of monomials $\{\pi^{n/2}z^n/\sqrt{n!}\}$, we have
\begin{equation}
  \label{series1}
\frac {F(z)}{\sqrt{T}}=1 +R(z),\quad z\in\C,
\end{equation}
where $R(z)$ is the entire function
\begin{equation}
  \label{defR}
R(z):=\sum_{n=2}^\infty \frac{a_n}{\sqrt{T}}\, \frac {\pi^{n/2} z^n}{\sqrt{n!}},\quad z\in\C.
\end{equation}
The assumption that $1=\Vert F\Vert_{\fock}^2$ takes the form $1=T+\sum_{n=2}^\infty |a_n|^2$, 
which we record in the form 
\begin{equation}
\label{estcoda}
\sum_{n=2}^\infty \frac{|a_n|^2}T=\frac {1-T}T =:\desqr,
\end{equation}
hereby defining $\delta$
(in the sequel we will often tacitly assume that $\delta$ is small enough, depending only on $t_0$; in
the end, the required smallness of $\delta$ will determine the threshold $T_0$ mentioned in the theorem).

From \eqref{defR}, Cauchy-Schwarz and \eqref{estcoda} we obtain
\begin{equation}
\label{eq99}
|R(z)|^2\leq \left(\sum_{n=2}^\infty \frac{|a_n|^2}T\right)
\left(\sum_{n=2}^\infty \frac {\pi^n |z|^{2n}}{n!}\right)
= \desqr \left(e^{\pi |z|^2}-1-\pi |z|^2\right).
\end{equation}
In particular, $|R(z)|^2\leq \desqr \left(e^{\pi |z|^2}-1\right)$, hence squaring \eqref{series1}
we have
\begin{equation}
  \label{estmodsq}
  \frac {|F(z)|^2}T \leq 1 + \desqr \left(e^{\pi |z|^2}-1\right) + h(z),
\end{equation}
where $h(z)$ is the 
real valued harmonic function
\begin{equation}
  \label{defh}
h(z):=2\mathop{\textrm{Re}} R(z),\qquad z\in\C.
\end{equation}
Since $|h(z)|\leq 2 |R(z)|$, the elementary inequality $e^x-1-x\leq \frac {x^2}2 e^x$, written with $x=\pi|z|^2$ and combined
with \eqref{eq99},  gives after taking square roots
\begin{equation}
  \label{esth2}
  |h(z)|\leq \sqrt{2}\pi \,\delta |z|^2 e^{\frac{\pi |z|^2}2},\qquad\forall z\in\C.
\end{equation}
Working in polar coordinates, we also need to estimate (uniformly w.r.to the angular variable $\theta$) the first and second
radial derivatives of $h(r e^{i\theta})$. Since by \eqref{defh}
and \eqref{defR} 
\begin{equation}
\label{powersh}
h(r e^{i\theta})= 2 \sum_{n=2}^\infty \left( \frac {\mathop{\textrm{Re}} a_n \, e^{i n\theta}}{\sqrt T} \right)
\frac {\pi^{n/2} r^n}{\sqrt {n!}},
\end{equation}
differentiating, and then using Cauchy--Schwarz and \eqref{estcoda} as in \eqref{eq99}, we have
\[
\left\vert \frac {\partial h(r e^{i\theta})}{\partial r} \right\vert
\leq 2 \sum_{n=2}^\infty  \frac {|a_n|}{\sqrt T}
\frac {n \pi^{n/2} r^{n-1}}{\sqrt {n!}}
\leq 2 \delta\left( \sum_{n=2}^\infty
\frac {n^2 \pi^{n} r^{2(n-1)}}{n!}
\right)^{\frac 1 2}
\]
and therefore, since when $n\geq 2$ one has $\frac {n^2}{n!}\leq \frac 2 {(n-2)!}$, it is easy to see that
\begin{equation}
\label{esthr}
\begin{aligned}
\left\vert \frac {\partial h(r e^{i\theta})}{\partial r} \right\vert\leq
\delta 2\sqrt{2}  \pi r e^{\frac {\pi r^2}2}.
\end{aligned}
\end{equation}
In a similar fashion, differentiating \eqref{powersh} twice, we find
\begin{equation}
\label{esthrr}
\left\vert \frac {\partial^2 h(r e^{i\theta})}{\partial r^2} \right\vert
\leq 2 \delta\left( \sum_{n=2}^\infty
\frac {n^2(n-1)^2 \pi^{n} r^{2(n-2)}}{n!}
\right)^{\frac 1 2}< 
\delta C (1+r^2) e^{\frac {\pi r^2}2}
\end{equation}
(the power series evaluates to $(2+4\pi r^2+\pi^2 r^4)\pi^2 e^{\pi r^2}$).

Now fix an arbitrary $t_0\in (0,1)$ and, assuming $T>t_0$, consider any $t\in [t_0,T)$ and 
any complex number $z=r e^{i\theta}$ ($r\geq 0$),
and observe that
\[
u(r e^{i\theta})>t\quad\iff\quad
\frac  {t e^{\pi r^2}}T < \frac {|F(r e^{i\theta})|^2}T. 
\]
Hence, by virtue of \eqref{estmodsq}, we obtain the implication
\begin{equation}
\label{implication}
u(r e^{i\theta})>t\quad\Longrightarrow\quad
g_\theta(r,1)
<1
\end{equation}
where, for every fixed $\theta\in [0,2\pi]$, $g_\theta$ is the function of two variables
\begin{equation}
\label{defg}
  g_\theta(r,\sigma):=  e^{\pi r^2}\left(\frac t T -\desqr\right)+\desqr  - \sigma \,h(r e^{i\theta}),\quad
  r\geq 0,\quad \sigma\in [0,1].
\end{equation}
The variable $\sigma\in [0,1]$ plays the role of a parameter, that defines the
family of planar sets
\[
E_\sigma:=\left\{ r e^{i\theta}\in\C\,|\,\, g_\theta(r,\sigma)<1\right\},\quad
\sigma\in [0,1].
\]
Since \eqref{implication} is equivalent to the set inclusion $\{u>t\}\subseteq E_1$, \eqref{newestmu} will be proved if we
show that
\begin{equation}
  \label{eqtesi2}
  |E_1|\leq \left(1+\frac {C\desqr}{t_0^4}\right)\log\frac T t.
\end{equation}
The advantage of the parameter $\sigma$ is that we can easily prove a similar estimate for $E_0$ (which is
a circle, because $g_\theta(r,0)$ is independent of $\theta$), and then show that this estimate is inherited
by every $E_\sigma$ (including $E_1$), 
by exploiting a cancellation effect due to the harmonicity of $h$.

We first show that each set $E_\sigma$ is \emph{star-shaped} with respect to the origin, by showing that
$g_\theta(r,\sigma)$ is  increasing in $r$ (for fixed $\theta$ and $\sigma$).
Using \eqref{esthr} and assuming e.g. that $\desqr+\sqrt 2\delta \leq t_0/2$, we have from \eqref{defg} that
\begin{equation}
\label{estgr}
\begin{aligned}
  \frac {\partial g_\theta(r,\sigma)}{\partial r} & =
 2\pi r e^{\pi r^2}\left(\frac t T -\desqr\right)+\sigma \frac {\partial h(r e^{i\theta})}{\partial r} \\
 &\geq 2\pi r e^{\pi r^2}\left(t_0  -\desqr\right) -  \delta 2\sqrt{2}  \pi r e^{\frac {\pi r^2}2} \\
 &\geq 2\pi r e^{\pi r^2}\left(t_0  -\desqr-\sqrt 2 \delta\right) \geq \pi t_0 r e^{\pi r^2}>0.
  \end{aligned}
\end{equation}
Since $g_\theta(0,\sigma)=t/T\geq t_0$, integrating the previous bound  we also obtain that
\begin{equation}
\label{growthg}
g_\theta(r,\sigma)\geq \frac {t_0}2\bigl(1+ e^{\pi r^2}\bigr),\quad \forall r\geq 0,\quad \forall \sigma\in [0,1]
\end{equation}
and hence, since $g_\theta(0,\sigma)=t/T< 1$, for every $\sigma\in [0,1]$ the equation
\begin{equation}
\label{defrs}
g_\theta(\rs,\sigma)=1
\end{equation}
has a unique solution $\rs> 0$ (also denoted by $\rs(\theta)$
 when the dependence of $\rs$ on the angle $\theta$ is to be stressed, as in \eqref{deffs} below). 
 Since $E_\sigma$ is star-shaped, using polar coordinates
we can compute its area $|E_\sigma|$ in terms of $\rs$,
  as
\begin{equation}
\label{deffs}
f(\sigma):= |E_\sigma|=\frac 1 2 \int_0^{2\pi}\rs(\theta)^2\diff \theta,\quad \sigma\in [0,1].
\end{equation}
Notice that  $f(1)$ is the area of $E_1$ that we want to estimate
as in \eqref{eqtesi2},  while
 \begin{equation}
   \label{fzez}
 f(0)=\frac 1 2\int_0^{2\pi} \rz(\theta)^2\diff \theta=\pi \rz^2
\end{equation}
because, when $\sigma=0$, equation  \eqref{defrs} simplifies to
\begin{equation}
  \label{defrz}
  e^{\pi \rz^2}\left(\frac t T -\desqr\right)+\desqr
  =1,
\end{equation}
so $\rz$ is independent of $\theta$ and $E_0$ is a ball of radius $\rz$. Note that the sets $E_\sigma$ are
uniformly bounded, since \eqref{growthg} and \eqref{defrs} entail that
\begin{equation}
  \label{boundrs}
   \pi\rs^2 \leq \log\frac 2 {t_0},\quad\forall \sigma\in [0,1].
\end{equation}
Moreover,
by \eqref{estgr} and the implicit function theorem,
 $\rs$ is (for every fixed value of $\theta\in  [0,2\pi]$) a smooth, bounded function of the
parameter $\sigma\in [0,1]$.
Denoting for simplicity by $\rs'$ its derivative w.r.to $\sigma$,
we have
\begin{equation}
  \label{rprime1}
\rs'=\frac{\partial \rs(\theta)}{\partial \sigma}=-\,\,
\frac {\frac {\partial g_\theta}{\partial\sigma}(\rs,\sigma)}
{\frac{\partial g_\theta}{\partial r}(\rs,\sigma)}=
\frac {h(\rs e^{i\theta})}{\frac {\partial g_\theta}{\partial r}(\rs,\sigma)},\quad
\sigma\in [0,1]
\end{equation}
and, using \eqref{esth2} 
and \eqref{estgr}, we find the bound
\begin{equation}
\label{boundrsp}
|\rs'|
\leq
\frac {\sqrt 2 \,\delta\rs}{t_0}e^{-\frac {\pi \rs^2}2 }.
\end{equation}
In particular, this implies that
\begin{equation}
  \label{cfrrsrz}
  \rs^2 \leq 2 \rz^2\quad\forall \sigma\in [0,1]
\end{equation}
because, since by \eqref{boundrsp} $|\rs'|/\rs\leq  \sqrt 2\,\delta/t_0 \leq \log\sqrt{2}$ (provided $\delta$
is small enough), we have for every $\sigma\in [0,1]$
\[
\log \rs =\log \rz + \int_0^\sigma \frac {r_s'}{r_s}\diff s\leq
\log \rz +\sigma\log\sqrt{2}\leq\log\rz+\log\sqrt{2}
\]
and \eqref{cfrrsrz} follows.

Differentiating \eqref{rprime1} w.r.to $\sigma$, we have
\begin{equation}
  \label{rsecond}
\rs''=\frac
{\frac {\partial h(r e^{i\theta})}{\partial r} \rs'}{ \frac {\partial g_\theta}{\partial r}}
-\frac
{h(r e^{i\theta})\left( \frac {\partial^2 g_\theta}{\partial\sigma\partial r}+
  \frac {\partial^2 g_\theta}{\partial r^2}\rs'\right)}
{\left(\frac{\partial g_\theta}{\partial r}\right)^2}.
\end{equation}
Since by \eqref{defg} and \eqref{esthr}
\begin{align*}
  \left\vert  \frac {\partial^2 g_\theta}{\partial\sigma\partial r}  \right\vert
= \left\vert  \frac {\partial h(r e^{i\theta})}{\partial r}  \right\vert
\leq \delta C  r e^{\frac{\pi r^2}2},
  \end{align*}
while by \eqref{defg} and \eqref{esthrr}
\begin{align*}
  \left\vert  \frac {\partial^2 g_\theta}{\partial^2 r}  \right\vert
  &\leq (2\pi +4\pi^2 r^2)e^{\pi r^2}\left(\frac t T-\desqr\right)+
  \sigma \left\vert \frac {\partial^2 h(r e^{i\theta})}{\partial r^2}  \right\vert
  \\
  & \leq \frac t T (2\pi +4\pi^2 r^2)e^{\pi r^2}+
  \delta C (1+ r^2)e^{\frac {\pi r^2}2}
  \leq
C(1+ r^2)e^{\pi r^2},
\end{align*}
from \eqref{rsecond} and \eqref{esth2} we see that
\begin{align*}
|\rs''|&\leq \frac
{\delta 2\sqrt 2\pi \rs e^{\frac{\pi \rs^2}2}}
{\pi t_0 \rs e^{\pi \rs^2}} |\rs'|  +
\frac
{\sqrt 2\pi\delta \rs^2 e^{\frac{\pi\rs^2}2 }}  {(\pi t_0 \rs e^{\pi \rs^2})^2}
\left(\delta C  \rs e^{\frac{\pi \rs^2}2}+  C(1+\rs^2)e^{\pi \rs^2} |\rs'|
\right)\\
&\leq \frac{\delta C e^{-\frac {\pi \rs^2} 2}}{ t_0} |\rs'|+\frac{\delta C e^{-\frac {3 \pi \rs^2} 2}}{ t_0^2}
\left(\delta  \rs e^{\frac{\pi \rs^2}2}+  (1+ \rs^2)e^{\pi \rs^2} |\rs'|
\right).
\end{align*}
Combining with \eqref{boundrsp},
\begin{equation}
 \label{boundrsecond}
|\rs''|\leq \frac{ \desqr  C \rs e^{-\pi \rs^2}}{ t_0^2 }
+\frac{\delta C e^{-\frac {3 \pi \rs^2} 2}}{ t_0^2}
\left(\delta  \rs e^{\frac{\pi \rs^2}2}+  (1+ \rs^2)\frac {\delta C\rs}{t_0}e^{\frac{\pi \rs^2}2}
\right)
\leq \frac {\desqr C\rs}{t_0^3}.
\end{equation}
Now, recalling the bounds \eqref{boundrs} and \eqref{boundrsp}, one can
differentiate under the integral in \eqref{deffs}, obtaining
\begin{equation}
\label{fprime}
f'(\sigma)=\int_0^{2\pi} \rs(\theta)\frac{\partial \rs(\theta)}{\partial\sigma}\diff \theta
=\int_0^{2\pi} \rs \rs' \diff \theta
\end{equation}
(in the last integral, the dependence of $\rs$ and $\rs'$ on $\theta$ is understood). 
Differentiating \eqref{fprime}
again, and then using \eqref{boundrsecond} and \eqref{boundrsp},  we obtain the estimate
\begin{align*}
|f''(\sigma)|\leq\int_0^{2\pi} \left( |\rs'|^2 +|\rs\rs''|\right)\diff \theta\leq
\frac {C\desqr}{t_0^3}\int_0^{2\pi} \rs^2 \diff \theta,\quad \sigma\in [0,1].
\end{align*}
This, combining with \eqref{cfrrsrz} and recalling \eqref{deffs}, gives
\begin{equation}
  \label{estfsecond}
  |f''(\sigma)|\leq \frac{C\desqr}{t_0^3} f(0)\quad\forall\sigma\in [0,1].
\end{equation}

We now claim that $f'(0)=0$, which is the crucial step of the proof. Indeed, when $\sigma=0$,
we see from \eqref{defg} and \eqref{estgr} that $\rz$ and $\partial g_\theta/\partial r$ are
independent of $\theta$, and therefore, by \eqref{rprime1}, when $\sigma=0$ we may write
\[
\left.\rs(\theta)\frac{\partial \rs(\theta)}{\partial\sigma}\right\vert_{\sigma=0}=
\rz \, \frac {h(\rz \, e^{i\theta})}{\displaystyle\frac {\partial g_\theta}{\partial r}(\rz,0)}=
\phi(\rz) h(\rz \, e^{i\theta}),
\]
where $\phi(r_0)\not=0$ depends on $\rz$ but is independent of $\theta$. Therefore, from \eqref{fprime},
\begin{align*}
  f'(0)=\int_0^{2\pi} \rs(\theta)\frac{\partial \rs(\theta)}{\partial\sigma}\bigl\vert_{\sigma=0}\diff \theta
  =\phi(\rz)\int_0^{2\pi} h(\rz \, e^{i\theta})\diff \theta,
\end{align*}
but the last integral vanishes because, $h$ being harmonic, the mean value theorem
for harmonic functions gives
\[
\frac 1 {2\pi \rz}\int_0^{2\pi} h(\rz e^{i\theta})\diff \theta=h(0)=0.
\]
Hence,  applying Taylor's formula 
\[
f(s)=f(0)+ f'(0)s +\frac{f''(\sigma)}{2}s^2\quad \text{for some $\sigma\in (0,s)$}
\]
with $s=1$, recalling that  $f'(0)=0$  and
using \eqref{estfsecond}  gives
\begin{equation}
\label{estfs}
|E_1|=
f(1)\leq f(0)+\frac {C\desqr}{t_0^3} f(0)=\left(1+\frac {C\desqr}{t_0^3}\right)f(0).
\end{equation}
Now, assuming (as we may) that $2\desqr\leq t_0$, we claim that
\begin{equation}
\label{estrz}
\pi r_0^2\leq \left(1+\frac{2\desqr}{t_0}\right)\log\frac T t
\end{equation}
which, according to \eqref{defrz}, is equivalent to
\begin{equation}
\label{estrz2}
e^{\pi\rz^2}=\frac {1-\desqr}{\frac t T-\desqr}\leq \left(\frac T t\right)^{1+\frac {2\desqr}{t_0}}.
\end{equation}
Setting for convenience $\kappa=1/t_0$ and defining the function 
\begin{equation}
\label{defpsi}
\psi(\tau):= \desqr+ \left(1-\desqr \tau \right)
\tau^{2\desqr\kappa },\quad\tau\in [1,\kappa],
\end{equation}
we observe that \eqref{estrz2} is equivalent to $\psi(T/t)\geq 1$. Since 
$\psi(1)=1$, $T/t\in [1,\kappa]$ 
and $\psi$ is concave (note that $2\desqr\kappa\leq 1$ by assumption), it suffices to
prove that $\psi(\kappa)\geq 1$. Indeed, we have
\begin{align*}
\psi(\kappa) &=  \desqr+ \left(1-\desqr \kappa\right)e^{2\desqr \kappa \log \kappa}
  \geq \desqr+ \left(1-\desqr \kappa\right)(1+2\desqr \kappa \log \kappa)\\
  &=1+\desqr\left(1 + 2\kappa (1 -\desqr \kappa)  \log \kappa - \kappa \right)\geq
  1+\desqr\left(1 + \kappa  \log \kappa - \kappa \right),
\end{align*}
having used $1-\desqr\kappa\geq \frac 1 2$ in the last passage. This shows that $\psi(\kappa)\geq 1$,
hence \eqref{estrz} is established.

Then \eqref{eqtesi2} follows, on combining \eqref{estfs} with \eqref{fzez} and \eqref{estrz}.

} 

\end{proof}

\donotshow{

\begin{numlemma}\label{lemma:super-level} 
Let $F \in \mathcal{F}^2$ be such that $\|F\|_{\mathcal{F}^2} = 1.$ There is an absolute, computable constant $C>0$  such that the following holds: if $T := \max_{z \in \C} |F(z)|^2 e^{-\pi|z|^2}$ satisfies $1-T < \frac{1}{C},$ then, letting
\[\mu_F(t) := |\{ z \in \C \colon |F(z)|^2 e^{-\pi|z|^2} > t\}|,
\] 
we have  
\[ 
\mu_F(t) < (1+C(1-T)) \log \left( \frac{T}{t}\right),
\]
whenever $\frac{t}{T} > 1 - \frac{1}{C}.$
\end{numlemma}

\begin{proof} Without loss of generality, we may assume that the maximum of $u_F(z) = |F(z)|^2 e^{-\pi|z|^2}$ occurs at $z =0.$ Moreover, we may also assume that $F(0)$ is a real number, and thus we have $T = F(0)^2.$ \medskip

\noindent\textsc{Step 1.} We will prove first the sub-optimal inclusion 
\[
\{ z \in \C \colon u_F(z) > (1-\varepsilon)^2 T \} \subset D(0,C \sqrt{\varepsilon}), 
\]
for some absolute constant $C>0.$ Indeed, if we write $F$ as 
\begin{equation}\label{eq:power-series-F}  
F(z) = \sum_{k =0}^{\infty} a_k \frac{\pi^{k/2}}{\sqrt{k!}} z^k,
\end{equation} 
then the estimate 
\begin{equation}\label{eq:triangular-F}  
|F(z)|e^{-\pi \frac{|z|^2}{2}} \le e^{-\pi \frac{|z|^2}{2}} \left( \sum_{k=0}^{\infty} |a_k| \frac{\pi^{k/2}}{\sqrt{k!}} |z|^k\right) 
\end{equation} 
follows at once. In addition to this, one easily observes that, if $u_F$ has a maximum at $z=0,$ then $|F(z)|^2$ must have a critical point there as well. On the other hand, 
\[
|F(z)|^2 = T + \pi^{1/2} T^{1/2} \cdot \Re(a_1 z) + O(|z|^2)
\]
holds in a neighbourhood of the origin, which implies that $a_1 =0$; see also the discussion at the beginning of Section \ref{subsec:GeometryOfSuperLevelSets}. Hence, by \eqref{eq:triangular-F}, if we restrict $z$ to the set $\{ z \in \C \colon u_F(z) > (1-\varepsilon)^2 T \},$ using that $a_1 = 0$ and the Cauchy-Schwarz inequality, 
\begin{align}\label{eq:chain-bound-F}
\begin{split}
(1-\varepsilon) T^{1/2} &\le |F(z)| e^{-\pi\frac{|z|^2}{2}} \cr 
                & \le T^{1/2} \cdot e^{-\pi \frac{|z|^2}{2}} + e^{-\pi \frac{|z|^2}{2}} \left(\sum_{k =2} |a_k| \frac{\pi^{k/2}}{\sqrt{k!}} |z|^k \right) \cr 
                & \le T^{1/2} \cdot e^{-\pi \frac{|z|^2}{2}} + e^{-\pi \frac{|z|^2}{2}} (1-T)^{1/2} \left( e^{\pi |z|^2} - 1 - \pi |z|^2 \right)^{1/2} \cr 
                & \le T^{1/2} \cdot e^{-\pi \frac{|z|^2}{2}} + (1-T)^{1/2} \left( 1- (1+\pi|z|^2)\cdot e^{-\pi |z|^2} \right)^{1/2}. 
\end{split}
\end{align} 
Let $\tau := \pi |z|^2$, so that \eqref{eq:chain-bound-F} may be rewritten as 
\[
(1-\varepsilon) \le e^{-\tau/2} + \left( \frac{1-T}{T}\right)^{1/2}\left(1- (1+\tau) e^{-\tau} \right)^{1/2}. 
\]
Using the elementary algebraic inequality 
$\left(1- (1+\tau) e^{-\tau} \right)^{1/2} \le 1 - e^{-\tau}$, valid for all $\tau \ge 0$, and writing $\delta=\left( \frac{1-T}{T}\right)^{1/2}$, we conclude that 
\[
1-\varepsilon \le e^{-\tau/2} + \delta (1-e^{-\tau}). 
\]
Finally, if we set $ x := 1- e^{-\tau/2}$, the last inequality can be rewritten as 
$$x - \varepsilon \le \delta ( 2x - x^2) \le  2\delta x,$$
which, in turn, implies $x \le \frac{\varepsilon}{1-2\delta}$ or, equivalently, 
$$\tau \le -2\log \left(  1- \frac{\varepsilon}{1-2\delta} \right)\leq C \frac{\e}{1-2\delta}.$$ Since $1-T<1/C$ we have $\delta \leq (\frac{1}{C-1})^{1/2}$ and, as $\tau = \pi |z|^2,$ one easily concludes the assertion of the first step. \medskip

\noindent\textsc{Step 2.} We now refine the computations above, in order to obtain the conclusion of the lemma. 

We first assume that $t$ is arbitrary, but in a moment it will be chosen to be so that $T/t$ is very small, depending on $\delta$, as in the statement. Suppose then that $u_F(z) > t.$ We use the power-series expansion \eqref{eq:power-series-F} in explicit form: indeed, we let 
\[
R(z) := \frac{F(z)}{F(0)} - 1 = \sum_{k=2}^\infty a_k \frac{\pi^{k/2}}{\sqrt{k!}} z^k.
\]
Notice first that, by the estimates in \eqref{eq:chain-bound-F}, we have 
\begin{equation}\label{eq:first-estimate-R} 
|R(z)|^2 \le (1-T) \left( e^{\pi |z|^2} - 1 - \pi |z|^2 \right) \le \delta^2 \left( e^{\pi |z|^2} - 1 - \pi |z|^2 \right)
\end{equation}
where $\delta = \left(\frac{1-T}{T}\right)^{1/2},$ as in \textsc{Step 1.} The condition $u_F(z) > t$ can be rewritten as
\[
\pi |z|^2 + \log(t/T) < \log\left( \frac{|F(z)|^2}{T}\right),
\]
and so
\begin{equation}\label{eq:fundamental-R}
\pi |z|^2 + \log(t/T) < \log \left( 1 + 2 \Re R(z) + |R(z)|^2\right) \le |R(z)|^2 + 2 \Re R(z). 
\end{equation}
We now relabel $h(z) := - 2\Re R(z).$ This, together with \eqref{eq:first-estimate-R}, implies 
\begin{equation}\label{eq:first-h} 
\pi |z|^2 + h(z) < \log(T/t) + \delta^2 \left( e^{\pi |z|^2} - 1 - \pi |z|^2 \right). 
\end{equation}
Choosing $t$ so that
\[\frac t T = (1-\e_0)^2 T,\]
\textsc{Step 1} readily implies that $\{u_F>t\}\subset D(0,C\sqrt \e_0)$; hence, if $\e_0$ is chosen sufficiently small, then $\pi |z|^2 \le 1$ whenever $u_F(z) > t.$  Thus, for such $z$ we have $e^{\pi |z|^2} - 1 - \pi |z|^2 < \pi |z|^2,$ and from \eqref{eq:first-h}, 
\begin{equation}\label{eq:def-V}
(1-\delta^2)\pi |z|^2 + h(z) < \log(T/t). 
\end{equation}

We now claim that there is an absolute constant $C>0$ such that, for $z \in B(0,1/\sqrt{\pi}),$ we have 
\begin{equation}\label{eq:estimates-h-derivatives} 
|h(z)| \le C \delta |z|^2, \quad |\nabla h(z)| \le C \delta |z|, \quad |\nabla^2 h(z)| \le C \delta. 
\end{equation}
Note that the bound \eqref{eq:first-estimate-R} yields
\begin{equation}
    \label{eq:boundR}
    |R(z)|\leq \pi \delta |z|^2.
\end{equation}
The first inequality in \eqref{eq:estimates-h-derivatives} now follows directly from the definition of $h$ and \eqref{eq:boundR}.
In order to show the estimates for the derivatives of $h$, note that $|\nabla h| \le 2|R'(z)|, \, \, |\nabla^2 h| \le 4|R''(z)|.$ But, by the Cauchy integral formula,
\begin{equation}\label{eq:Cauchy-integral-R}  
|R^{(n)}(z)| \le \frac{n!}{2\pi} \int_{\gamma_z} \frac{|R(w)|}{|w-z|^{n+1}} |\diff w|, 
\end{equation} 
where $\gamma_z$ is a closed, simple curve containing $z.$  Applying \eqref{eq:Cauchy-integral-R} with $\gamma_z=\p D(0,2|z|)$ and using \eqref{eq:boundR} we obtain the estimate for $\nabla h$, while applying \eqref{eq:Cauchy-integral-R} with $\gamma_z=\p B(0,2/\sqrt{\pi})$ and using again \eqref{eq:boundR}, the estimate for $\nabla^2 h$ also follows.

Inspired by these considerations, we let 
\[
V_{\sigma}(r,\theta) : = (1-\delta^2)\pi r^2 + \frac{\sigma}{\delta} h(re^{i\theta}) 
\]
where $\sigma \in (-2\delta,2\delta)$.
By \eqref{eq:estimates-h-derivatives}, we see that 
\begin{equation}
    \label{eq:invertibilityrsigma}
    \partial_r V_{\sigma} \ge  2r(1-\delta^2) - C \delta r > 0
\end{equation}
as long as $0<r<1/\sqrt{\pi}$ and $
\delta$ is sufficiently small. Moreover, $\lim_{r\to0} V_{\sigma}(r,\theta) = 0 < \log(T/t),$ and, for each fixed $\theta \in [0,2\pi],$ we have
\[
V_{\sigma}(1/\sqrt \pi,\theta) > (1-\delta^2) - C\delta > \log(T/t) 
\]
 since we are assuming that $\log(T/t)$ is sufficiently small. Therefore, given $t > 0$ with $T/t$ sufficiently close to 1, for each $\theta \in [0,2\pi)$ there is a unique number $r_{\sigma}(\theta) > 0$ such that 
\begin{equation}\label{eq:def-implicit-r}
V_{\sigma}(r_{\sigma}(\theta),\theta) = \log(T/t).
\end{equation}
Due to \eqref{eq:invertibilityrsigma}, the implicit function theorem shows that $\sigma \mapsto r_{\sigma}(\theta)$ is a smooth function of $\sigma$ for each fixed $\theta \in [0,2\pi).$ We can thus differentiate \eqref{eq:def-implicit-r} with respect to $\sigma$ to obtain 
\begin{equation}\label{eq:equation-r-derivative}
2(1-\delta^2)\pi r_\sigma r_\sigma' + \frac{1}{\delta} h(r_\sigma e^{i\theta}) + \frac{\sigma}{\delta} \langle \nabla h(r_{\sigma} e^{i \theta}), e^{i\theta}\rangle r_\sigma' = 0. 
\end{equation}
Solving for $r_\sigma',$ we obtain 
\begin{equation}\label{eq:formula_r_sigma} 
    r_\sigma' = - \frac{ \frac{1}{\delta} h(r_\sigma e^{i \theta})}{2(1-\delta^2)\pi r_\sigma + \frac{\sigma}{\delta} \langle \nabla h(r_{\sigma} e^{i \theta}), e^{i\theta}\rangle}
\end{equation}
and so, by \eqref{eq:estimates-h-derivatives}, we have $
|r_\sigma'| \le C r_\sigma.$ We now compute the second derivative of $r_\sigma$: differentiating \eqref{eq:equation-r-derivative} once more with respect to $\sigma,$ we obtain 
\begin{align}
\begin{split}
\label{eq:second-derivative-radius} 
 2(1-\delta^2)\pi\left( (r_\sigma')^2  + r_\sigma r_\sigma'' \right) & + \frac{2}{\delta}\langle\nabla h(r_\sigma e^{i\theta}), e^{i \theta}\rangle r_\sigma' \\
& + \frac{\sigma}{\delta}  \nabla^2 h(r_{\sigma} e^{i \theta})[e^{i\theta}, e^{i\theta}] + \frac{\sigma}{\delta} \langle \nabla h(r_{\sigma} e^{i \theta}), e^{i\theta}\rangle r_{\sigma}'' = 0. 
\end{split}
\end{align}
Solving for $r_\sigma''$ in \eqref{eq:second-derivative-radius} and using \eqref{eq:estimates-h-derivatives} and $|r_\sigma'| \le C r_\sigma$ repeatedly, we obtain that $|r_\sigma''| \le C r_\sigma.$ We summarize these in the following: 

\begin{lemma}\label{prop:estimate-r-derivatives}
Let $r_\sigma$ be defined through \eqref{eq:def-implicit-r}. Then there is an absolute, computable constant $C> 0$ such that, for $\sigma \in (-2\delta,2\delta),$ we have 
\begin{equation}\label{eq:bounds-derivatives-r} 
|r_{\sigma}'| + |r_{\sigma}''| \le C r_{\sigma}.
\end{equation}
\end{lemma}

We will use these new estimates in order to estimate the area of the super-level set $\{ u_F > t\}.$ Indeed, we trivially have $\{ u_F > t \} \subset \{ z: 0 < |z| < r_\delta(\theta)\},$ and thus 
\[
\mu_F(t) \le \frac{1}{2} \int_0^{2\pi} r_\delta(\theta)^2 \diff \theta, 
\]
which motivates the definition of
\[ f(\sigma) \coloneqq \frac{1}{2} \int_0^{2\pi} r_{\sigma}(\theta)^2 \diff \theta,  \]
where as before we assume that $\sigma \in (-2\delta,2\delta).$ Differentiating $f$ and passing the derivative inside the integral (which, due to the bound \eqref{eq:bounds-derivatives-r}, is allowed), we obtain 
\begin{equation}
    \label{eq:boundsecondderf}
    |f''(\sigma)| \le C \int_0^{2\pi} r_\sigma(\theta)^2  \diff \theta. 
\end{equation}
Furthermore, notice that $r_\sigma'|_{\sigma = 0} = \frac{h(r_0 e^{i\theta})}{\delta(1-\delta^2)\pi r_0}$ by \eqref{eq:formula_r_sigma}. Recall that, by definition, $h$ is harmonic and so, by the mean value theorem, 
$$f'(0) = \frac{2}{\delta(1-\delta^2)} h(0) = 0.$$
The Taylor expansion of $f$ becomes
\[
f(\delta) = f(0) +  \frac{f''(\sigma)}{2} \delta^2,
\]
for some  $\sigma\in(0,\delta)$.
Using \eqref{eq:boundsecondderf}, we get 
\begin{equation}\label{eq:function-bound}  
f(\delta) \le f(0) + C\delta^2 \int_0^{2\pi} r_\sigma(\theta)^2 \, \diff \theta. 
\end{equation} 
On the other hand, by \eqref{eq:def-implicit-r} evaluated at $0$ and $\sigma,$ we have 
\[
(1-\delta^2)\pi(r_\sigma^2 - r_0^2) = - \frac{\sigma}{\delta} h(r_\sigma e^{i\theta}).
\]
By using \eqref{eq:estimates-h-derivatives} and taking $\delta$ sufficiently small, we have $ r_\sigma^2 \le 2 r_0^2$ for all $\sigma \in (-2\delta,2\delta).$  Thus, 
$$\mu_F(t) \le f(\delta) \le \frac{1+ C \delta^2}{2} \int_0^{2\pi} r_0^2  \diff \theta = \frac{1+C\delta^2}{1-\delta^2} \log\frac{T}{t},$$
since $(1-\delta^2) \pi r_0^2 = \log \frac{T}{t}$.
\end{proof}

} 

\begin{numcor}[Uniqueness and non-degeneracy of $t^*$]\label{cor:nondegen}
    If $F\in\Fock$ is such that
    $\Vert F\Vert_{\Fock}=1$ and $T<1$, then
there is a unique value $t^*\in (0,T)$
satisfying \eqref{deft*}. Moreover,
\begin{equation}
    \label{eq:nondegen}
t^* \leq \tau^*,
\end{equation}
for some universal constant $\tau^*\in (0,1)$.
\end{numcor}
Note that the uniqueness of $t^*$ implies the uniqueness of
$s^*$ defined in \eqref{defs*}, and $t^*=e^{-s^*}$. We also
note that there cannot be any universal lower bound on
$t^*$, since $t^*\leq T$ and $T$ can be arbitrarily small.
\begin{proof}
 If \eqref{deft*} were true for two distinct values $t_1<t_2<T$
 of $t^*$, then we would have $\mu(t)=\log 1/t$
 for every $t\in [t_1,t_2]$, whence
  $\mu'(t)=-1/t $ for every $t\in (t_1,t_2)$.
 But the proof of \cite[Lemma~3.4]{NicolaTilli} shows that 
 this happens if  and only if the corresponding
 sets $\{ u > t \}$ are balls,  $|\nabla u|$ being constant on each boundary $\partial \{ u > t \} = \{ u = t \}$: this in turn 
 implies that $u(z) = e^{-\pi|z-z_0|^2},$ for some $z_0 \in \C$ and hence $u(z_0) = 1$ (or equivalently that $u^*(s)\equiv e^{-s}$), contradicting to our assumption
 that $T<1$.
 
Now let $T_0$ and $C_0$ be the constants provided by Lemma \ref{lemma:super-level-new}
when $t_0=\frac 1 2$, and define
\[
\tau^*:= \max\left\{\frac 1 2, T_0,e^{-\,\frac 1 {C_0}} \right\}.
\]
Given $F$ as in our statement, if $t^*\leq 1/2$ then clearly $t^*\leq \tau^*$, and
the same is true if $T< T_0$, because certainly $t^*\leq T$. Finally, if $t^*>1/2$
and $T\geq T_0$, then \eqref{newestmu} written with $t=t^*$ becomes
\[
\log \frac 1 {t^*}=\mu(t^*) \leq (1+C_0(1-T))\log \frac T {t^*},
\]
which is equivalent to
\[
\frac 1 {t^*} \leq \left(\frac T {t^*}\right)^{1+C_0(1-T))},\quad\text{that is,}\quad
t^* \leq T ^{1+\frac 1 {C_0(1-T)}}.
\]
But then, since $T\leq 1$, we obtain
\[
t^* \leq T ^{\,\frac 1 {C_0(1-T)}} \leq e^{-\,\frac 1 {C_0}}
\]
and $t^*\leq\tau^*$ also in this case (notice that $x^{\frac 1 {1-x}}\leq e^{-1}$ for every  $x\in (0,1)$).
\end{proof}

We are now ready to start the comparison between $u^*(s)$
and $e^{-s}$, where the number $s^*$,  
uniquely defined  by \eqref{defs*} if $T<1$,
will play a crucial role. In the next two lemmas, however,
it is not necessary to assume that $T<1$, since when $T=1$ (and
$u^*(s)=e^{-s}$) their claims remain true (though trivial) for
all values of $s^*$.

\begin{numlemma}\label{lemma:QuantitativeMaxuNew}
    For every $F\in\Fock$ such that $\Vert F\Vert_{\Fock}=1$ and every $s_0>0$, there holds
    \begin{equation}
        \label{estlemma1}
    \frac{(1- T)^2}2 \leq \int_0^{s^*} \bigl(e^{-s}-u^*(s)\bigr)\diff s\leq  \delta_{s_0} e^{s_0},
    \end{equation}
where $T$ is as in \eqref{def:T} and
\begin{equation}
    \label{defdeltas0}
\delta_{s_0}:=1- \frac{\int_{\{u>u^*(s_0)\}} u(z)\diff z}{1-e^{-s_0}}
=1- \frac{\int_0^{s_0} u^*(s)\diff s}{1-e^{-s_0}}.
\end{equation}
\end{numlemma}

Note that $\delta_{s_0}$ coincides with the deficit
$\delta(F;\Omega)$ of Theorem \ref{thm:FKforSTFTFock} when
$\Omega=\{u>u^*(s_0)\}$ is the super-level set of $u$, with measure $s_0$.

\begin{proof}
    Instead of writing explicitly $e^{-s}$, we will use the notation
    \begin{equation}
        \label{tempv*}
        \ee(s):=e^{-s},\quad s\geq 0.
    \end{equation}
This will be particularly useful in Section \ref{sec:Generalize}, when we adapt the current proof to higher
dimensions. 

Since $u^*(x)\leq T$ and $\ee(s)\geq 1-s$, the first inequality in \eqref{estlemma1}
follows from
\begin{equation}
    \label{eq2003}
    \int_0^{s^*} \bigl(\ee(s)-u^*(s)\bigr)\diff s\geq
    \int_0^{s^*} \bigl(1-s-T\bigr)_+\diff s=\int_0^{1-T}(1-s-T)\diff s=\frac {(1-T)^2}2.
\end{equation}
To prove the second inequality,
note that $1-e^{-s_0}=\int_0^{s_0} \ee(s)\diff s,$ and hence
we can rewrite \eqref{defdeltas0} as
\begin{equation}\label{eq2001}
\eps:= \delta_{s_0} \int_0^{s_0} \ee(s)\diff s =        \int_{s_0}^\infty  \left( u^*(s)-\ee(s)  \right) \diff s = \int_0^{s_0} \left(v^*(s) - u^*(s)\right) \diff s.
    \end{equation}
The key of the proof is that the ratio
\begin{equation}
    \label{ratiodecr}
    r(s):=\frac{u^*(s)}{\ee(s)}\quad \text{is an increasing function  on $[0,+\infty)$,}
\end{equation}
as follows immediately from \eqref{eqn:DiffInequ} since $r(s)=e^{s}u^*(s)$. In order to implement such an idea, we must now
distinguish between some cases:

\textsc{Case 1:} $ s_0 > s^*$. Since $r(s^*)=1$ and $r(s)$ is increasing by the convexity inequality \eqref{eqn:DiffInequ}, we have from \eqref{eq2001}
\begin{align*}
 \eps
    = \int_{s_0}^\infty  u^*(s)\left( 1 -\frac 1 {r(s)}  \right) \diff s
    \geq \left( 1 -\frac 1 {r(s_0)}  \right) \int_{s_0}^\infty  u^*(s) \diff s.
    \end{align*}
On the other hand, for the same reason,
\[
\int_{s^*}^{s_0} \left(u^*(s)-\ee(s)\right)\diff s
= \int_{s^*}^{s_0} u^*(s) \left(1-\frac 1 {r(s)}\right)\diff s
\leq \left(1-\frac 1 {r(s_0)}\right) \int_{s^*}^{s_0} u^*(s) \diff s,
\]
which, combined with the previous estimate, gives
\[
\int_{s^*}^{s_0} \left(u^*(s)-\ee(s)\right)\diff s
\leq \eps \,\, \frac{\int_{s^*}^{s_0} u^*(s) \diff s}{\int_{s_0}^\infty  u^*(s) \diff s}.
\]
Thus, recalling \eqref{eq2001} and using the last inequality, we find that
\begin{equation}
    \label{eq2005}
\int_{s^*}^{\infty} \left(u^*(s)-\ee(s)\right)\diff s\leq \eps+\eps \,\, \frac{\int_{s^*}^{s_0} u^*(s) \diff s}{\int_{s_0}^\infty  u^*(s) \diff s}=
\eps\, \frac {\int_{s^*}^{\infty} u^*(s) \diff s}{\int_{s_0}^\infty  u^*(s) \diff s}
\leq \frac \eps{\int_{s_0}^\infty  \ee(s) \diff s},
\end{equation}
having used \eqref{massone} for the numerator, and the fact that $u^*(s)\geq\ee(s)$ when $s\geq s^*$,
for the denominator. Given that clearly $\eps\leq\delta_{s_0}$, 
the second inequality in \eqref{estlemma1} follows immediately since 
$\int_0^{s^*} \ee(s)\diff s=1-e^{-s}.$

\textsc{Case 2:} $s_0 \leq s^*$. As $r(s^*)=1$ and $r(s)$ is increasing, we have from
\eqref{eq2001} again that
\begin{align*}
 \eps
    = \int_0^{s_0}  \ee(s)\left( 1 -r(s)  \right) \diff s
    \geq \left( 1 -r(s_0)  \right) \int_0^{s_0}  \ee(s) \diff s.
    \end{align*}
On the other hand, for the same reason,
\[
\int_{s_0}^{s^*} \left(\ee(s)-u^*(s)\right)\diff s=
\int_{s_0}^{s^*}  \ee(s)\left( 1 -r(s)  \right) \diff s
    \leq \left( 1 -r(s_0)  \right) \int_{s_0}^{s^*}  \ee(s) \diff s,
\]
which combined with the previous estimate gives
\[
\int_{s_0}^{s^*} \left(\ee(s)-u^*(s)\right)\diff s
\leq \eps \,\, \frac{\int_{s_0}^{s^*} \ee(s) \diff s}{\int_0^{s_0} \ee(s) \diff s}.
\]
Thus, using the last inequality, we find
\begin{equation}
    \label{eq2006}
\int_0^{s^*}\left(\ee(s)-u^*(s)\right)\diff s\leq \eps+\eps \,\, \frac{\int_{s_0}^{s^*} \ee(s) \diff s}{\int_0^{s_0}  \ee(s) \diff s}=
\eps\, \frac {\int_0^{s^*} \ee(s) \diff s}{\int_0^{s_0} \ee(s) \diff s}
\leq \frac \eps{\int_0^{s_0}  \ee(s) \diff s}=\delta_{s_0},
\end{equation}
and the second inequality in \eqref{estlemma1} follows also in this case.
\end{proof}

\donotshow{

\begin{numlemma}\label{lemma:QuantitativeMaxuNew}
    For every $F\in\Fock$ such that $\Vert F\Vert_{\Fock}=1$ and every $s_0>0$, there holds
    \begin{equation}
        \label{estlemma1}
    \frac{(1- T)^2}2 \leq \int_0^{s^*} \bigl(e^{-s}-u^*(s)\bigr)\diff s\leq 2 \delta_{s_0} e^{s_0},
    \end{equation}
where $T$ is as in \eqref{def:T} and
\begin{equation}
    \label{defdeltas0}
\delta_{s_0}:=1- \frac{\int_{u>u^*(s_0)} u(z)\diff z}{1-e^{-s_0}}
=1- \frac{\int_0^{s_0} u^*(s)\diff s}{1-e^{-s_0}}
\end{equation}
\end{numlemma}

Note that $\delta_{s_0}$ coincides with the deficit
$\delta(F;\Omega)$ of Theorem \ref{thm:FKforSTFTFock}, when
$\Omega=\{u>u^*(_0)\}$ is the super-level set of $u$, with measure $s_0$.

\begin{proof} The first inequality in \eqref{estlemma1} follows immediately from $e^{-s}\geq 1-s$
(see Figure \ref{fig:LemmaMaxuFirstCaseAreas}, \textbf{to be updated!}). To prove the second inequality,
we distinguish two cases:

    \textsc{Case 1:} $\quad s_0 \geq s^*$. Define the function
    \begin{equation}\label{eq:horizontal-L}  
    L(t) = \mu(t) + \log(t) ,
    \end{equation} 
    which is the length of the horizontal segment joining the graphs of $e^{-s}$ and $u^*(s)$ at height $t \ls e^{-s^*}$ (see Figure~\ref{fig:LemmaMaxuFirstCaseL}).
    
    \begin{figure}[H]
        \centering
        \def\svgwidth{0.55\textwidth}
        \import{files/figures/}{LemmaMaxuFirstCaseL.pdf_tex}
        \caption{The function $L(t)$}
        \label{fig:LemmaMaxuFirstCaseL}
    \end{figure}  
    
    Observe that $L(t)$ is decreasing, since $L'(t) \leq 0$ by \eqref{eqn:DiffInequ}. Rearranging the expression for $\delta_{s_0}$ we see that
    \begin{equation}\label{eqn:InLemmaMaxuProof1}
        \int_0^{s_0} \left( u^*(s) - e^{-s} \right) \diff s = - \delta_{s_0} (1-e^{-s_0}) 
    \end{equation}
which, recalling \eqref{massone}, 
 means that
    \begin{equation}\label{eqn:InLemmaMaxuProof2}
        A \coloneqq \int_{s_0}^{\infty} \left( u^*(s) - e^{-s}\right) \diff s = \delta_{s_0} (1-e^{-s_0}) . 
    \end{equation}
    Define now $A_1$ to be the area of the region enclosed between the graphs of $e^{-s}$ and $u^*(s)$, and below the horizontal line $t=e^{-s_0}$ (see Figure~\ref{fig:LemmaMaxuFirstCaseAreas}), that is
    \[ A_1 \coloneqq \int_0^{-e^{s_0}} L(t) \diff t .\]
    Then $A \geq A_1$ and, since $L$ is decreasing,
    \[ A_1 \geq L(e^{-s_0}) e^{-s_0} = (s_1-s_0) e^{-s_0} ,\]
    where $s_1$ is such that $u^*(s_1) = e^{-s_0}$. Combining this with \eqref{eqn:InLemmaMaxuProof2} we get
    \begin{equation*}
        s_1-s_0 \leq \delta_{s_0} e^{s_0} (1-e^{-s_0}) .
    \end{equation*}
    If we now set $A_2$ to be the area of the region enclosed between the graphs of $e^{-s}$ and $u^*(s)$, but this time above the line $t=e^{-s_0}$ and below the line $t=e^{-s^*}$, then we find that
    \begin{align*}
        A_2 \coloneqq \int_{e^{-s_0}}^{e^{-s^*}} L(t) \diff t  & \leq L(e^{-s_0}) (e^{-s^*} - e^{-s_0})\\
        & = (s_1 - s_0) (e^{-s^*} - e^{-s_0})\leq \delta_{s_0} e^{s_0} (1-e^{-s_0}) (e^{-s^*} - e^{-s_0}).
    \end{align*}
    We can now combine \eqref{eqn:InLemmaMaxuProof2} with the above estimate on $A_2$ to get
    \begin{align}\label{eq:bound-trapped-area}  
    \begin{split}
        \int_{s^*}^{\infty} (u^*(s) - e^{-s}) \diff s &= \int_{s^*}^{s_0} (u^*(s)-e^{-s}) \diff s + \int_{s_0}^{\infty} (u^*(s)-e^{-s}) \diff s \cr
        &\leq A_2 + \delta_0 (1-e^{-s_0}) \leq \delta_0 (1-e^{-s_0}) e^{s_0-s^*}.
    \end{split}
    \end{align}
Finally, again from \eqref{massone},  we find
    \begin{equation}
    \label{eq:nonsharplemmaaux}
    \int_0^{s^*} \left( e^{-s} - u^*(s) \right) \diff s = \int_{s^*}^{\infty} \left( u^*(s) - e^{-s} \right) \diff s \leq \delta_{s_0} (1-e^{-s_0}) e^{s_0-s^*} 
    \end{equation}
and the second inequality in \eqref{estlemma1} follows.

    \textsc{Case 2:} $s_0 < s^*.\quad$  Since by \eqref{eqn:DiffInequ}
    \[
    \frac d {ds}\bigl( e^{-s}-u^*(s)\bigr)=
    -e^{-s}-(u^*)'(s)\leq - \bigl(e^{-s} - u^*(s)\bigr)\quad\text{for a.e. $s\geq 0$,}
    \]
    on the interval $[0,s^*]$ (where $e^{-s}\geq u^*(s)$) the function $e^{-s}-u^*(s)$
    is decreasing, hence from \eqref{eqn:InLemmaMaxuProof1}
    \begin{equation}
    \label{eq1000}
     \delta_{s_0} (1-e^{-s_0}) = \int_0^{s_0}  \bigl( e^{-s}-u^*(s)\bigr)\diff s
     \geq s_0 \bigl( e^{-s_0}-u^*(s_0)\bigr).
    \end{equation}
   On the other hand, since by \eqref{eqn:DiffInequ} 
   also $1-e^s u^*(s)$  is a decreasing function,
   we have
   \begin{align*}
   \int_{s_0}^{s^*}  \bigl( e^{-s}-u^*(s)\bigr)\diff s
   &=
   \int_{s_0}^{s^*}  e^{-s}\bigl( 1- e^s u^*(s)\bigr)\diff s\leq 
   \bigl( 1- e^{s_0} u^*(s_0)\bigr)
   \int_{s_0}^{s^*}  e^{-s}\diff s\\
   &  
   \leq \bigl( 1- e^{s_0} u^*(s_0)\bigr) e^{-s_0}=e^{-s_0} -u^*(s_0),
   \end{align*}
which combined with the previous inequality yields
\[
\int_{s_0}^{s^*}  \bigl( e^{-s}-u^*(s)\bigr)\diff s \leq \delta_{s_0} \frac{1-e^{-s_0}}{s_0}\leq 
\delta_{s_0}.
\]
This inequality, combined with the equality in \eqref{eq1000}, yields
\[
\int_{0}^{s^*}  \bigl( e^{-s}-u^*(s)\bigr)\diff s\leq 
\delta_{s_0} (1-e^{-s_0})  +\delta_{s_0} \leq 2\delta_{s_0},
\]
and the second inequality in \eqref{estlemma1} follows.

    \begin{figure}[H]
        \centering
        \def\svgwidth{0.6\textwidth}
        \import{files/figures/}{LemmaMaxuSecondCaseb.pdf_tex}
        \caption{$A_4$ and $A_5$.}
        \label{fig:LemmaMaxuSecondCaseb}
    \end{figure}

\end{proof}

} 

We are now ready to show that, in \eqref{estlemma1}, the first inequality holds in fact in a much stronger form.

\begin{numlemma}
    \label{reinforcedlemma} 
    Under the same assumptions as in Lemma \ref{lemma:QuantitativeMaxuNew}, there holds
\begin{equation}
\label{estreinforced}
    1-T \leq C \int_0^{s^*} \bigl(e^{-s}-u^*(s)\bigr)\diff s,
\end{equation}
where $C>0$ is a universal constant.
\end{numlemma}

\begin{proof}
Passing to the inverse functions, and recalling that $\mu(t)$, restricted to $(0,T)$, is
the inverse of $u^*(s)$, we have
\begin{equation}
    \label{eq1002}
\int_0^{s^*} \bigl(e^{-s}-u^*(s)\bigr)\diff s
\geq  \int_0^{s^*} \bigl(\min\{T,e^{-s}\}-u^*(s)\bigr)\diff s
=\int_{t^*}^T \bigl(\log\frac 1 t -\mu(t)\bigr)\diff t.
\end{equation}
Observe that, given any universal constant $\tau\in (0,1)$,
in proving \eqref{estreinforced} we may assume
(if convenient) that 
\begin{equation}
\label{eq1004}
    T \geq \tau,  
\end{equation}
because otherwise
\eqref{estreinforced} would
 immediately follow
from
the first inequality in \eqref{estlemma1},
as soon as
$C\geq 2/(1-\tau)$.
In particular,
letting $T_0$ and $C_0$ be the constants provided by Lemma \ref{lemma:super-level-new}
when $t_0=\tau^*$, where $\tau^*$ is the constant obtained in Corollary \ref{cor:nondegen}, we may assume
that $T\geq T_0$, so that \eqref{newestmu}
reads
\begin{equation}
\label{eq1005}    
\mu(t)\leq (1+C_0(1-T))\log\frac T t\quad\forall t\in[\tau^*,T].
\end{equation}
Relying on \eqref{eq:nondegen}, we now use \eqref{eq1005}
to minorize the last integral
in \eqref{eq1002}.
More precisely, 
letting $\tau_1\in [\tau^*,1)$ denote a universal constant to be
chosen later, and further assuming
(in addition to $T\geq T_0$) that \eqref{eq1004} holds also with $\tau=\tau_1$,
from \eqref{eq:nondegen},   
\eqref{eq1005} and \eqref{eq1002} we find
\begin{equation}
    \label{eq1010}
\int_0^{s^*} \bigl(e^{-s}-u^*(s)\bigr)\diff s
\geq  \int_{\tau_1}^T \bigl(\log\frac 1 t -
(1+C_0(1-T))\log\frac T t\bigr)\diff t.
\end{equation}
Using $-\log T\geq 1-T$, for every $t\in (\tau_1,T)$ we have
\[
\log\frac 1 t - (1+C_0(1-T))\log\frac T t=-\log T
- C_0(1-T)\log\frac T t \geq 1-T- C_0(1-T)\log\frac 1{\tau_1},
\]
and choosing now $\tau_1\in[\tau^*,1)$ sufficiently close to $1$ in such a way that
\begin{equation}
    \label{epspos}
\eps_1:=1-C_0 \log\frac 1{\tau_1}>0,
\end{equation}
from \eqref{eq1010} and the subsequent estimate we obtain
\begin{equation}
    \label{eq1021}
\int_0^{s^*} \bigl(e^{-s}-u^*(s)\bigr)\diff s
\geq
\int_{\tau_1}^T (1-T)\left(1-C_0\log\frac 1 {\tau_1}\right)\diff t\geq \eps_1(1-T)(T-\tau_1).
\end{equation}
Finally, choosing a larger number $\tau_2\in (\tau_1,1)$ and further assuming that \eqref{eq1004}
holds also with $\tau=\tau_2$, we obtain
\[
\int_0^{s^*} \bigl(e^{-s}-u^*(s)\bigr)\diff s
\geq
\eps_1(\tau_2-\tau_1)(1-T)
\]
and \eqref{estreinforced} follows, by letting $C^{-1}=\eps_1(\tau_2-\tau_1)$.
\end{proof}

\donotshow{

We use the previous estimates in order to obtain a better version of Lemma \ref{lemma:QuantitativeMaxu} from the previous section. 

\begin{numlemma}\label{lemma:QuantitativeMaxuSharper}
    Let $F \in \mathcal{F}^2$ be such that $\|F\|_{\mathcal{F}^2} = 1.$
    There is an absolute, computable constant $C>0$ such that
    \[ 
    1- \max u \leq
            C \cdot \max\{e^{s_0}-1,1\} \delta_{s_0}(F).
    \] 
\end{numlemma}

\begin{proof}
    Before starting, note that we can and do assume that  $T \ls e^{-s_0}$, since in the complementary regime Lemma \ref{lemma:QuantitativeMaxu} already yields optimal bounds.

    \textsc{Step 1.} With the aid of Lemma \ref{lemma:super-level}, we find a suitable lower bound on $s^*$, as defined just before the proof of Lemma \ref{lemma:QuantitativeMaxu}. 
    
    Observe first that the function $L(t)$ defined in \eqref{eq:horizontal-L} increases as $t$ decreases, since $L'(t) \leq 0$ by \eqref{eqn:DiffInequ}. We now claim that this implies that either $u^*$ equals $e^{-s}$ \emph{everywhere}, or there is a \emph{unique} solution $s^*$ to the equation $u^*(s) = e^{-s}.$  Indeed, suppose $u^*(s) \not\equiv e^{-s}.$ If there are two $s_1^* < s_2^*$ such that $u^*(s_i^*) = e^{-s_i^*}, \, i=1,2,$ then, since $L' \le 0$ almost everywhere, we must have $u^*(s) = e^{-s}$ for each $s \in [s_1^*,s_2^*].$ But then this implies $L'(t) = 0$ for $t \in [e^{-s_2^*},e^{-s_1^*}].$ But the proof of \cite[Lemma~3.4]{NicolaTilli} shows that there is equality in $L'(t) \ge 0$ in a point where $\mu$ is differentiable if, and only if, the set $\{ u > t \}$ is a ball, and $|\nabla u|$ is constant on the boundary $\partial \{ u > t \} = \{ u = t \}.$ But this directly implies that $u(z) = e^{-\pi|z-z_0|^2},$ for some $z_0 \in \C,$ and hence $u^*(s) = e^{-s},$ for each $s \ge 0,$ which contradicts the assumption that they do not coincide everywhere. 

    By Lemma \ref{lemma:super-level} we have:
    \begin{equation}\label{eq:bound-level-for-s*}  
    \mu_F(t) \le  (1+C(1-T)) \cdot \log\left(\frac{T}{t}\right). 
    \end{equation} 
    Set then $t = e^{-s^*}$ in \eqref{eq:bound-level-for-s*} above. As $t = u^*(s^*),$ we obtain 
    \[
    s^* \le (1+C(1-T)) \cdot \log\left(\frac{T}{t}\right) = (1+C(1-T)) \cdot(\log T + s^*).
    \]
    Rearranging, one obtains 
    \begin{equation}
        \label{eq:lowerbounds*}
        C(1-T)\cdot s^* \ge (1+C(1-T))\log\left(\frac{1}{T}\right), 
    \end{equation}
    which implies that $s^* \ge c_0$ for some $c_0 > 0$ absolute and computable. 

    \textsc{Step 2.}  
    The purpose of this step is to use the bound from Lemma~\ref{lemma:super-level} to find a comparison function for $e^{-s}$ that can substitute $u^*$ and improves the control on $1-\max u$ found in Lemma~\ref{lemma:QuantitativeMaxu}.

    Notice that said estimate reads $s \leq (1+C(1-T)) \log ({T}/{u^*(s)})$, which may be rewritten as
    \[ u^*(s) \leq T e^{-\frac{s}{1+C(1-T)}} . \]
    Define the point $\tilde{s}$ such that $T e^{-\frac{\tilde{s}}{1+C(1-T)}} = e^{-\tilde{s}};$
    that is, 
    \[
    \tilde{s} = \frac{1+C(1-T)}{C(1-T)}\log\left(\frac{1}{T}\right).
    \]
    Observe in passing that, from \eqref{eq:lowerbounds*}, we have
    $$s^* \geq c_1 \geq \tilde{s} \geq c_0$$ whenever $T$ satisfies the conditions of Lemma \ref{lemma:super-level}.
    From the definition of $\tilde s$ we can compute
    \begin{align}\label{eq:lower-bound-integral}
    \begin{split}
    \int_0^{\tilde{s}} \left(e^{-s} - T \cdot e^{-\frac{s}{1+C(1-T)}}\right) \diff s 
    & = (1-T)(1+ C(e^{-\tilde{s}}-T)) \\
     & = (1-T)(1 + C \cdot T (e^{\frac{\log T}{C(1-T)}} - 1)). 
    \end{split}
    \end{align}
    In order to estimate the right-hand side, we analyze the function 
    \begin{equation}\label{eq:final-analysis-function} 
    \mathcal{R}(C,T) = C \cdot T (1-e^{\frac{\log T}{C(1-T)}} ).
    \end{equation}
    Notice that, for fixed $C>0,$ 
    $$\lim_{T\to 1} \mathcal{R}(C,T)= C(1-e^{-1/C})<1,\qquad T\mapsto \mathcal{R}(C,T) \text{ is decreasing}.$$
     Since $C$ is computable, absolute and finite, if we suppose that $1-T$ is small enough (depending only on $C$ from Lemma \ref{lemma:super-level}) then $\mathcal{R}(C,T) < 1-\rho_0,$ where $\rho_0 > 0$ is some absolute and computable constant. Thus, by \eqref{eq:lower-bound-integral},
    \begin{align}\label{eq:lower-bound-integral-tight} 
    \begin{split}
    \int_0^{\tilde s} \left(e^{-s} - T \cdot e^{-\frac{s}{1+C(1-T)}}\right)  \diff s
    & \ge (1-T) \left(1 - T \cdot \mathcal{R}(C,T)\right)\\
    & \ge (1-T)\left( (1-T) + \rho_0 T \right) > \frac{\rho_0}{2}(1-T).
    \end{split}
    \end{align}
    
    \textsc{Step 3}. To conclude the desired estimate, we divide the analysis depending on the relative positions between $s^*, \tilde s, s_0$; recall that $s^*\geq \tilde s.$ Before beginning the case analysis, recall also from the proof of Lemma \ref{lemma:QuantitativeMaxu}, and \eqref{eqn:InLemmaMaxuProof1} in particular, that the deficit $\delta_0$ satisfies the relation
    \begin{equation}
        \label{eq:auxdefdelta0}
        \int_0^{s_0} \left( e^{-s}-u^*(s) \right) \diff s = \delta_0 (1-e^{-s_0}).
    \end{equation}
    
    \textit{Case 1:} $s_0 \geq s^*\geq \tilde s$. 
Since $\mu_F(u^*(s))=s$, \eqref{eq:bound-level-for-s*} implies directly that $u^*(s)  \le T \cdot e^{-\frac{s}{1+C(1-T)}}$ and hence, by \eqref{eq:lower-bound-integral-tight},
    \begin{align*}\label{eq:lower-bound-integral}
    \begin{split} 
    \int_0^{s^*} \left(e^{-s} - u^*(s)\right) \diff s 
     \ge \int_0^{\tilde{s}} \left(e^{-s} - T \cdot e^{-\frac{s}{1+C(1-T)}}\right) \diff s
     \geq \frac {\rho_0}{2}  (1-T),
    \end{split}
    \end{align*}
and the conclusion follows from \eqref{eq:nonsharplemmaaux}, since
    $$\delta_0 (1-e^{-s_0}) e^{s_0-s^*} \geq \int_0^{s^*} \left(e^{-s} - u^*(s)\right) \diff s.$$

    \textit{Case 2: $s^*\geq s_0 \geq \tilde{s}$.} This case follows directly from 
    \eqref{eq:lower-bound-integral-tight} and \eqref{eq:auxdefdelta0}, since
    $$\int_0^{s_0} \left(e^{-s} - u^*(s)\right) \diff s \geq
     \int_0^{\tilde{s}} \left(e^{-s} - T \cdot e^{-\frac{s}{1+C(1-T)}}\right) \diff s
     \geq \frac {\rho_0}{2}  (1-T).$$
    
    
    \textit{Case 3: $ s^*\geq \tilde{s}\geq s_0$.} In this case, notice that if $s\in [0,\tilde s]$, then
    \[
    \partial_s(e^{-s} - T e^{-\frac{s}{1+C(1-T)}}) = -e^{-s} + \frac{T}{1+C(1-T)}e^{-\frac{s}{1+C(1-T)}} < 0.
    \]
    Hence, again by \eqref{eq:lower-bound-integral-tight},
    \begin{align*}\label{eq:final-bound-between-rearrangements} 
    \begin{split}
    (e^{s_0} - 1)\delta_0 
    & \geq \int_0^{s_0} \left(e^{-s} - u^*(s)\right) \diff s \\
    &\ge \frac{s_0}{\tilde{s}} \int_0^{\tilde{s}} \left(e^{-s} - T e^{-\frac{s}{1+C(1-T)}} \right) \diff s  \geq s_0 \frac{\rho_0}{c_1} (1-T).
     \end{split}
    \end{align*} 
     This finishes the proof, since $s_0 \mapsto \frac{e^{s_0} - 1}{s_0}$ is bounded from above and below by positive, absolute constants.  
\end{proof} 

} 

The final ingredient we need  is a well-known lemma, whose statement and proof are well-known in the theory of Reproducing Kernel Hilbert spaces. For completeness, we provide its proof here. 

\begin{numlemma}\label{lemma:kern}
 If $F\in\Fock$ and $\Vert F\Vert_{\Fock}=1$, then
 \begin{equation}
     \label{eq:kern}
     \min_{\substack{z_0\in\C \\ |c|=1}}
     \Vert F-cF_{z_0}\Vert_{\Fock}^2=2\left( 1-\sqrt T \,\right)
     \leq 2(1-T).
 \end{equation}
\end{numlemma}
\begin{proof} 
    Since $\Vert F_{z_0}\Vert_{\Fock}=1$ for every $z_0\in\C$,
    for any $c$ with $|c|=1$ we have
\begin{equation}\label{eq:Fock-Distance-reduce}
\Vert F-cF_{z_0}\Vert_{\Fock}^2 =
2- 2 \tRe \inner{\overline{c} F,  F_{z_0}}_{\mc F^2},
\end{equation}
and, since $\Fock(\C)$ is a reproducing kernel Hilbert space with kernel $K_{w}(z) = e^{\frac{\pi}{2} \abs{w}^2} F_{w}(z)$, we have $\inner{F,F_{z_0}}_{\mc F^2} = F(z_0) e^{- \frac{\pi}{2} \abs{z_0}^2}$.
Therefore, 
\[
\Vert F-cF_{z_0}\Vert_{\Fock}^2 =
2- 2 \tRe \overline{c} F(z_0) e^{- \frac{\pi}{2} \abs{z_0}^2},
\]
and choosing the unimodular $c$ that minimizes the last term, 
we obtain for every $z_0$
\[
\min_{|c|=1} \Vert F-cF_{z_0}\Vert_{\Fock}^2 =
2- 2  |F(z_0)| e^{- \frac{\pi}{2} \abs{z_0}^2}
=2-2 \sqrt{u(z_0)}.
\]
The equality in \eqref{eq:kern} then follows by minimizing over $z_0\in\C$, while the inequality is a direct consequence thereof.
\end{proof}

We are now ready to prove \eqref{eqn:StabilityFunctionFock}.

\begin{proof}[Proof of \eqref{eqn:StabilityFunctionFock}] 
By homogeneity, in \eqref{eqn:StabilityFunctionFock} one
can assume that $F\in\Fock$ and $\Vert F\Vert_{\Fock}=1$.
Then, 
given $\Omega$ as in Theorem \ref{thm:StabilityFockSpace}
and letting $s_0=|\Omega|$,
on combining \eqref{eq:kern} with \eqref{estreinforced}
and the second inequality in \eqref{estlemma1}, one finds
\begin{equation}
    \label{eq1008}
\min_{\substack{z_0\in\C \\ |c|=1}}
     \Vert F-cF_{z_0}\Vert_{\Fock}^2
     \leq C\delta_{s_0} e^{s_0},
\end{equation}
where $\delta_{s_0}$ is the deficit defined in \eqref{defdeltas0},
relative to the super-level set $\{u>u^*(s_0)\}$. But
\eqref{eqn:DefnI} (rewritten with $s=s_0$) reveals that
$\delta_{s_0}\leq \delta(F;\Omega)$,
where $\delta(F;\Omega)$ is the deficit relative to $\Omega$ as defined in \eqref{defdeltaFO}. Then 
\eqref{eqn:StabilityFunctionFock} follows
from \eqref{eq1008}, taking square roots.
\end{proof}

 \donotshow{

In following the structure of the original proof, a quantitative version of \cite[Proposition~2.1]{NicolaTilli} is now needed. Luckily, the structure of $\Fock(\C)$ as a reproducing kernel Hilbert space \cite{Grochenig}, with kernel $K_{\om}(z) = e^{\frac{\pi}{2} \abs{\om}^2} F_{\om}(z)$, provides an easy proof. Define
\[ \Delta_F(z) = \frac{\Fnorm{F}^2 - \abs{F(z)}^2 e^{- \pi \abs{z}^2}}{\Fnorm{F}^2} . \]
The next proposition allows us to conclude a non-sharp stability result from the previous lemma. 

\begin{numprop}\label{prop:QuantitativeFock}
    Let $F \in \Fock(\C)$. Then for every $z_0 \in \C$ there exists a constant $c \in \C$ with $\abs{c} = \Fnorm{F}$ such that
    \[ \frac{\Fnorm{F - c F_{z_0}}}{\Fnorm{F}} \leq \sqrt{2} \Delta_F(z_0)^{1/2} .\]
\end{numprop}
\begin{proof}
    We may assume that $\Fnorm{F} = 1$ and compute
    \begin{equation}\label{eqn:InProofQuantitativeProp}
    \begin{split}
        \Fnorm{F-cF_{z_0}}^2 &= \Fnorm{F}^2 + \abs{c}^2 \Fnorm{F_{z_0}}^2 - 2 \tRe \inner{F, c F_{z_0}}_{\mc F^2} \\
        &= 1 + \abs{c}^2 -2 \tRe \inner{F,cF_{z_0}}_{\mc F^2}  .
    \end{split}
    \end{equation}
    Since $\Fock(\C)$ is a reproducing kernel Hilbert space with kernel $K_{w}(z) = e^{\frac{\pi}{2} \abs{w}^2} F_{w}(z)$, we have $\inner{F,F_{z_0}}_{\mc F^2} = F(z_0) e^{- \frac{\pi}{2} \abs{z_0}^2}$. By choosing $c = F(z_0)/\abs{F(z_0)}$ we find that
    \[ \R \ni \inner{F , cF_{z_0}}_{\mc F^2} = \abs{F(z_0)} e^{- \frac \pi 2 \abs{z_0}}\geq 
    \abs{F(z_0)}^2 e^{- \pi \abs{z_0}^2} = 1- \Delta_F (z_0) ,\]
    since $0\leq 1-\Delta_F(z_0)\leq 1.$
    Plugging this in \eqref{eqn:InProofQuantitativeProp} gives
    \[ \Fnorm{F - cF_{z_0}}^2 = 2-2 \inner{F,cF_{z_0}}_{\mc F^2} \leq 2 \Delta_F (z_0) \]
    and the conclusion follows by taking the square-root.
\end{proof}

From Proposition \ref{prop:QuantitativeFock} above, we conclude directly, by using Lemma \ref{lemma:QuantitativeMaxu}, that, if $F$ satisfies the conditions of Theorem \ref{thm:StabilityFockSpace}, then there are $z_0 \in \C$ and $c \in \C, |c|=1$ with 
\[
\|F-c F_{z_0}\|_{\mathcal{F}^2} \le 8 (e^{|\Omega|} - 1)^{1/4} \delta_{|\Omega|}(F)^{1/4},
\]
whenever $\delta_{|\Omega|}(F) < 1/2.$ This readily proves a weaker version of Theorem \ref{thm:StabilityFockSpace}, with the exponent $1/2$ replaced by $1/4.$

}  

\section{The geometry of super-level sets}\label{subsec:GeometryOfSuperLevelSets}


In this section we study, for a fixed number $t>0$, some basic geometric properties of the super-level sets $\{z\in \C:u_F(z)>t\}$. In the proof of Lemma \ref{lemma:super-level-new} we saw that the function $g_\theta(r,\sigma)$, defined in \eqref{defg}, is monotone increasing in $r$ and, in particular, its sub-level sets are star-shaped. We will soon see that, by doing a finer analysis, we can prove a stronger version of this result, namely Proposition \ref{prop:levelsets}.

We begin by discussing some useful normalizations that we will use throughout the next sections. Let us first consider the quantity
$$\rho(F) : =\min_{z_0 \in \C, c \in \C} \frac{ \|F - c \cdot F_{z_0}\|_{\mc F^2}}{\|F\|_{\mc F^2}}.$$
Without loss of generality, we will assume that 
\[
\rho(F) =\min_{c\in \C} \frac{\| F - c \|_{\mc F^2}}{\|F\|_{\mc F^2}},
\]
that is, the closest function to $F$ in $\{c F_{z_0}\}_{z_0 \in \C, c\in \C}$ is a multiple of the constant function $F_0\equiv 1$. This follows by \eqref{eq:Fock-Distance-reduce}, since $\rho(F)^2 = \min_{z_0 \in \C} \frac{\|F\|_{\mc F^2}^2 - |F(z_0)|^2 e^{-\pi|z_0|^2}}{\|F\|_{\mc F^2}^2}.$
Moreover, we can also assume that
$$F(0) = 1.$$
Now, we note that, by the previous assumptions, we have 
\begin{align}
\label{eq:max}
\begin{split}
\|F - c \cdot F_{z_0}\|_{\mc F^2}^2 & = \|F\|_{\mc F^2}^2 + |c|^2 - 2 \Re(\overline{c} F(z_0))e^{-\pi|z_0|^2/2} \cr 
& \ge \|F\|_{\mc F^2}^2 + |c|^2 - 2|c||F(z_0)|e^{-\pi|z_0|^2/2} \ge \|F\|_{\mc F^2}^2 - \max_{z_0 \in \C} |F(z_0)|^2e^{-\pi|z_0|^2}.
\end{split}
\end{align} 
This shows that $\rho(F)$ is attained at $z_0 = 0$ if and only if $0$ is a maximum for $u_F$; hence, our normalization also implies
$$F'(0)=0.$$

Observe that $\rho$ differs slightly from the distance to the extremizing class used in \eqref{eqn:StabilityFunctionFock}, due to the condition on $c$. However, it is equivalent to this distance: indeed, by Lemma \ref{lemma:kern}, we have
\begin{equation}
    \label{eq:equivalentqtties}
    \rho(F) \leq \min_{z_0 \in \C, |c|=\|F\|_{\Fock(\C)}} \frac{ \|F - c \cdot F_{z_0}\|_{\mc F^2}}{\|F\|_{\mc F^2}} = \min_{|c|=\|F\|_{\Fock(\C)}} \frac{ \|F - c \|_{\mc F^2}}{\|F\|_{\mc F^2}} \leq \sqrt{2} \rho(F).
\end{equation}
Lemma \ref{lemma:kern} also shows that
\begin{equation}
    \label{eq:controllednorm1}
    2\frac{\|F\|_{\Fock}-1}{\|F\|_{\Fock}} =2\left(1-\frac{|F(0)|}{\|F\|_{\Fock}} \right)= \min_{|c|=\|F\|_{\Fock(\C)}} \frac{ \|F - c \|_{\mc F^2}^2}{\|F\|_{\mc F^2}^2}.
\end{equation}
Note that, by our normalizations, we have $F(0)=1\leq \|F\|_{\mc F^2}$. One the other hand, \eqref{eq:controllednorm1}, Lemma \ref{lemma:kern}, and \eqref{estlemma1} show that 
\begin{equation}
    \label{eq:controllednorm2}
    1\leq \|F\|_{\Fock}\leq \frac{2}{2-C(e^{|\Omega|}\delta(F;\Omega))^{1/4}}\leq 2,
\end{equation}
provided that the deficit is sufficiently small. 

In addition to $\rho(F)$, it will be convenient to consider the slightly different quantity
$$\e(F):= \|F-1\|_{\mc F^2}.$$
This allows us to write
\begin{equation}
    \label{eq:decompFG}
    F=1+\e G, \qquad \text{where } \|G\|_{\mc F^2} = 1 \text{ and } \e = \e(F),
\end{equation}
and then the above assumptions are  translated as 
\begin{equation}
    \label{eq:normG}
    \quad \langle G, 1\rangle_{\mc F^2} = \langle G, z\rangle_{\mc F^2} = 0.
\end{equation}
Note that, by \eqref{eq:controllednorm2}, we can assume that $\e$ is sufficiently small: indeed, if $e^{|\Omega|}\delta(F;\Omega)$ is sufficiently small, then
\begin{equation}
    \label{eq:estimateeps}
    \frac{\e(F)}{2} \leq \frac{\e(F)}{\|F\|_{\Fock}}=  \rho(F) \leq\min_{|c| = \|F\|_{\mc F^2}}\frac{\| F - c \|_{\mc F^2}}{\|F\|_{\Fock}}\leq C \Big(e^{|\Omega|}\delta(F;\Omega)\Big)^{1/4}.
\end{equation}

We are now ready to begin the main part of this section. We begin with a key technical lemma which shows that, above a certain threshold, all level sets of the function $u_F$ behave like those of the standard Gaussian, as long as $\e(F)$ is sufficiently small. 

\newcommand{\mcg}{\mathcal{G}}

\begin{numlemma}\label{lemma:Level-Set-Regularity} Let $F \in \mathcal{F}^2$ satisfy the normalizations in the beginning of this section. There are constants $\e_0, c_1 >0$ with the following property: if $\e(F)\leq \e_0$, then for any $\alpha \in [0,2\pi],$ the function 
\[
\mcg_{\alpha}(r) \coloneqq u_F(r e^{i\alpha}) = |F(r e^{i\alpha})|^2 e^{-\pi r^2}
\]
is strictly decreasing on the interval $\left[0,c_1 \sqrt{\log(1/\e(F))}\right].$
\end{numlemma}

\begin{proof} Without loss of generality we will take $\alpha=0$. In order to prove the desired assertion, we shall divided our analysis in two cases. 

\textsc{Case 1:} $1/10 < r < c_1 \sqrt{\log(1/\e)}.$ We differentiate the function $\mcg_0$ in terms of $r,$ which gives us 
\begin{align} 
\begin{split}
\mcg_0'(r) & = -2\pi r |F(r)|^2 e^{-\pi r^2} + 2\Re(F'(r) \overline{F(r)} ) e^{-\pi r^2} \cr 
        & = -2 \pi r(1 + 2 \e \Re(G(r)) + \e^2 |G(r)|^2) e^{-\pi r^2} + 2 \e \Re(G'(r)\overline{(1+\e G(r))}) e^{-\pi r^2},  
        \label{eq:g0'}
        \end{split}
\end{align} 
where in the last line we used \eqref{eq:decompFG}.
In order to bound the last term we note that, by the Cauchy integral formula, 
    \begin{equation}\label{eq:bound-derivative-G}
    |G'(w)| \le \frac{2}{|w|} \sup_{|z| = 2|w|} |G(z)| \leq 2 \frac{e^{2\pi |w|^2}}{|w|},
    \end{equation} 
    since $|G(z)|\leq e^{\pi |z|^2/2}$, and in addition  
    $$|1+\e G(r)|\leq \sqrt{2}e^{\frac \pi 2 r^2},$$ since $\|1+\e G\|_{\mc F^2}\le 2$.
Thus, as the $\e^2$-term in \eqref{eq:g0'} is negative, and $|\Re G(r)|e^{-\pi r^2}\leq e^{-\pi r^2}\leq 1$ as $r>1/10$, we can estimate
\[
\mcg_0'(r) \leq -2 \pi r(e^{-\pi r^2} - 2 \e) + \frac{8 \e}{r} e^{\frac{3\pi}{2}r^2}. 
\]
Since $ r < c_1 \sqrt{\log(1/\e)},$ we obtain that $e^{\pi r^2} \le e^{\pi c_1^2 \log(1/\e)} = \e^{-\pi c_1^2}$. For all $\e$ small enough, we have $e^{-\pi r^2}- 2 \e \geq \e^{\pi c_1^2}- 2 \e > 0$ provided that $\pi c_1^2 < 1$. Since also $1/10\leq r$, we have
\[
\mcg_0'(r) \leq -2 \pi r(\e^{\pi c_1^2}-2 \e) +  80 \e^{1-\frac{3\pi}{2} c_1^2}. 
\]
Hence, as long as 
\begin{equation}
    \label{eq:c1small}
    \frac{5\pi}{2} c_1^2<1,
\end{equation} 
the term $-2 \pi r \e^{\pi c_1^2}$ dominates over the others.  Thus, for sufficiently small $\e$,  we have $\mcg_0'(r) < 0.$\medskip 

\textsc{Case 2:} $0 < r \leq 1/10.$ Notice that this case is more subtle, as $\mcg_0'(r) \to 0$ when $r \to 0.$ We will show that the \emph{second derivative} of $\mcg_0$ is strictly negative for $r \in (0,1/10)$: thus the first derivative decreases in $(0,1)$ and, as $\mcg_0'(0) = 0$, it follows that $\mcg_0'(r) < 0$ in this interval, proving the claim. 

Starting from \eqref{eq:g0'}, we compute: 
\begin{align*}
\mcg_0''(r)  = & -2 \pi (1 + 2 \e \Re(G(r)) + \e^2 |G(r)|^2) e^{-\pi r^2} \cr 
        & - 4 \pi \e r (\Re(G'(r))+2 \e \Re(G'(r)\overline{G(r)}))e^{-\pi r^2} \cr 
        &+ 4\pi^2 r^2 (1 + 2 \e \Re(G(r)) + \e^2 |G(r)|^2) e^{-\pi r^2}  \cr 
        & + 2 \e \Re(G''(r)\overline{(1+\e G(r))}) e^{-\pi r^2} + 2\e^2 |G'(r)|^2e^{-\pi r^2}  \cr 
        & - 2 \pi r \e \Re(G'(r)\overline{(1+\e G(r))}) e^{-\pi r^2}. 
\end{align*}
We now follow the same strategy as in the first case. For $|w|\leq 1$, we find the estimates
\begin{align}\label{eq:bound-G-der-local} 
    |G'(w)|\leq  \frac{1}{2\pi} \max_{|z|=2} |G(z)|\leq 2 e^{2\pi},\qquad
    |G''(w)| \leq  \frac 1 \pi \max_{|z|=2} |G(z)| \leq 4 e^{2\pi}.
\end{align}
Therefore we have
\begin{align*} 
\mcg_0''(r) & \leq -2\pi(1-2\pi r^2)e^{-\pi r^2} + \e \mathfrak{h}(\e),
\end{align*} 
where $\mathfrak{h} \colon \R\to [0,\infty)$ is a smooth function. Since $r<1/10$, we have $1-2\pi r^2>0$ and so the first term above is negative. Hence, if $\e$ is sufficiently small, it holds that $\mcg_0''(r) < 0$ for all $r \in (0,1/10)$, and the conclusion follows. 
\end{proof}

In spite of its simple nature, we can derive several important conclusions from  Lemma \ref{lemma:Level-Set-Regularity}, such as the following result.

\begin{numlemma}\label{lemma:star-shaped-levels}
Under the same hypotheses of Lemma \ref{lemma:Level-Set-Regularity}, one may find a small constant $c_2> 0$ such that, for $t > \e(F)^{c_2},$ the level sets 
$$\{z \in \C \colon u_F(z) > t \}$$
are all star-shaped with respect to the origin. Moreover, for such $t,$ the boundary $\partial\{ u_F > t\} = \{ u_F = t \}$ is a smooth, closed curve. 
\end{numlemma} 

\begin{proof}
Let $c_1 >0$ be given by Lemma \ref{lemma:Level-Set-Regularity}. We first prove the following assertion: if $|z| > c_1 \sqrt{\log(1/\e)}$ then 
\begin{equation}
    \label{eq:boundsu}
u_F(z) < 4\e^{\pi c_1^2}.
\end{equation}
As before, we use the decomposition \eqref{eq:decompFG} to write 
\[ 
u_F(z) = (1 + 2 \e \Re(G(z)) + \e^2 |G(z)|^2)e^{-\pi \abs{z}^2}. 
\]
For $|z| > c_1 \sqrt{\log(1/\e)},$ and $\e$ sufficiently small, since $\|G\|_{\mc F^2}=1$ one readily sees that $$u_F(z) \le \e^{\pi c_1^2}(1 + 2 \e + \e^2) < 4 \e^{\pi c_1^2},$$
since we can choose $\pi c_1^2 \leq \frac 1 2,$ cf.\ \eqref{eq:c1small}.

We now claim that the conclusion of the lemma holds with $c_2 =\frac{\pi c_1^2}{2}.$ If this is not the case, there is $t_0 > \e^{c_2}$ such that $A_{t_0} \coloneqq \{z \in \C \colon u_F(z) > t_0 \}$ is \emph{not} star-shaped with respect to $0.$ Thus, there would be a point $w_0 \in A_{t_0},$ such that, for some $r \in (0,1),$ $r\cdot w_0 \not \in A_{t_0}.$ By \eqref{eq:boundsu}, we must have that 
\begin{equation}
    \label{eq:boundw0}
    |w_0| < c_1 \sqrt{\log(1/\e)};
\end{equation}
indeed, if $|w_0|>c_1 \sqrt{\log(1/\e)}$ then, by choosing $\e$ even smaller if need be, we would have
$u(w_0) < 4 \e^{\pi c_1^2} < \e^{\pi c_1^2/2}<t_0$,
contradicting the fact that $w_0\in A_{t_0}$. However, \eqref{eq:boundw0} leads to a contradiction already: if we write $e^{i \alpha_0} = \frac{w_0}{|w_0|}$ then Lemma \ref{lemma:Level-Set-Regularity} ensures that the function $s \mapsto |F(s e^{i \alpha_0})|^2 e^{-\pi s^2}$ is strictly decreasing for $s < c_1 \sqrt{\log(1/\e)}$ and thus we would have 
$$t_0 > u_F(r  w_0) > u_F(w_0) > t_0,$$
which is a contradiction. Hence $A_{t_0}$ is star-shaped with respect to the origin. 

The final claim of the lemma, concerning the smoothness of the boundary $\partial \{u_F > t\} = \{ u_F =t\}$, follows from the Inverse Function Theorem. Indeed, by \eqref{eq:boundsu} we see that if $z$ is such that $u_F(z) = t > \e^{c_2}$ then $|z| < c_1 \sqrt{\log(1/\e)},$ and Lemma \ref{lemma:Level-Set-Regularity} then guarantees that $\nabla u_F(z) \neq 0.$ Thus $t$ is a regular value of $u_F$ and the set $\{u_F=t\}$ is a smooth curve.
\end{proof}

Lemmata \ref{lemma:Level-Set-Regularity} and \ref{lemma:star-shaped-levels} already show that the super-level sets of $u_F$ are regular and have controlled geometry. We now show that they are in fact \emph{convex}:

\begin{numprop}\label{prop:convexity}
Under the same assumptions as in Lemma \ref{lemma:Level-Set-Regularity}, there are small constants $\e_0,c_3 > 0$ such that, as long as $\e(F)\leq \e_0$ and $s<-c_3 \log(\e(F)),$ the set 
\[
A_{u_F^*(s)} \coloneqq \{ z \in \C \colon u_F(z) > u_F^*(s) \}
\]
has convex closure. 
\end{numprop}

\begin{proof} Choosing $\e_0$ appropriately, we can apply Lemmas \ref{lemma:Level-Set-Regularity} and \ref{lemma:star-shaped-levels} to conclude that, for $t > \e(F)^{c_2},$ the level sets $\{z \in \C \colon u_F(z) > t\}$ are all star-shaped with respect to the origin and have smooth boundary. 

We write, for shortness, $u=u_F$ and $u_0=e^{-\pi |\cdot|^2}$ throughout the rest of this proof. By the triangle inequality, we have 
\begin{align}
\label{eqn:EstimateOnAbsDifference}
\begin{split}
    |u -u_0| & = \left| (|F|^2 -1) e^{-\pi |z|^2}\right| \\
    & \leq |F-1| (|F|+1) e^{-\pi |z|^2} \leq \e(F) (\|F\|_{\mc F^2} + \|1\|_{\mc F^2})= 2 \e(F)
\end{split}
\end{align}
and so
\begin{equation}\label{eq:first-set-comparison}         
\inset{u_0 \gs u^*(s) + 2\e(F)} \subset \inset{u \gs u^*(s)} \subset \inset{u_0 \gs u^*(s) - 2\e(F)}.
\end{equation} 
This implies that
\[
s = |\{u>u^*(s)\}| \geq |\{u_0 > u^*(s) + 2 \e(F)\}| = -\log(u^*(s) + 2 \e(F))
\]
or, rearranging, 
\begin{equation}
    \label{eq:estimateu*sfirst}
    u^*(s)\geq e^{-s} - 2 \e(F),
\end{equation}
In particular, if $e^{-s} >\e(F)^{c_3},$ then 
\begin{equation}
    \label{eq:estimateu*s}
    u^*(s) \geq \frac 12 \e(F)^{c_3} \geq \e(F)^{c_2}
\end{equation}
 provided  $c_3$ and $\e_0$ are chosen sufficiently small. Thus, for our choice of parameters, the set $A_{u^*(s)}$ is star-shaped and has a smooth boundary. 

Arguing similarly to Lemma \ref{lemma:Level-Set-Regularity} we see that, by further shrinking $c_3$ if needed, we have
\begin{equation}
    \label{eq:C2estimate}
    \|u-u_0\|_{C^2(A_{u^*(s)})} \le C_s \e(F),
\end{equation}
whenever $s \le - c_3 \log(\e(F)).$ Indeed, recalling again \eqref{eq:decompFG}, \eqref{eq:C2estimate} is equivalent to 
\[
\left\|\left(\Re(G) + \frac{\e}{2} |G|^2\right) \cdot u_0 \right\|_{C^2(A_{u^*(s)})} \le \frac{C_s}{2}.
\]
Using \eqref{eq:bound-derivative-G}, \eqref{eq:bound-G-der-local}, a suitable version of the first of those estimates for the second derivative, and \eqref{eqn:EstimateOnAbsDifference}--\eqref{eq:estimateu*s}, we see that 
\begin{equation}\label{eq:C-2-norm-difference} 
\left\|\left(\Re(G) + \frac{\e}{2} |G|^2\right) \cdot u_0 \right\|_{C^2(A_{u^*(s)})} \le C \sup_{w \in A_{u*(s)}} e^{10\pi |w|^2}. 
\end{equation} 
If $w\in A_{u^*(s)}$, by \eqref{eq:first-set-comparison} and similarly to \eqref{eq:estimateu*sfirst}, we have
$$e^{-\pi|w|^2} \ge u^*(s) - 2 \e(F) \geq \frac{e^{-s}}{2},$$
and hence \eqref{eq:C-2-norm-difference} implies \eqref{eq:C2estimate} with $C_s = C \cdot e^{4s},$ where $C$ is an absolute constant. 

Let then $\kappa_s$ denote the curvature of $\p A_{u^*(s)}= \{ u = u^*(s)\}$, thus
\begin{equation*} 
\kappa_s = -\frac{\nabla^2 u[\nabla u, \nabla u]}{|\nabla u|^3}.
\end{equation*} 
For $0<s<-c_3 \log(\e(F))$, by \eqref{eq:first-set-comparison} and \eqref{eq:estimateu*s} we have $\{u_0>\frac 1 2\e(F)^{c_3}\}\subset A_{u^*(s)}$, and hence 
\begin{equation}
    \label{eq:lowerbound}
    |\nabla u_0(z)|= 2 \pi |z| e^{-\pi |z|^2} \geq C \e(F)^{c_3} \quad \text{in } \p A_{u^*(s)}.
\end{equation}
 Let us denote by $\tilde \kappa_s>0$ the curvature of the circle $\{u_0 = u^*(s) \}$ and notice that, by \eqref{eq:estimateu*s}, $\tilde \kappa_s\to \infty$ as $s\to 0$. By \eqref{eq:C2estimate} and \eqref{eq:lowerbound}, choosing $c_3$ and $\e_0$ sufficiently small, we have an estimate
$$\left|\kappa_s- \tilde \kappa_s\right| \leq C C_s \frac{\e(F)}{\e(F)^{c_3}} \leq \e(F)^{\frac 1 2},$$
where we used the bound on $s$ in the last inequality. Combining the last two facts, we see that we can first choose $s_0$ small enough so that
$\e_0^{1/4} \leq \tilde \kappa_s \text{ if } s\leq s_0,$
and then choose $\e_0$ even smaller so that we have both $\e_0^{1/4}\leq \tilde \kappa_s$ for $s_0<s<\frac 1 4$ and $\e_0^{1/2}\leq \frac 1 2 \e_0^{1/4}.$ These choices ensure that
$$\frac{\e_0^{1/4} }{2} \leq \tilde \kappa_s - \e_0^{1/2}\leq \kappa_s$$
for all $s<-c_3 \log \e(F)$. This lower bound implies that $A_{u^*(s)}$ is locally convex. We then use the well-known Tietze--Nakajima theorem (see \cite{Tietze, Nakajima}) which asserts that, as $\overline{A_{u^*(s)}}$ is a closed, connected set, its local convexity implies its convexity, and the assertion is proved. 
\end{proof} 

\begin{proof}[Proof of Proposition \ref{prop:levelsets}]
Proposition \ref{prop:levelsets} follows immediately from Proposition \ref{prop:convexity} and \eqref{eq:estimateeps}, taking $\Omega=A_{u_F^*(s)}$ as usual.
\end{proof}

\section{Proof of the set stability}\label{sec:set-stability}

In this section we complete the proof of our main Theorem \ref{thm:Stability}. As explained in the introduction, it suffices to prove its Fock space analogue,  Theorem \ref{thm:StabilityFockSpace}.

\newcommand{\mcT}{\mathcal{T}}

\begin{proof}[Proof of Theorems \ref{thm:Stability} and \ref{thm:StabilityFockSpace}]
Since the stability for the function has already been proved in Section \ref{sec:firstproof}, it remains to prove stability of the set, i.e.\ estimate \eqref{eqn:StabilitySet}.

Fix  $f \in L^2$ as in the statement of Theorem \ref{thm:Stability}, let $F = \mathcal{B}f$, $u_F(z) = |F(z)|^2 e^{-\pi \abs{z}^2}$ and let us write $\delta = \delta(F;\Om)$ for simplicity. Clearly we may assume that $\delta\leq \delta_0$, for some arbitrarily small constant $\delta_0$. We may also suppose that $F$ is normalized as at the beginning of Section \ref{subsec:GeometryOfSuperLevelSets} and so, as in \eqref{eq:decompFG}, we can write $F=1+\e G,$ where $\|G\|_{\Fock}=1$ satisfies \eqref{eq:normG} and $\e$ satisfies \eqref{eq:estimateeps}.

Let $A_{\Omega} \coloneqq A_{u_F^*(|\Om|)}$, as in \eqref{eqn:DefnSuperLevelSetsOfu}. Let $\mcT$ be any transport map $\mcT \colon A_{\Om} \setminus \Om \to \Om \setminus A_{\Om}$, that is,
$$1_{\Om \setminus A_{\Om}}(\mcT (x)) \det \nabla \mcT (x) = 1_{A_{\Om}\setminus \Om }(x),$$
cf.\ \cite[page 12]{Figalli2021} for details on the existence of such a map. Define 
\[
B \coloneqq \inset{ x \in A_{\Om} \setminus \Om \colon |\mcT(x)|^2 -|x|^2 > C_{|\Om|} \gamma },
\]
where $C_{|\Om|}, \gamma$ are constants to be chosen later. Since $\mcT$ is a transport map,
\begin{equation}\label{eq:difference-transp}  
\int_B \left(u(z) - u(\mcT(z)) \right) \diff z = \int_B u - \int_{\mcT(B)} u \le \int_{A_\Om} u - \int_\Om u =: d(\Omega). 
\end{equation}
In \eqref{eq:difference-transp}, the inequality holds by the fact that, for $z \in A_\Om \setminus \Om, \quad u_F(z) > u_F^*(|\Omega|),$ and the reverse inequality holds for $z \in \Om \setminus \Om.$ Note that from \eqref{eq:boundI} we have the bound
\begin{equation}
    \label{eq:boundd}
    d(\Omega) \leq \|F\|_{\Fock}^2(1-e^{-|\Omega|})- \int_\Omega u = \|F\|_{\Fock}^2 (1-e^{-|\Omega|}) \delta \leq 2 (1-e^{-|\Omega|})\delta,
\end{equation}
by \eqref{eq:controllednorm2} and the assumption that $\delta$ is sufficiently small.

\textsc{Step I}. \textit{Control over $B.$} In this step, we will show that
\begin{equation}
    \label{eq:lower-bound-u}
    u(z)-u(\mcT(z))\geq 5 \gamma \quad \text{for } z\in B,
\end{equation}
after choosing $C_{|\Omega|}$ and $\gamma$ correctly. To see this, we begin by writing
\begin{align*}
u(z) - u(\mcT(z)) & = e^{-\pi |z|^2} - e^{-\pi |\mcT(z)|^2} \cr 
 &  + 2\e \left( \Re(G(z)e^{-\pi |z|^2}) - \Re(G(\mcT(z))e^{-\pi |\mcT(z)|^2}) \right) \cr 
   &  + \e^2 \left(|G(z)|^2 e^{-\pi |z|^2} - |G(\mcT(z))|^2 e^{-\pi |\mcT(z)|^2} \right). 
\end{align*} 
Since $|G|^2e^{-\pi|\cdot|^2} \le 1,$  we have promptly 
\begin{align}
\label{eq:lowerboundutgeneral}
\begin{split}
u(z) - u(\mcT(z)) & \ge e^{- \pi |z|^2 } - e^{-\pi |\mcT(z)|^2} - (4\e + \e^2) \cr 
                & = e^{-\pi |z|^2} \left( 1 - e^{-\pi\left(|\mcT(z)|^2 - |z|^2\right)} \right) - 4\e - \e^2, 
\end{split}                
\end{align} 
whenever $z \in B.$ If $\pi (|\mcT(z)|^2-|z|^2) \geq 1$ then we find
$$u(z)- u(\mcT(z)) \geq \frac{e^{-|\Omega|}}{2} (1- e^{-1}) - 4 \e - \e^2 \geq \frac{e^{-|\Omega|}}{4} - 4 \e - \e^2.$$
On the other hand, if $\pi (|\mcT(z)|^2-|z|^2) \leq 1$, we argue as follows: since $z \in B\subset A_{\Om},$ \eqref{eqn:EstimateOnAbsDifference} shows that 
$$e^{-\pi|z|^2} \ge u(z) - 2 \e \ge u^*(|\Om|) -  2 \e \ge e^{-|\Om|} - 4 \e > \frac{e^{-|\Om|}}{2};$$
here we used also 
\begin{equation}
    \label{eq:boundu*stabilityset}
    e^{-|\Om|} - 2 \e \leq u^*(|\Omega|)\leq e^{-|\Omega|} + 2 \e,
\end{equation}
cf.\ \eqref{eq:first-set-comparison} and \eqref{eq:estimateu*sfirst}.
Thus, from \eqref{eq:lowerboundutgeneral},
\begin{equation*} 
    \label{eq:lower-bound-u-T}
u(z) - u(\mcT(z)) \ge \frac{e^{-|\Om|} \left(|\mcT(z)|^2 - |z|^2\right)}{2} - 4 \e - \e^2  \ge C_{|\Om|}e^{-|\Om|} \frac{\gamma}{2} - 4\e - \e^2.
\end{equation*}
Choosing $C_{|\Om|} = 20 e^{|\Om|}$ and $\gamma \ge \e$, the previous estimates yield the desired \eqref{eq:lower-bound-u}.

\textsc{Step II.}\textit{Showing that $\Omega$ is close to $A_\Om.$} Note the identities
$$|\Omega|-|B|= |\Omega|- |\mcT(B)| = |\Omega\setminus \mcT(B)| = |\Omega|-|A_{\Omega}\setminus \Om| + |(\Om \setminus \mcT(B))\setminus A_\Omega|,$$
hence
\begin{equation}
    \label{eq:symdiffidentity}
    \frac 12|\Om\Delta A_\Om| = |A_\Om \setminus \Om| = |B|+ |(\Om\setminus \mcT(B))\setminus A_\Omega|.
\end{equation}
In this step, we want to estimate both terms on the right-hand side. The estimate for the first term follows by combining \eqref{eq:difference-transp}, \eqref{eq:boundd} and  \eqref{eq:lower-bound-u}: 
\begin{equation}
    \label{eq:estimateB}
    |B| \le \frac{d(\Om)}{5\gamma} \le \frac{2\delta(1-e^{-|\Om|})}{5 \gamma}.
\end{equation}
To estimate the second term, note that $\Om \setminus \mcT(B)$ is contained in a $C_{|\Omega|} \gamma$-neighborhood of $A_\Omega$; in turn, by  \eqref{eq:first-set-comparison}, $A_\Omega$ is nested between two concentric balls:
\begin{equation}
    \label{eq:nestedstabilityset}
    \{z:e^{-\pi |z|^2} > u^*(|\Omega|)+ 2 \e\}\subset A_\Om \subset \{z:e^{-\pi |z|^2} > u^*(|\Omega|)- 2 \e\}=: E_\Omega.
\end{equation}
Combining this information with \eqref{eq:boundu*stabilityset}, we can estimate
\begin{align}
    \label{eq:estimatesmallset}
    \begin{split}
    |(\Om \setminus \mcT(B))\setminus A_\Omega| & \leq 4 C_{|\Om|}\gamma \sqrt{-\log(e^{-|\Om|} - 4 \e)} + C_{|\Om|}^2 \gamma^2 \\
    & \leq 4 C_{|\Om|} \gamma \sqrt{|\Om| + 8 \e e^{|\Om|}} + C_{|\Om|}^2 \gamma^2 \\
    & \leq 4 C_{|\Om|} \gamma \left(|\Om|^{1/2} + 4 \e \frac{e^{|\Om|}}{|\Om|^{1/2}}\right) + C_{|\Om|}^2 \gamma^2  \\
    & \leq 4 C_{|\Om|} \gamma \left(|\Om|^{1/2} + 4 C\delta^{1/2} \frac{e^{2|\Om|}}{|\Om|^{1/2}}\right) + C_{|\Om|}^2 \gamma^2, 
    \end{split}
\end{align}
provided that $\e$ is sufficiently small, depending on $|\Om|.$ Choosing $\e\leq \gamma= C (e^{|\Omega|}\delta)^{1/2}$, where $C$ is the constant provided by Theorem \ref{thm:StabilityFockSpace}, and combining \eqref{eq:symdiffidentity}, \eqref{eq:estimateB} and \eqref{eq:estimatesmallset}, we get
$$|\Om \Delta A_\Omega| \leq C\delta^{1/2},$$
for some new but still explicitly computable  constant $C=C(|\Om|)$.

\textsc{Step III}. \textit{Conclusion.} To conclude, we just need to compare $\Omega$ with the ball $S_\Omega:=\{z:e^{-\pi |z|^2}\geq e^{-|\Omega|}\}$. 
%
By \eqref{eq:boundu*stabilityset} and \eqref{eq:nestedstabilityset}, we have $S_\Om \subset E_\Om$ and 
\begin{equation}
    \label{eq:estimateEomSom}
    |E_\Om \setminus S_\Om| \le C \e \leq C(|\Om|) \delta^{\frac 1 2},
\end{equation}
where we also used \eqref{eq:estimateeps}. It follows that
\begin{align*}
    |S_\Omega\Delta \Omega| &\leq |\Omega\setminus E_\Omega| + |E_\Omega\setminus S_\Omega|+|S_\Omega\setminus \Omega|\\
    & \leq |\Omega\setminus E_\Omega| + |E_\Omega\setminus S_\Omega|+|E_\Omega\setminus \Omega| \leq |E_\Omega\Delta\Omega| + C\delta^{1/2}
\end{align*}
and so it is enough to bound $|E_\Om \triangle \Om|.$ We then estimate
\begin{align*}
    |E_\Om \triangle \Om| = |E_\Om \setminus \Om | + |\Om \setminus E_\Om| 
     \leq  |E_\Om \setminus A_\Om| + |A_\Om \setminus \Om| + |\Om \setminus A_\Om| 
     \leq C \delta^{1/2},
\end{align*}
where in the last inequality we estimate $|E_\Omega\setminus A_\Omega|$ as in \eqref{eq:estimateEomSom} and we also used the estimate from the last step.
We have now proved \eqref{eqn:StabilitySet2} and thus also \eqref{eqn:StabilitySet}.
\end{proof} 

We remark that, in spite of the sharp exponent of $\delta$ in the result above, the asymptotic growth of the constant $K(|\Om|)$ in \eqref{eqn:StabilitySet} from the proof above is likely not sharp: as we shall see in Section \ref{sec:sharpness}, one expects, from the functional stability part, that the sharp growth of the constant should be of the form $\sim e^{|\Omega|/2},$ while the proof above yields $K(|\Om|) \sim e^{2|\Omega|}.$ 

Although there is room for improving such a constant with the current methods, it is unlikely that these will suffice in order to upgrade $K(|\Om|)$ to the aforementioned conjectured optimal growth rate. For that reason, we consider this to be a genuinely interesting problem, which we wish to revisit in a future work. 

\section{An alternative variational approach to the function stability}
\label{sec:alternative-proof}


The purpose of this section is to give a variational proof of Theorem \ref{thm:StabilityFockSpace}. Fix $s>0$ and consider the functional
$$\K\colon \Fock(\C)\to \R, \qquad \K[F] := \frac{I_F(s)}{\|F\|_{\mathcal F^2}^2},$$
where we recall that $I_F(s)$ is the integral of $u_F$ over its superlevel set of measure $s$, cf.\ \eqref{eqn:DefnI}. We will prove the following result:

\begin{numthm}\label{thm:optimalperturbative}
Fix $s\in (0,\infty)$. There are explicit constants $\e_0(s),C(s)>0$ such that, for all $\e \in (0,\e_0),$ we have
$$\K[1] - \K[1+ \e G]  \geq C(s)\e^2, $$
whenever $\|G\|_{\mathcal F^2}=1$ satisfies \eqref{eq:normG}.    
\end{numthm}

The proof of Theorem \ref{thm:optimalperturbative} is almost independent of the results of Section \ref{sec:firstproof}, as we will only rely on the suboptimal stability result from Lemma \ref{lemma:QuantitativeMaxuNew}. This lemma, in turn, does not rely on the other results from that section.

Let us first note that Theorem \ref{thm:optimalperturbative} indeed implies the function stability part Theorem \ref{thm:StabilityFockSpace}, although without the optimal dependence of the constant on $|\Omega|$.

\begin{proof}[Alternative proof of \eqref{eqn:StabilityFunctionFock}, assuming Theorem \ref{thm:optimalperturbative}] Without loss of generality, we can assume the normalizations detailed at the beginning of Section \ref{subsec:GeometryOfSuperLevelSets}. By the same argument as in \eqref{eq:estimateeps}, if the deficit is sufficiently small we see that
$$\|F-1\|_{\Fock(\C)} = \e(F)= \|F\|_{\Fock(\C)} \rho(F)\leq C  \big(e^{|\Omega|}\delta(F;\Omega)\big)^{1/4}.$$
where now the last inequality follows by combining Lemma \ref{lemma:QuantitativeMaxuNew} with the simple Lemma \ref{lemma:kern}, instead of using \eqref{eqn:StabilityFunctionFock}. Here, we take $\Omega=A_{u_F^*(t)}=\{u_F>u_F^*(t)\}.$ Hence, we can write 
$$F = 1 + \e  G,\qquad \|G\|_{\mc F^2} = 1,$$ 
where $G$ satisfies \eqref{eq:normG},
and we can assume that $\e$ is sufficiently small. Theorem \ref{thm:optimalperturbative} then implies that 
\[
(1-e^{-\vol{\Om}})\delta(F;\Omega) = \mathcal{K}[1] - \mathcal{K}[F] \ge C(|\Omega|) \e^2 = C(|\Omega|) \|F-1\|_{\mc F^2}^2.
\]
To complete the proof it suffices to note that, by our normalizations, $F(0)=1\leq \|F\|_{\mc F^2}$. Thus
$$\|F-1\|_{\Fock} \geq \rho(F) \geq 2^{-1/2} \min_{z_0 \in \C, |c|=\|F\|_{\Fock(\C)}} \frac{ \|F - c \cdot F_{z_0}\|_{\mc F^2}}{\|F\|_{\mc F^2}},$$
where the last inequality follows from \eqref{eq:equivalentqtties}.
\end{proof}

The proof of Theorem \ref{thm:optimalperturbative} is based on the following technical result:

\begin{numlemma}\label{lemma:control3rdder}
There is $\e_0=\e_0(s)$ and a modulus of continuity $\eta$, depending only on $s$, such that
    $$|\K[1+ \e G] - \K[1]| \le  \left| \frac{\e^2}{2}\nabla^2 \K[1](G,G)\right| + \eta(\e)\e^2 $$
    for all $0\leq \e \leq \e_0(t)$ and $G\in \mc F^2(\C)$ such that $\|G\|_{\mc F^2}=1$ and which satisfy \eqref{eq:normG}. Here we have defined
    \begin{equation*}\label{eq:definition-second-variation-K}
\nabla^2 \mathcal{K}[1](G,G) \coloneqq \frac{\diff^{\,2}} 
{\diff \e^{\,2}}  \mathcal{K}[1+\e G]\Big|_{\e = 0}.
\end{equation*}
\end{numlemma}

The proof of Lemma \ref{lemma:control3rdder} is rather technical and standard, for which reason we moved it to Appendix \ref{sec:appendix}. Lemma \ref{lemma:control3rdder} shows that  $\mc K[1+\e G]-\mc K[1]$ is essentially controlled by the second variation of $\mc K$ at 1, in the direction of $G$. Since $1$ is a local maximum for $\mc K$, this variation is negative definite, but to prove Theorem \ref{thm:optimalperturbative} we need to show that it is \textit{uniformly} negative definite. This is the content of the next proposition, which is is the main result of this section.


\begin{numprop}\label{prop:negdef}
For all $G\in \mc F^2(\C)$ such that $\|G\|_{\mc F^2}=1$ and which satisfy \eqref{eq:normG}, we have
    $$\frac 1 2\nabla^2 \K[1](G,G)\leq -s  e^{-s}.$$
\end{numprop}

It is clear that Theorem \ref{thm:optimalperturbative} is an immediate consequence of the above two results:

\begin{proof}[Proof of Theorem \ref{thm:optimalperturbative}]
    Combining Lemma \ref{lemma:control3rdder} and Proposition \ref{prop:negdef}, we have
    $$\K[1]-\K[1+\e G] \geq - \e^2\Big(\frac 1 2 \nabla^2 \K[1](G,G) + \eta(\e)\Big)
    \geq \e^2\Big(\frac{C(s)}{2}-\eta(\e)\Big).
    $$
    The conclusion now follows by choosing $\e_0=\e_0(s)$ even smaller so that $\frac{C(s)}{4}\geq \eta(\e_0).$
\end{proof}

The rest of this section is dedicated to the proof of Proposition \ref{prop:negdef}. Clearly we first need to compute the second variation of $\mc K$ and, in order to do so, our strategy is to consider the sets
\begin{equation}
    \label{eq:defOmegae}
    \Omega_\e := \{u_\e>u_\e^*(s)\}, \qquad u_\e :=u_{1+\e G} = |1+\e G|^2 e^{-\pi |\cdot|^2},
\end{equation}
and to write 
$$\Omega_\e = \Phi_\e(\Omega_0),$$
for a suitable volume-preserving flow $\Phi_\e$. In order to construct such a flow, we first prove a general lemma which allows us to build a flow that deforms the unit disk into a given family of graphical domains over the unit circle. This type of result is well-known, and we refer the reader for instance to \cite[Theorem~3.7]{Acerbietal} for a more general statement. 

\begin{numlemma}\label{lemma:flows} 
Denote by $D_0 \subset \R^2$ the unit disk, and suppose that we are given a one-parameter family $\{D_\e\}_{\e \in [0,\e_0]}$ of domains, whose boundaries are given by smooth graphs over the unit circle:
$$\p D_\e = \{(1+g_\e(\omega))\omega:\omega\in \mb S^1\}.$$
Suppose, additionally, that the family $\{g_\e\}_{\e \in [0,\e_0]}$ depends smoothly on $\e.$ 

Then there exists a family $\{Y_\e\}_{\e \in [0,\e_0]}$ of smooth vector fields, which depends smoothly on the parameter $\e,$ such that, if $\Psi_\e$ denotes the flow associated with $Y_\e$, i.e. if
$$\frac{\diff}{\diff \e} \Psi_\e = Y_\e(\Psi_\e),$$
 then $\Psi_\e(D_0) = D_\e.$ In addition, $Y_\e$ is such that $\ddiv(Y_\e) = 0$ in a neighbourhood of $\mb S^1.$
\end{numlemma}

\begin{proof} 
By translating into polar coordinates $r=|z|$ and $ \omega= z/|z|$ we see that, if we define a vector field $Y_\e$ locally on a neighbourhood of $\mb S^1$ by 
$$Y_\e(r,\omega) = \frac 1 r (1+g_\e(\omega))\p_\e g_\e(\omega) \omega,$$
then $Y_\e$ satisfies $\ddiv(Y_\e) = 0$ in a neighbourhood of $\mb S^1,$  Moreover, in the same neighbourhood of $\mb S^1,$ we may write the flow $\Psi_\e$ of $Y_\e$ explicitly as 
\begin{equation}
    \label{eq:defPsi}
    \Psi_\e(r,\omega) = (r^2 + (1+g_\e(\omega))^2-1)^{\frac 1 2} \omega.
\end{equation}
We then extend $\Psi_\e$ from the neighbourhood of $\mb S^1$ to the whole complex plane, in such a way that $\Psi_\e(D_0) = D_\e$ and the map $(\e,x) \mapsto \Psi_\e(x)$ is smooth. Taking $\tilde{Y}_\e$ to be the flow of the extended version of $\Psi_\e,$ we see that $\tilde{Y}_\e$ is an extension of $Y_\e$ to the whole space, and moreover, the map $(\e,x) \mapsto Y_\e(x)$ is smooth. 
\end{proof} 

Using the results of Section \ref{subsec:GeometryOfSuperLevelSets} we can readily apply Lemma \ref{lemma:flows} to the sets $\Omega_\e$:

\begin{numlemma}
    \label{lemma:flows-specific}
Let $G\in \mc F^2(\C)$ satisfy \eqref{eq:normG}. 
There is $\e_0 =\e_0(s,\|G\|_{\mc F^2})> 0$ such that, for all $\e\in [0,\e_0],$ there are globally defined smooth vector fields  $X_\e$, with associated flows $\Phi_\e$, such that 
$$\Omega_\e = \Phi_\e(\Omega_0).$$ 
Moreover, $X_\e$ depends smoothly on $\e$ and is divergence-free in a neighborhood of $\partial \Omega_0$. We also have
\begin{equation}
\label{eq:zeromeanXeps}
\int_{\partial \Omega_\e} \langle X_{\varepsilon} , \nu_\e \rangle = 0,
\end{equation}
where $\nu_\e$ denotes the outward-pointing unit vector field on $\p \Omega_\e.$
\end{numlemma}

\begin{proof}
Up to dilating by a constant (which depends only on $s$) we can assume that $\Omega_0=B_1$. Lemma \ref{lemma:star-shaped-levels} shows that, if $\e_0$ is chosen sufficiently small, the boundaries $\p\Omega_\e$ are smooth and the sets $\Omega_\e$ are star-shaped with respect to zero, hence they can be written as graphs over $\mb S^1$:
\begin{equation}
    \label{eq:deffeps}
    \p \Omega_\e = \{ (1+f_\e(\omega)) \, \omega \colon \omega \in \mb S^1\}.
\end{equation}
We now claim that the function $(\e,\omega) \mapsto f_\e(\omega)$ is \emph{smooth} as long as $\e$ is sufficiently small. 

Indeed, for fixed $\e,$ the function $\omega \mapsto f_{\e}(\omega)$ is smooth, by Lemma \ref{lemma:star-shaped-levels}, since it is implicitly defined by $u_{\e}((1+f_\e(\omega))\cdot \omega) = u_\e^*(s).$ Moreover, since $\nabla u_{\e}$ is \emph{bounded} by a constant depending only on $s$ when restricted to $\{u_\e = u_\e^*(s)\}$ (this follows, for instance, from the proof of Lemma \ref{lemma:Level-Set-Regularity}), any careful quantification of the proof of the implicit function theorem (cf. \cite{Krantz}) implies that there is a universal $\e_0(s) > 0$ such that, if $\e < \e_0(s),$ then $\e \mapsto f_\e(\omega)$ is smooth for any fixed $\omega \in \mathbb{S}^1.$ This proves the desired smoothness claim.


By Lemma \ref{lemma:flows}, the associated vector fields are explicitly given in a neighbourhood of $\mb S^1$ by
\begin{equation}
    \label{eq:defXeps}
    X_\e(r,\omega) = \frac 1 r (1+f_\e(\omega))\p_\e f_\e(\omega) \omega,
\end{equation}
and they are divergence-free in a neighbourhood of $\mb S^1.$ Their smoothness then follows from the smoothness of $f_\e$ in $\e$.     

To prove the final claim we note that, since $\Omega_{\varepsilon}$ has constant measure equal to $s$ for all $\varepsilon\in [0,\e_0]$,  by a calculation in polar coordinates we see that the function
$$ \mathcal{A}(\varepsilon) :=  \int_{\partial \Omega_0} (1+f_{\varepsilon}(\omega))^2 \diff \mathcal{H}^1(\omega)$$
is constant in the interval $[0,\e_0]$. Thus,
\[
0 = \frac{\diff \mc A(\e)}{\diff \e} = 2 \int_{\partial \Omega_0} \partial_{\varepsilon} f_{\varepsilon}(\omega) (1 + f_{\varepsilon}(\omega) ) \diff \mathcal{H}^1(\omega)
= 2 \int_{\partial \Omega_0} \langle X_{\varepsilon}, \nu \rangle \diff \mathcal{H}^1(\omega),
\]
and so the integral above has to vanish; here we used the fact that $\Omega_0$ is a ball. Since $\ddiv(X_{\e}) = 0$ in a neighbourhood of $\partial \Omega_0,$ and as $\int_{\partial \Omega_0} \langle X_{\e}, \nu \rangle \diff \mathcal{H}^1= 0,$ the divergence theorem shows that, for any Lipschitz Jordan curve $\gamma$ in the same neighbourhood of $\partial \Omega_0,$ we have 
\begin{equation*}\label{eq:almost-full-div-free} 
\int_{\gamma} \langle X_{\e}, \nu_{\gamma} \rangle \diff \mc H^1 = 0 ,
\end{equation*}
where $\nu_{\gamma}$ denotes the outward-pointing normal field on $\gamma.$ Thus \eqref{eq:zeromeanXeps} follows.
\end{proof}

Having the previous lemma at our disposal, we can now obtain an explicit formula for $\nabla^2 \K[1]$.

\begin{numlemma}\label{lemma:second-variation-K}
For all $G\in \Fock(\C)$ which satisfy \eqref{eq:normG}, we have
    \begin{equation}\label{eq:second-variation} 
 \frac 1 2\nabla^2 \K[1](G,G)
 = \int_{\Omega_0} |G|^2 e^{-\pi |z|^2}\diff z - \|G\|_{\mc F^2}^2\int_{\Omega_0} e^{-\pi |z|^2}\diff z + e^{-s} \fint_{\p \Omega_0} |G|^2 \diff \mc H^1(z).
\end{equation}
\end{numlemma}


\begin{proof}
Setting $\diff \sigma(z) := e^{-\pi |z|^2} \diff z$ for brevity, let us introduce the auxiliary functions
\begin{equation}
    \label{eq:defIJ}
    I_\e := \int_{\Omega_\e} |1+\e G|^2 \diff \sigma, \qquad J_{\e}:=\int_\C |1+\e G|^2 \diff \sigma;
\end{equation}
we also write $K_\e \coloneqq \K[1+\e G]$, where we recall  that $\Omega_\e \coloneqq \{u_\e>u_\e^*(s)\}$, cf.\ \eqref{eq:defOmegae}. We will always take $\e\leq \e_0$, where $\e_0$ is as in Lemma \ref{lemma:flows-specific}.
Here and henceforth, we shall denote derivatives of the quantities $K_\e, \quad I_\e, \quad J_\e$ in the $\e$ variable with primes, that is, $K_\e',\quad J_\e',\quad I_\e',$ etc. With that in mind, we have:
\begin{align}
\begin{split}
\label{eq:derK}
K_\e' & = \Big(I'_\e - I_\e \frac{J'_\e}{J_\e}\Big) \frac{1}{J_\e},\\
K_\e'' &= \Big(I''_\e - I_\e \frac{J''_\e}{J_\e}\Big) \frac{1}{J_\e}- \frac{2 J'_\e}{J_\e^2} \Big(I_\e'-I_\e \frac{J'_\e}{J_\e}\Big),
\end{split}
\end{align}
and using Reynold's theorem we further compute
\begin{align}
\label{eq:derIJ}
I'_\e & = 2 \int_{\Omega_\e} \text{Re}(G \cdot \overline{1+\e G}) \diff \sigma,  
& J'_\e  = 2\int_\C \text{Re} (G \cdot \overline{1+ \e G}) \diff \sigma,\\
\label{eq:derIJsecond}
I''_\e & = 2 \int_{\Omega_\e} |G|^2 \diff \sigma + 2\int_{\p \Omega_\e} \text{Re}(G \cdot \overline{1+\e G}) \langle X_\e,\nu_\e\rangle e^{-\pi |z|^2},
& J''_\e  = 2 \int_\C |G|^2 \diff \sigma.
\end{align}
Here and in what follows, we write $X_\e$ to be the vector fields built in Lemma \ref{lemma:flows-specific}. Note that, to obtain \eqref{eq:derIJ}, we used the fact that $u_\e$ is constant on $\p \Omega_\e$, together with the cancelling property \eqref{eq:zeromeanXeps} of the vector fields.

Since $\langle G, 1\rangle_{\mc F^2}=0$, $\Omega_0$ is a ball and $G$ is holomorphic, from \eqref{eq:derIJ} it is easy to see that 
\begin{equation}
    \label{eq:derK0}
    I'_0=J'_0=0 \qquad \implies \qquad \frac{\diff}{\diff \e} \K[1+\e G] \Big|_{\e=0} = K'_0 = 0,
\end{equation}
where the implication follows from the first equation in \eqref{eq:derK}.

Combining \eqref{eq:derK}--\eqref{eq:derK0}, we arrive at
\begin{equation}\label{eqn:second-variation-first-formula} 
 \frac{\diff^{\,2}}{\diff \e^{\,2}}\mc K[1+\e G]\Bigg|_{\e = 0 } 
=2\Big( \int_{\Omega_0} |G|^2 e^{-\pi |z|^2} - \|G\|_{\mc F^2}^2\int_{\Omega_0} e^{-\pi |z|^2} +  \int_{\p \Omega_0} \Re(G(z))  \langle X_0, \nu\rangle e^{-\pi |z|^2}  \Big).
\end{equation}
Since $\p\Omega_0$ is a circle of radius $r_0$, where $\pi r_0^2 =s$, our main task is to simplify the last term: specifically, we want to show that
\begin{equation}
    \label{eq:goalsecondderK}
    \int_{\p \Omega_0} \Re G \langle X_0,\nu \rangle \diff \mc H^1 = \fint_{\p \Omega_0} |G|^2 \diff \mc H^1.
\end{equation}
In order to prove \eqref{eq:goalsecondderK}, we have to understand how to write $X_0$ in terms of $G$ on $\partial \Omega_0$.

We first claim that
\begin{equation}
\label{eq:dermu}
\frac{\diff}{\diff \e}\Big|_{\e=0} \mu_\e(t) = 0,
\end{equation}
where $\mu_\e(t) \coloneqq \mu_{1+\e G}(t) =|\{u_\e>t\}|$.
To prove this claim we build, exactly as in Lemma \ref{lemma:flows-specific}, a family of vector fields $Y_\e$ with associated flows $\Psi_\e$ such that $\Psi_\e(\{u_0>t\}) = \{u_\e>t\}$ (note that, by Lemma \ref{lemma:star-shaped-levels}, these sets have smooth boundaries, hence we can apply Lemma \ref{lemma:flows}).  We compute, for $z \in \partial \{ u_0 > t \} = \{u_0 = t\},$
\begin{equation}\label{eqn:boundary-vector-expansion} 
t \equiv u_\e(\Psi_\e(z))= \left(1 + 2 \e \Re G(z)- 2\pi \e \langle Y_0(z),z\rangle + O(\e^2)\right )e^{-\pi |z|^2},
\end{equation}
and thus, since the first order term in $\e$ vanishes, we have
\begin{equation}\label{eqn:condition-vector-field-boundary} 
\Re G(z) = \pi \langle Y_0,z\rangle.
\end{equation}
We can now prove \eqref{eq:dermu}: again by Reynold's formula, we have
\begin{equation}\label{eqn:condition-derivative-measure}
\frac{\diff}{\diff \e}\Big|_{\e=0} \mu_\e(t) = \int_{\p\{u_0>t\}} \langle Y_0,\nu\rangle 
= 2 \pi \fint_{\p\{u_0>t\}} \langle Y_0,z\rangle 
= 2 \fint_{\p\{u_0>t\}} \Re G = 2 \Re G(0)=0,
\end{equation}
since $\p\{u_0>t\}$ is a circle and  $\Re G$ is harmonic, where the last equality follows from \eqref{eq:normG}.

We now make the following remark:  the function $\e \mapsto \mu_\e(t)$ is smooth in $\e,$ whenever $\e$ is sufficiently small, and also smooth in $t,$ for $t \in (\e^{c_2},\max u_\e)$, where $c_2$ is as in Lemma \ref{lemma:star-shaped-levels}. This can be seen explicitly as follows: the smoothness in $\e$ follows by Lemma \ref{lemma:flows} and the fact that $\Psi_\e(\{u_0 > t\}) = \{u_\e > t\}.$ On the other hand, by \cite[Lemma~3.2]{NicolaTilli} we have
\[
-\partial_t \mu_\e(t) = \int_{\inset{u_\e = t}} |\nabla u_\e|^{-1} \, \diff \mathcal{H}^1 \,\, \text{for almost every } t \text{ in } (0,\max u_\e). 
\]
By the proof of Lemma \ref{lemma:Level-Set-Regularity} we see that 
\begin{equation}
    \label{eq:lowerboundforsmoothnessmu}
    |\nabla u_\e(z)| \geq C(\e^{c_2}, \|G\|_{\Fock(C)})|z| ,
\end{equation}  hence $\mu_\e\in C^{0,1}_\textup{loc}(\e^{c_2}, \max u_\e)$. Moreover, the divergence theorem allows us to write 
\begin{align*} 
-\partial_t\mu_\e(t) + \partial_t\mu_\e(t_0) & = \int_{\inset{u_\e = t}} |\nabla u_\e|^{-1} \, \diff \mathcal{H}^{1} - \int_{\inset{u_\e = t_0}} |\nabla u_\e|^{-1} \, \diff \mathcal{H}^{1} \cr 
 & = \int_{\inset{u_\e = t}} \frac{\nabla u_\e}{|\nabla u_\e|^2} \cdot \frac{\nabla u_\e}{|\nabla u_\e|} \, \diff \mathcal{H}^{1} - \int_{\inset{u_\e = t_0}} \frac{\nabla u_\e}{|\nabla u_\e|^2} \cdot \frac{\nabla u_\e}{|\nabla u_\e|} \, \diff \mathcal{H}^{1} \cr 
 & = - \int_{\inset{t_0 > u_\e > t}} \ddiv\left( \frac{\nabla u_\e}{|\nabla u_\e|^2}  \right) \diff z,
\end{align*} 
whenever $t < t_0.$ By \eqref{eq:lowerboundforsmoothnessmu}, $|\nabla u_\e|^{-2}$ is bounded and smooth in the set $\{t_0 > u_\e > t\},$ which shows that $\partial_t \mu_\e\in C^{0,1}_\textup{loc}(\e^{c_2}, \max u_\e).$ By a straightforward use of the coarea formula, iterating such an argument yields the desired smoothness property of $\mu_\e.$ Moreover, we also have that $\partial_t\mu_\e(t) \le - \frac{1}{t}$, cf.\ \eqref{eqn:diffineq}, and so by the Implicit Function Theorem the functions $u_\e^*(s)$ are differentiable in the variable $\e$ for all fixed $s,$ for $\e < \e_0(s)$ sufficiently small. 

Now let us fix $s>0$ and recall that $G(0)=0$. Using the smoothness of $\e \mapsto \mu_\e$ first and then the smoothness of $\mu_0$ on a neighbourhood of $u_0^*(s),$ we obtain:  
\begin{align*}
s & = \mu_\e(u_\e^*(s)) \\
& = \mu_0(u_\e^*(s))+ 2\e \Re(G(0)) + O(\e^2)\\
& = s + (u_\e^*(s)-u_0^*(s))\frac{\diff}{\diff t}\Big|_{t=u^*(s)} \mu_0(t) + O(\e^2)\\
& = s - \frac{u_\e^*(s)-u_0^*(s)}{u_0^*(s)}  + O(\e^2),
\end{align*}
where we used the fact that $G(0)=0$ by \eqref{eq:normG}.
Thus, after rearranging, we find
$$u_\e^*(s) =  (1 + O(\e^2))e^{-s}.$$

Since $\Phi_\e$ is the flow of $X_\e$, we have
\begin{equation}
    \label{eq:expansionPhieps}
    \Phi_\e(z) = \Phi_0(z) + \e X_0(\Phi_0(z)) + O(\e^2) = z+ \e X_0(z) + O(\e^2).
\end{equation}
We now compare the two expansions
\begin{align*}
u_\e(\Phi_\e(z)) & = (1 + 2 \e \Re G(z)- 2\pi \e \langle X_0(z),z\rangle + O(\e^2) )e^{-\pi |z|^2},\\
 u_\e^*(s) & = (1 + O(\e^2)) e^{-s},
\end{align*}
and we deduce that, on $\p \{u_0>u_0^*(s)\} = \p \Omega_0$, the first order terms in $\e$ must be the same, thus
\begin{equation}\label{eq:vector-field-X_0} 
\pi \langle X_0,z\rangle = \Re G(z) \text{ on } \partial \Omega_0. 
\end{equation}
Finally, since $G$ is holomorphic and $G(0)=0$, we have
\begin{align*}
\fint_{\p \Omega_0} (\Re G)^2 \diff \mc H^1 
 = \fint_{\p \Omega_0} \Re(G^2) \diff \mc H^1+ \fint_{\p \Omega_0} (\Im G)^2 \diff \mc H^1
=  \fint_{\p \Omega_0} (\Im G)^2 \diff \mc H^1.
\end{align*}
Now \eqref{eq:goalsecondderK} follows by combining this identity with \eqref{eq:vector-field-X_0}:
\begin{align*}
   \int_{\p \Omega_0} \Re G \langle X_0,\nu \rangle \diff \mc H^1 
    & = 2\pi \fint_{\p \Omega_0} \Re G \langle X_0,z\rangle \diff \mc H^1 \\
    & = 2 \fint_{\p \Omega_0} (\Re G)^2 \diff \mc H^1 
    = 2 \fint_{\p \Omega_0} (\Re G)^2 \diff \mc H^1 
    = \fint_{\p \Omega_0} |G|^2 \diff \mc H^1 ,
\end{align*}
as wished.
\end{proof}

\begin{proof}[Proof of Proposition \ref{prop:negdef}]
Since $G\in \mc F^2$ satisfies \eqref{eq:normG}, we can write
$$G(z) = \sum_{k=2}^\infty a_k \Big(\frac{\pi^k}{k!}\Big)^\frac 1 2 z^k,
\qquad \|G\|_{\mc F^2}^2 = \sum_{k=2}^{\infty} |a_k|^2.$$
 It is direct to see that \eqref{eq:second-variation} can be rewritten using the power series for $G$ as
\begin{align}\label{eq:expansion-second-variation-coeff} 
\begin{split}
 \frac 1 2\nabla^2 \K[1](G,G) =& \sum_{k=2}^\infty |a_k|^2 V_k(s),
\end{split}
\end{align} 
where
\begin{equation}\label{eq:definition-V-k} 
V_k(s)\coloneqq \frac{\pi^k}{k!} \int_{B(0,\sqrt{\frac s \pi})} |z|^{2k} e^{-\pi |z|^2}
-  \int_{B(0,\sqrt{\frac s \pi})} e^{-\pi |z|^2} + e^{- s}\frac{s^k}{k!}.
\end{equation} 

We claim that $V_k(s) \le 0$ for all $s \ge 0.$ This follows by a simple calculus observation:
\begin{align*} 
V_k(s) &= - \frac{\pi^k}{k!} \int_{\C \setminus B(0,\sqrt{s/\pi})} |z|^{2k} e^{-\pi|z|^2} \, \diff z + \left(1 + \frac{s^k}{k!}\right)e^{-s} \cr 
& = - \frac{\Gamma(k+1,s)}{k!} +\left(1 + \frac{s^k}{k!}\right)e^{-s} = - \left( \sum_{j=1}^{k-1} \frac{s^j}{j!} \right) e^{-s},
\end{align*} 
where $\Gamma(a,s) \coloneqq \int_s^{\infty} r^{a-1} e^{-r} \diff r$  denotes the upper incomplete Gamma function. 
In order to conclude the desired bound, notice that $\lim_{k \to \infty} V_k(s) = e^{-s} - 1 <0,$ and, since $V_k(s)$ is \emph{decreasing} in $k$ for $s>0$ fixed,
$$\inf_{k \ge 2} (-V_k(s)) = -V_2(s) = se^{-s}.$$ 
The conclusion of Proposition \ref{prop:negdef} follows then directly from \eqref{eq:expansion-second-variation-coeff}. 
\end{proof}

It is important to note that \eqref{eq:normG} is crucial as a normalization for the above proof to work: as a matter of fact, many of the cancellations in the proof of Lemma \ref{lemma:second-variation-K} only appeared since $G(0) = \langle G, 1 \rangle = 0.$  Moreover, if $\langle G,z \rangle \neq 0,$ it could happen that  $\nabla^1 \mathcal{K}[1](G,G) = 0,$ which would cause the proof of sharp stability to collapse. 

The argument implicit in the reduction to \eqref{eq:normG} is hence a vital part of the proof: heuristically, it plays the pivotal role of providing us with a single point $z_0$ -- which, through translations, may be assumed to be the origin -- for which one can compare the level sets of the functions $u_\e$ to balls centered at $z_0.$ The fact that $z_0$ is given by the point where each $u_\e$ attains its maximum allows thus for a connection between the analytic and geometric natures of the problem, highlighting further the importance of the aforementioned reduction. 



\section{Sharpness of the stability estimates}
\label{sec:sharpness}

In this short section we prove the sharpness claimed in Remark \ref{rmk:sharpness} concerning the estimates in Theorem \ref{thm:Stability}. We will see that the variational approach of the previous section is quite useful in this regard. The following is the key proposition we require:

\begin{numprop}\label{prop:estimate-eigenvalues} Let $s>0$ be a fixed positive real number. For each $\varepsilon > 0$ sufficiently small there is a constant $C>0$ and sequences $\{\Omega_\e\}_{\e}$ and $\{\tilde F_\e\}_{\e} \subset \mc F^2(\C) $ with $\|\tilde F_\e \|_{\mc F^2} = 1, \quad \forall \, \e > 0,$ and such that:
\begin{enumerate}
    \item  $\Omega_0$ a ball and $|\Omega_\e| = s$;
    \item $\inf_{c,z_0 \in \C} \|\tilde F_\e - c \cdot F_{z_0}\|_{\mc F^2} \ge \frac{\e}{C}$;
    \item the deficit satisfies $\delta(\tilde F_\e;\Omega_\e) \le C \frac{s e^{-s}}{1-e^{-s}} \e^2.$
\end{enumerate}
\end{numprop}

\begin{proof} Let $F_{\eps}(z) = 1 + \eps z^2$ and as usual let us write $u_\eps(z):= |F_\eps(z)|^2 e^{-\pi |z|^2}.$ Consider, as in \eqref{eq:defOmegae}, the domains 
\[
\Omega_\e = \{ z \in \C \colon u_\eps(z) > u_\e^*(s)\},
\]
where $s> 0$ is fixed. We then have 
\begin{align}\label{eq:upper-bound-difference-eigenvalues} 
 \begin{split}
 (1-e^{-s})\delta(F_\e;\Omega_\e)  & =  \mathcal{K}[1] - \mathcal{K}[F_\eps]\\ & \leq  -\frac{\e^2}{2} \nabla^2 \mathcal{K}[1](z^2,z^2) + \e^2\eta(\e) = \frac{2se^{-s}}{\pi^2} \e^2 - \eta(\e)\e^2,
 \end{split}
\end{align} 
where we used Lemma \ref{lemma:control3rdder} to pass to  the second line and also \eqref{eq:second-variation} in the last equality. Now note that taking $\e$ sufficiently small yields the desired upper bound if we choose $\tilde F_\e = \frac{F_\e}{\|F_\e\|_{\mc F^2}}.$ For the lower bound on $\|\tilde F_\e - c F_{z_0}\|_{\mc F^2},$ we recall from \eqref{eq:max} that
\begin{align*} 
\|\tilde F_\e - c \cdot F_{z_0}\|_{\mc F^2}^2\ge 1 - \max_{z_0 \in \C} |\tilde F_\e(z_0)|^2e^{-\pi|z|^2}.
\end{align*} 
In order to finish, we  only need to show that the only global maximum of $|F_\e(z)|^2 e^{-\pi|z|^2}$ occurs at $z=0,$ which is equivalent to showing that
\[
(1+ 2 \e(x^2 - y^2) + \e^2 \abs{z}^4) < e^{\pi |z|^2}, 
\]
for each $z \in \C.$ As $1 + \pi|z|^2 + \frac{\pi^2}{2} |z|^4<e^{\pi|z|^2},$ this inequality is true if $\e < \frac{\pi}{4}.$ Thus, for such $\e$,
\[
\|\tilde F_\e - c F_{z_0}\|_{\mc F^2}^2 \ge 1 - \frac{1}{1 + \frac{2}{\pi^2}\e^2} \ge \frac{\e^2}{\pi^2},
\]
which concludes the proof.
\end{proof}

We are now ready to prove the claims in Remark \ref{rmk:sharpness}.

\begin{numcor}\label{cor:sharpness} The following assertions hold:
\begin{enumerate}
    \item\label{it:optimaldelta} The factor $\delta(f;\Omega)^{1/2}$ cannot be replaced by $\delta(f;\Om)^{\beta},$ for any $\beta > 1/2,$ in \eqref{eqn:StabilityFunction} and \eqref{eqn:StabilitySet};
    \item\label{it:optimalconstant} There is no $c\in(0,1)$ such that, for all measurable sets $\Omega \subset \C$ of finite measure, we have
    $$ \min_{z_0\in \C, |c|=\|f\|_{2}} \frac{\|f - c\, \varphi_{z_0} \|_2}{\|f\|_2} \leq C \Big(e^{c |\Omega|} \delta(f;\Om)\Big)^{1/2}. $$
\end{enumerate}
\end{numcor}

\begin{proof} 
Notice that \ref{it:optimalconstant} follows directly from the statement of Proposition \ref{prop:estimate-eigenvalues} by taking $s \to \infty$, so we just have to prove 
\ref{it:optimaldelta}. The fact that one cannot improve the exponent in \eqref{eqn:StabilityFunction} follows directly from Proposition \ref{prop:estimate-eigenvalues} above.

To see that one cannot improve the exponent in \eqref{eqn:StabilitySet} we argue as follows. For the domains $\Omega_\e$ built in Proposition \ref{prop:estimate-eigenvalues},  we may use Lemma \ref{lemma:flows-specific} to write $\Omega_\e = \Phi_\e(\Omega_0)$, provided that $\e$ is small enough. As we saw in \eqref{eq:expansionPhieps} we may write 
$$\Phi_\e(z) = z + \e X_0(z) + O(\e^2),\qquad \text{where } X_0(z) = h_0(z) z,$$ for some scalar function $h_0\colon \C\backslash \{0\}\to \R$. Indeed, that $X_0$ has this form follows from its explicit formula \eqref{eq:defXeps} in Lemma \ref{lemma:flows-specific}. Since $\pi \langle X_0(z), z \rangle = \Re(z^2)$ on $\partial \Omega_0$ by \eqref{eq:vector-field-X_0},  we have 
$$h_0(z) = \frac{\Re(z^2)}{\pi |z|^2} = \frac{\cos(2 \theta)}{\pi}$$ for $z = r(\Omega_0)e^{i \theta} \in \partial \Omega_0,$ where $r(\Omega_0)$ denotes the radius of the ball $\Omega_0$. Hence, 
\[
|\Omega_\e \triangle \Omega_0| \geq |\Omega_\e \setminus \Omega_0| \ge  \Big|\Big\{z=re^{i \theta} \colon r(\Omega_0)< r <r(\Omega_0)+\e\frac{\cos(2\theta)}{\pi} -C\e^2 \Big\}\Big|> c \, r(\Omega_0)^2 \e, 
\]
which concludes the proof.
\end{proof}

\section{Generalizations to higher dimensions}\label{sec:Generalize} 

In this section we will provide the generalization of Theorems~\ref{thm:Stability} and \ref{thm:StabilityFockSpace} to 
higher dimensions $d\geq 1$.
Given a  window function $g \in L^2(\R^d)$,
the STFT of a function $f \in L^2(\R^d)$ is defined as
\begin{equation*}
\label{eqn:DefnSTFTDimensiond}
   V_g  f (x, \om) \coloneqq \int_{\R^d} e^{-2 \pi i y \cdot \om} f(y) \conj{g(x-y)} \diff y , \quad x , \om \in \R^d,
\end{equation*}
coherently with \eqref{defSTFT}. As in dimension 1, we will only be interested in the case where $g(x)=\varphi(x)$ is the
standard Gaussian window
defined as
\begin{equation}\label{eqn:DefnGaussianDimensiond}
    \varphi(x) \coloneqq 2^{d/4} e^{- \pi \abs{x}^2} , \quad x \in \R^d,
\end{equation}
so as before we set $\STFT f \coloneqq V_{\varphi} f.$ Note that \eqref{eqn:DefnGaussianDimensiond} reduces to \eqref{eqn:GaussianWindow} when $d=1$.


The $d$-dimensional version of Theorem~\ref{thm:FKforSTFT}, proved in \cite{NicolaTilli}, can be stated as follows.
\begin{theorem}[\cite{NicolaTilli}; Faber-Krahn inequality for the STFT in dimension $d$]\label{thm:FKforSTFTDimensiond}
   If $\Omega\subset \R^{2d}$ is a measurable set with finite Lebesgue measure $|\Omega|>0$, 
    and $f\in L^2(\R^{d})\setminus\{0\}$ is an arbitrary function,
    then
    \begin{equation}\label{eqn:FKforSTFTDimensiond}
    \frac{\int_{\Omega} |\STFT f(x,\omega)|^2 \diff x \diff \omega}
    {\Vert f\Vert_{L^2(\R^d)}^2}
    \leq 
    \int_0^{|\Omega|} e^{-(d!\, s)^{\frac 1 d}} \diff s.
    \end{equation}
    Moreover, 
    equality is attained  if and only if $\Om$ coincides (up to a set of measure zero) with a ball
    centered at some $z_0=(x_0, \om_0)\in\R^{2d}$  and, at the same time, for some $c\in \C\setminus\{0\}$
    \begin{equation}\label{eqn:FKEqualityFunctionsd}
        f(x) = c\, \varphi_{z_0}(x) , \qquad  \varphi_{z_0}(x) := e^{2 \pi i \om_0\cdot  x} \varphi(x-x_0),
    \end{equation}
    where $\varphi$ is the Gaussian defined in \eqref{eqn:DefnGaussianDimensiond}.
\end{theorem}
We point out that in \cite[Theorem 4.1]{NicolaTilli} the right hand side of \eqref{eqn:FKforSTFTDimensiond} is expressed 
in terms of the incomplete Gamma function and implicit constants depending on $d$, whereas the present formulation (which is 
equivalent but more explicit)
is taken from \cite{NicolaTilli2} (see the remark after Theorem 2.3 therein).

It appears from \eqref{eqn:FKforSTFTDimensiond} that, in  dimension $d\geq 1$, the function
\begin{equation}
\label{defe*}
    \ee(s):=e^{-(d!\, s)^{\frac 1 d}},\quad s\geq 0,
\end{equation}
plays a crucial role, since when $d=1$, $\ee(s)=e^{-s}$ and the right hand side of \eqref{eqn:FKforSTFTDimensiond} reduces
to $1-e^{-|\Omega|}$, as in \eqref{eqn:FKInequalityFock}. 
To state a stability result in dimension $d$, we must modify the deficit $\delta$ defined in \eqref{eqn:Deficit} 
to suit the right hand side of  \eqref{eqn:FKforSTFTDimensiond}, so we let
\begin{equation}\label{eqn:DeficitDimensiond}
    \delta(f;\Om) \coloneqq 
    1 - \frac
    {\int_\Om |\mathcal{V} f(x,\om)|^2 \, \diff x \, \diff \om}
    {\|f\|_{L^2(\R^d)}^2\,\int_0^{|\Omega|} \ee(s) \diff s },
\end{equation}
which once again reduces to \eqref{eqn:Deficit} when $d=1$.
Redefining
the asymmetry index $\mcalA (\Om)$ by simply replacing $\R^2$ with $\R^{2d}$ in \eqref{eqn:FraenkelAsymmetry},  
our extension of Theorem~\ref{thm:Stability} to dimension $d$ can be stated as follows.

\begin{numthm}[Stability of the FK inequality for the STFT in dimension $d$]\label{thm:StabilityDimensiond}
       There is an explicitly computable constant $C=C(d)>0$ such that, for all  measurable sets $\Om \subset \R^{2d}$ with finite measure
    $|\Omega|>0$ and all functions $f \in L^2(\R^d)\backslash\{0\}$, we have 
    \begin{equation}\label{eqn:StabilityFunctionDimensiond}
        \min_{z_0\in \C^d, |c|=\|f\|_{2}} \frac{\|f - c\, \varphi_{z_0} \|_2}{\|f\|_2} \leq C  
        \left(\frac{\delta(f;\Om)}
      {\int_{|\Omega|}^\infty   e^{-(d!\, s)^{\frac 1 d}} \diff s  } \right)^{\frac 1 2}  .     
    \end{equation}
    Moreover, for some explicit constant $K=K(d,|\Omega|)$ we also have
    \begin{equation}\label{eqn:StabilitySetDimensiond}
        \mcalA(\Om) \leq K  \delta(f;\Om)^{1/2} .
    \end{equation}
\end{numthm}

As in the case of Theorem~\ref{thm:Stability}, the first step is to translate the problem into the Fock space
$\Fock(\C^d)$, now defined as the Hilbert space of all holomorphic functions $F \colon \C^d \to \C$ such that
\[ \Fnorm{F} \coloneqq \of{\int_{\C^d} \abs{F(z)}^2 e^{- \pi \abs{z}^2} \diff z}^{1/2} \ls \infty, \]
with its induced inner product. An orthonormal basis -- that reduces to \eqref{monomials} when $d=1$ -- is given, using multi-index notation, by the normalized monomials
\begin{equation}
    \label{monomialsd} e_\alpha(z)= (\pi^{|\alpha|}/\alpha!)^{1/2} \,z^\alpha,\quad \alpha\in\N^d,\quad z\in\C^d,
\end{equation}
while the reproducing kernels are the functions $K_w(z)=e^{\frac \pi 2 |w|^2} F_w(z)$, where, in analogy to \eqref{eqn:DefnFz0},
\begin{equation}
    \label{defFzd} F_{z_0}(z)= e^{-\frac{\pi}{2} \abs{z_0}^2} e^{\pi z \cdot\conj{z_0}} .
\end{equation}
The Bargmann transform is now an unitary operator from $L^2(\R^d)$ onto $\Fock(\C^d)$, defined  as in \eqref{defB}, with $\R^d$ and $\C^d$ in place of $\R$ and $\C$, and the multi-index notation
being adopted. 
Moreover, the functions $F_{z_0}$ in \eqref{defFzd} are, much as in \eqref{eqn:DefnFz0}, the Bargmann transforms of the optimal
functions $\varphi_{z_0}$ defined in \eqref{eqn:FKEqualityFunctionsd}.
In this setting, 
since by an identity similar to \eqref{concB} the concentration of a function $f$ on $\Omega$ 
can still be expressed in terms of its Bargmann transform,
one can rephrase  Theorem~\ref{thm:StabilityDimensiond}  in terms of Fock spaces.

\begin{numthm}[Fock space version of Theorem \ref{thm:StabilityDimensiond}]\label{thm:StabilityFockSpaceDimensiond}
   There is a computable constant $C=C(d)>0$ such that, for all  measurable sets $\Om \subset \R^{2d}$ with finite measure $|\Omega|>0$ and all functions $F\in \Fock(\C^d)\backslash\{0\}$, we have 
    \begin{equation}\label{eqn:StabilityFunctionFockDimensiond}
        \min_{\substack{|c|=\|F\|_{\Fock},\\ z_0\in \C^d}} \frac{\Fnorm{F-cF_{z_0}}}{\Fnorm{F}} \leq
        C
        \left(\frac{\delta(F;\Om)}
      {\int_{|\Omega|}^\infty   e^{-(d!\, s)^{\frac 1 d}} \diff s  } \right)^{\frac 1 2},
    \end{equation}
where
\begin{equation}
\label{defdeltaFOd}
\delta(F;\Omega):=1-\frac{\int_\Omega |F(z)|^2 e^{-\pi |z|^2}\diff z}
{\Vert F\Vert_{\Fock}^2\, \int_0^{|\Omega|} \ee(s)\diff s}.
\end{equation}
Moreover, for some explicit constant $K=K(d,|\Omega|)$ we also have
    \begin{equation}\label{eqn:StabilitySetFockDimensiond}
        \mcalA(\Om) \leq K \delta(F;\Om)^{1/2} .
    \end{equation}
\end{numthm}
We point out that \eqref{eqn:StabilityFunctionFockDimensiond} reduces to \eqref{eqn:StabilityFunctionFock}, when
$d=1$ and $\ee(s)=e^{-s}$.

    The proof of \eqref{eqn:StabilityFunctionFockDimensiond} can be obtained by arguments 
  similar to those  given in Section \ref{sec:firstproof}, where every result has a suitable analogue in dimension $d$.
    Therefore, we limit ourselves to describing the  relevant, and not always trivial, changes  that
    are necessary to adapt Section \ref{sec:firstproof} to dimension $d$.
    
    We start with
    the necessary background results from  \cite{NicolaTilli}, which we discussed
    in Subsection \ref{subsec:12} only in dimension one. We warn the reader that
in \cite{NicolaTilli} some numerical constants were written in terms of 
$\bol_{2d}$, the volume of the
unit ball in $\R^{2d}$: here, in \eqref{diffineqd} and \eqref{diffinequd},
we write them explicitly using the fact that $\bol_{2d}=\pi^d/d!$, as done in \cite{NicolaTilli2}.

Given $F\in \Fock(\C^d)$, the function $u$ and its super-level sets $A_t$ are  defined as in \eqref{eqn:Defnu} and \eqref{eqn:DefnSuperLevelSetsOfu},
 now and henceforth with $z\in\C^d$. The distribution function $\mu(t)$ is defined as in \eqref{def:mut}, with
the adopted convention that $|\cdot|$ denotes
Lebesgue measure in $\R^{2d}$, but with \eqref{eqn:diffineq} being replaced (see \cite[\S 4]{NicolaTilli}) by
    \begin{equation}
    \label{diffineqd}
        \mu'(t) \leq   \,- \,\, \frac{d\,\mu(t)^{1-1/d}}{(d\,!)^{\frac 1 d} \, t}  \quad\text{for a.e. $t\in (0,T),\quad T \coloneqq \max_{z\in\C^d} u(z)$},
    \end{equation}
while
\eqref{estmut1} becomes
\[
\mu(t)\geq \frac 1 {d!} \left(\log_+ \frac T t\right)^d\quad \text{for all } t>0.
\]
Similarly, the decreasing rearrangement $u^*(s)$ (i.e.,\ the inverse function of $\mu(t)$)
is defined exactly as in \eqref{eqn:defu*}, but now with \eqref{eqn:DiffInequ} being replaced by
\begin{equation}
    \label{diffinequd}
    (u^*)'(s)+\frac {(d\,!)^{\frac 1 d}\,\, u^*(s)}{d \,\, s^{1-\frac 1 d}}\geq 0,\quad \text{for a.e. $s>0$.}
\end{equation}
These changes are natural in dimension $d$,  since when $F$  equals one of the optimal
functions defined in \eqref{defFzd}, we have $u(z)=e^{-\pi |z-z_0|^2}$ and its distribution function is
\begin{equation}
    \label{defmuspec}
\mu(t)=
\frac 1 {d!} \left(\log_+ \frac 1 t\right)^d,\quad t>0,
\end{equation}
as the explicit volume of the unit ball is $\bol_{2d}=\pi^d/d!$. Note that, for the particular $\mu$ in \eqref{defmuspec}, \eqref{diffineqd} is an equality.
Moreover,
if in \eqref{defmuspec} we let $\mu(t)=s>0$ and solve for $t$, the resulting inverse function
is just the function $\ee(s)$ defined in \eqref{defe*}, much as $e^{-s}$ is the inverse of $\log_+ \frac 1 t$ when $d=1$. In particular, note that \eqref{diffinequd} becomes an equality when $u^*=\ee$.

Finally, the fact, expressed by \eqref{eqn:DefnI},  that super-level sets maximize the concentration under a volume constraint,
is clearly still valid, as so is \eqref{eq:repkernel}, with equality if and only if $F$ is a multiple of some $F_{z_0}$.

For what concerns Section \ref{sec:firstproof}, beside obvious changes such as $\C$ being replaced with $\C^d$ and analogous changes, a general rule is that $e^{-s}$ should always be replaced by $\ee(s)$, e.g. in \eqref{massone} and \eqref{defs*}, 
and $\log$ with $(\log)^d/d!$, e.g. in \eqref{deft*}.
Accordingly, the claim of Lemma~\ref{lemma:super-level-new} becomes
\[ \mu(t) \leq \inv{d!} \of{1+C_0(1-T)} \of{\log{T/t}}^d \quad\forall t\in [t_0,T], \]
where now the underlying constants may depend on the dimension $d$.
The proof follows the same pattern, with some changes being necessary, which we now indicate. Replacing $n$ by $\alpha=(\alpha_1,\ldots,\alpha_d)\in {\mathbb N}^d$ and adopting
the multi-index notation, \eqref{defR} becomes
\begin{equation}
\label{defRd}
R(z) \coloneqq \sum_{|\alpha|\geq 2} \frac{a_\alpha}{\sqrt{T}}\, \frac {\pi^{|\alpha|/2} z^\alpha }{\sqrt{\alpha!}},\quad z=(z_1,\ldots,z_d)\in\C^d,
\end{equation}
where we used the basis defined in \eqref{monomialsd}. Accordingly, \eqref{estcoda} changes into 
\begin{equation}
  \label{estcodad}
  \sum_{|\alpha|\geq 2} \frac {|a_\alpha|^2}T=\frac {1-T}T=:\delta^2.
\end{equation}
Some additional caution is needed in order to estimate the
subsequent powers series.
The outcome of \eqref{eq99} is unchanged, now with $z\in\C^d$,
but after Cauchy--Schwarz one faces the multivariate power series
\begin{equation}
\label{eq99bis}
\sum_{|\alpha|\geq 2} \frac {\pi ^{|\alpha|} \left|z^{2\alpha}\right|}{\alpha !}=
\sum_{|\alpha|\geq 2} \frac {\pi ^{|\alpha|} \left|z_1^{2\alpha_1}\cdots z_d^{2\alpha_d} \right|}
{\alpha_1!\cdots \alpha_d!}=e^{\pi |z|^2}-1-\pi |z|^2.
\end{equation}

Subtler changes are needed in
\textsc{Step II}. After defining $h(z)$ as in \eqref{defh} and obtaining \eqref{esth2},
one replaces \eqref{Rprime} by
\begin{equation}
\label{Rprimed}
\left|\frac{\partial R(z)}{\partial z_i}\right|\leq
\delta
\sqrt{2} \pi |z| e^{\frac{\pi |z|^2}2},\quad z\in \C^d,\quad
1\leq i\leq d.
\end{equation}
For instance, letting $e_1=(1,0,\ldots,0)\in {\mathbb N}^d$, differentiating \eqref{defRd}, and then
using Cauchy-Schwarz and \eqref{estcodad} one obtains
\begin{equation}
\label{Rprimedd}
\left|\frac{\partial R(z)}{\partial z_1}\right|
\leq
\sum_{|\alpha|\geq 2} \frac{|a_\alpha|}{\sqrt{T}}\, \frac {\pi^{|\alpha|/2} \alpha_1  \left| z^{\alpha-e_1}
\right|}{\sqrt{\alpha!}}
\leq
\delta \left(
\sum_{|\alpha|\geq 2}  \frac {\pi^{|\alpha|} \alpha_1^2  \left| z^{2(\alpha-e_1)}
\right|}{\alpha!}
\right)^{\frac 1 2}.
\end{equation}
Focussing on multi-indices $\alpha$ of a given size
$k\geq 2$, we have
\[
\sum_{|\alpha|=k}
\frac {\pi^{|\alpha|} \alpha_1^2  \left| z^{2(\alpha-e_1)}
\right|}{\alpha!}
=
\sum_{\substack{|\alpha|=k\\ \alpha_1\geq 1}}
\frac {\pi^{k} \alpha_1  \left| z^{2(\alpha-e_1)}
\right|}{(\alpha_1 -1)!\alpha_2!\cdots \alpha_d!}
=
\sum_{|\beta|=k-1}
\frac {\pi^{k} (1+\beta_1)  \left| z^{2\beta}
\right|}{\beta_1!\cdots \beta_d!}.
\]
Since $1+\beta_1\leq k\leq 2(k-1)$, we have,
from the multinomial theorem,
\[
\sum_{|\alpha|=k}
\frac {\pi^{|\alpha|} \alpha_1^2  \left| z^{2(\alpha-e_1)}
\right|}{\alpha!}
\leq
\frac {2\pi^k} {(k-2)!}\sum_{|\beta|=k-1}
\frac {(k-1)!   \left| z^{2\beta}
\right|}{\beta_1!\cdots \beta_d!}
=\frac {2\pi^k} {(k-2)!} (|z_1|^2+\cdots |z_d|^2)^{k-1}
\]
and, since $|z_1|^2+\cdots |z_d|^2=|z|^2$, summing over $k\geq 2$ we obtain
\[
\sum_{|\alpha|\geq 2}  \frac {\pi^{|\alpha|} \alpha_1^2  \left| z^{2(\alpha-e_1)}
\right|}{\alpha!}
\leq
2 \sum_{k=2}^\infty \frac {\pi^k |z|^{2(k-1)}} {(k-2)!} =2\pi^2 |z|^2 e^{\pi |z|^2},
\]
which combined with \eqref{Rprimedd} proves \eqref{Rprimed} when $i=1$ (when $i>1$ the proof being the same).
The analogue of \eqref{Rsecond} now 
reads
\begin{equation}
  \label{estRsecondd}
\left|\frac{\partial^2 R(z)}{\partial z_i\partial z_j}\right|\leq
C\delta
 (1+|z|^2) e^{\frac {\pi |z|^2}2},\quad z\in \C^d, \quad 1\leq i,j\leq d.
\end{equation}
For instance, differentiating \eqref{defR} twice,
then using Cauchy--Schwarz and \eqref{estcodad}, one finds
\begin{equation}
\label{derR12}
\left|\frac{\partial^2 R(z)}{\partial z_1\partial z_2}\right|\leq
\delta
\left(
\sum_{|\alpha|\geq 2}  \frac {\pi^{|\alpha|} \alpha_1^2 \alpha_2^2  \left| z^{2(\alpha-e_1-e_2)}
\right|}{\alpha!}
\right)^{\frac 1 2},
\end{equation}
where $e_2=(0,1,0,\ldots,0)\in {\mathbb N}^d$. Since the  sum
can be restricted to those multi-indices $\alpha$ where $\alpha_1\geq 1$ and $\alpha_2\geq 1$
(which imply that $|\alpha|\geq 2$), letting $\beta=\alpha-e_1-e_2$ we have
\begin{align*}
\sum_{|\alpha|\geq 2}  \frac {\pi^{|\alpha|} \alpha_1^2 \alpha_2^2  \left| z^{2(\alpha-e_1-e_2)}
\right|}{\alpha!}
& =
\sum_{\beta\in {\mathbb N}^d} 
\frac {\pi^{2+|\beta|} (1+\beta_1)(1+\beta_2)  \left| z^{2 \beta}
\right|}{\beta!}\\
& = \pi^2 S(z_1)S(z_2)\prod_{j=2}^d \left(\sum_{\beta_j=0}^\infty \frac {\pi^{\beta_j} 
| z_j|^{2\beta_j}}
{\beta_j!}
\right)=\pi^2 S(z_1)S(z_2) e^{\pi (|z_3|^2+\cdots +|z_d|^2)},
\end{align*}
where
\[
S(z_i):=\sum_{\beta_i=0}^\infty \frac {\pi^{\beta_i} (1+\beta_i)| z_i|^{2\beta_i}}
{\beta_i!} < \pi (1+|z_i|^2) e^{\pi |z_i|^2}.
\]
Plugging these estimates into \eqref{derR12}, one obtains \eqref{estRsecondd} when $i=1$ and $j=2$, and hence also for all $i\not=j$, by the same argument. Finally, the case where $i=j$ can be treated similarly,
by a suitable modification of \eqref{Rsecond}.
  
Then, by \eqref{defh} and the Cauchy--Riemann equations, it is easy to see that \eqref{Rprimed}
and \eqref{estRsecondd} provide bounds for the gradient $\nabla h(z)$ and the Hessian $D^2 h(z)$:
in particular, in polar coordinates $r\omega\in \R^{2d}$, one has the following bounds for
the radial derivatives
\begin{equation}
\label{esthrd}
\left\vert \frac {\partial h(r \omega)}{\partial r} \right\vert\leq
\delta C r e^{\frac {\pi r^2}2},\qquad
\left\vert \frac {\partial^2 h(r \omega)}{\partial r^2} \right\vert
\leq
\delta C (1+r^2) e^{\frac {\pi r^2}2},\quad r\geq 0,\quad
\omega\in {\mathbb S}^{2d-1},
\end{equation}
which replace \eqref{esthr} and \eqref{esthrr}.
The rest of the proof requires only minor changes, such as the systematic
usage of polar coordinates $r\omega\in \R^{2d}$ (instead of $r e^{i\theta}$) 
with $r\geq 0$ and $\omega\in \mbbS^{2d-1}$, as in \eqref{esthrd}.
In particular, in \eqref{defrs} and in the sequel, $r_\sigma=r_\sigma(\theta)$
becomes $r_\sigma=r_\sigma(\omega)$. Also integrals should be changed accordingly, e.g., \eqref{deffs}
now becomes
\[
f(\sigma):=|E_\sigma|=\frac 1 {2d} \int_{\mbbS^{2d-1}} r_\sigma(\omega)^{2d}\diff S(\omega),\quad \sigma\in [0,1],
\]
with the obvious related changes, e.g. in \eqref{fprime},
while \eqref{fzez} becomes
\[
f(0)=\frac 1 {2d} \int_{\mbbS^{2d-1}} r_0(\omega)^{2d}\diff S(\omega)=|B(0,r_0)|=\frac {\pi^d}{d!} r_0^{2d}.
\]
Corollary \ref{cor:nondegen} is unchanged and has  a   similar proof, where one of course should replace $\log$ with
$(\log )^d/d!$ as already mentioned.

The claim \eqref{estlemma1} of Lemma \eqref{lemma:QuantitativeMaxuNew}
must be rewritten as
\begin{equation}
    \label{estlemma1d}
\frac {(1-T)^{d+1}}{(d+1)!}\leq \int_0^{s^*}\left(\ee(s)-u^*(s)\right)\diff s\leq
\frac {\delta_{s_0}}{\int_{s_0}^\infty \ee(s)\diff s},
\end{equation}
with the proof, after rewritten in terms of $\ee(s)$ remaining valid almost ad litteram to prove  \eqref{estlemma1d},
the only necessary changes being the following.
For the first inequality in \eqref{estlemma1d}, in \eqref{eq2003}
one should use, instead of $\ee(s)\geq 1-s$ as when  $d=1$ and $\ee(s)=e^{-s}$, the similar inequality
\[
\ee(s)=e^{-(d!\,s)^{\frac 1 d}}\geq 1 - (d!\,s)^{\frac 1 d}, 
\]
and change \eqref{eq2003} into
\[
\int_0^{s^*}\left(\ee(s)-u^*(s)\right)\diff s\geq \int_0^{s^*}\left(1 - (d!\,s)^{\frac 1 d}-T\right)_+\diff s
=
\int_0^{\frac {(1-T)^d}{d!}} \left(1 - (d!\,s)^{\frac 1 d}-T\right)\diff s,
\]
which after a routine computation  yields the first  inequality in \eqref{estlemma1d}. 

Then the proof goes on unaltered, except that now \eqref{ratiodecr} follows from \eqref{diffinequd}
and \eqref{defe*}, rather than \eqref{eqn:DiffInequ} and \eqref{tempv*}, and 
in \textsc{Case 1} one arrives at
\eqref{eq2005}. Since $\eps\leq\delta_{s_0}$ by virtue of \eqref{eq2001}, \eqref{eq2005} proves
\eqref{estlemma1d}. Finally, \textsc{Case 2} requires no changes, since \eqref{eq2006}
already implies \eqref{estlemma1d}.

The claim of Lemma \ref{reinforcedlemma} now becomes
\begin{equation}
    \label{estreinforcedd}
1-T\leq C\int_0^{s^*} \bigl( \ee(s)-e^*(s)\bigr)\diff s.
\end{equation}
In the proof, the first inequality in \eqref{estlemma1d} can now be used to justify, in a similar way, why on proving \eqref{estreinforcedd} one can freely assume \eqref{eq1004} if needed.
Thus, replacing also $\log\frac 1 t$ by the right hand side of \eqref{defmuspec}, and
rewriting \eqref{eq1005} as
\[
\mu(t)\leq \frac 1 {d!} (1+C_0(1-T))\log\frac T t\quad\text{for all } t\in[\tau^*,T],
\]
one obtains the following version of \eqref{eq1010}:
\begin{equation}
\label{eq1010d}
d! \int_0^{s^*} \bigl(\ee(s)-u^*(s)\bigr)\diff s
\geq  \int_{\tau_1}^T \left(\left(\log\frac 1 t\right)^d -
(1+C_0(1-T))\left(\log\frac T t\right)^d\right)\diff t.
\end{equation}
Then,  using 
$a^d-b^d\geq (a-b)a^{d-1}$ with the choice $a=\log\frac 1 t$
and $b=\log\frac T t$, since $a-b=-\log T$ and $-\log T\geq 1-T$ we can replace
the minorization after \eqref{eq1010} by
\begin{align*}
\left(\log\frac 1 t\right)^d -
(1+C_0(1-T))\left(\log\frac T t\right)^d
& \geq (-\log T)\left(\log\frac 1 t\right)^{d-1}-
C_0(1-T))\left(\log\frac T t\right)^d\\
& \geq (1-T) \left(\log\frac 1 t\right)^{d-1}-
C_0(1-T))\left(\log\frac 1 t\right)^d \\
& \geq (1-T) \left(\log\frac 1 t\right)^{d-1}\left(1-C_0 \log\frac 1 {\tau_1}\right),
\end{align*}
for all $t\in [\tau_1,T]$.
Finally, fixing $\tau_1\in (\tau^*,1)$ in analogy to \eqref{epspos}, from \eqref{eq1010d}
and the previous estimate,
in place of \eqref{eq1021}
now one obtains
\[
d! \int_0^{s^*} \bigl(\ee(s)-u^*(s)\bigr)\diff s
\geq 
\eps_1(1-T)\int_{\tau_1}^T \left(\log\frac 1 t\right)^{d-1}\diff t.
\]
As explained after \eqref{eq1021}, one can proceed by further assuming $T\geq\tau_2>\tau_1$, now obtaining
\[
d! \int_0^{s^*} \bigl(\ee(s)-u^*(s)\bigr)\diff s
\geq 
\eps_1(1-T)\int_{\tau_1}^{\tau_2} \left(\log\frac 1 t\right)^{d-1}\diff t,
\]
which proves \eqref{estreinforced}.

Lemma \ref{lemma:kern}, as is well known, remains valid, with the obvious notational changes and the reproducing
kernels described before \eqref{defFzd}. As a consequence, the proof of \eqref{eqn:StabilityFunctionFockDimensiond}
can be completed essentially  as the proof of \eqref{eqn:StabilityFunctionFock}, replacing \eqref{eq1008} by
\[
\min_{\substack{z_0\in\C^d \\ |c|=1}}
     \Vert F-cF_{z_0}\Vert_{\Fock(\C^d)}^2
     \leq C\frac {\delta_{s_0}} {\int_{s_0}^\infty \ee(s)\diff s}.
\]

Finally, the proof of the set stability remains virtually unchanged. This finishes the proof of Theorem \ref{thm:StabilityDimensiond}. 
    

\vspace{2mm} 

We now discuss the sharpness of the estimate in Theorem \ref{thm:StabilityDimensiond}, in analogy to the discussions of Sections \ref{sec:alternative-proof} and \ref{sec:sharpness}, and we explain how to adapt the arguments of these sections to the case of general dimension; as before, we keep the notation from these sections. 

The first observation to be made  is that Lemmas \ref{lemma:flows} and \ref{lemma:flows-specific} hold after the obvious changes have been made, with essentially identical proofs. As in Section \ref{sec:alternative-proof}, we wish to compute the second variation $\partial_\e^2\K[1+\e G]|_{\e=0} ,$ where $G$ now satisfies that 
$$\langle G,1 \rangle = \langle G,z_i \rangle = 0, \quad i = 1,\dots,d.$$ 
We note that the same argument as in the proof of Lemma \ref{lemma:second-variation-K} implies that \eqref{eqn:second-variation-first-formula} still holds when passing to the higher-dimensional case. Hence, we only need to compute $\langle X_0, \nu \rangle,$ where $X_0$ is defined as the vector field associated with the flows $\Phi_\e$ at $\e=0$ in the analogue of Lemma \ref{lemma:flows-specific}, and $\nu$ denotes the unit normal at $\partial \Omega_0.$ 

In order to do so, we adapt the proof of Lemma \ref{lemma:second-variation-K}: first, note that equations \eqref{eqn:boundary-vector-expansion}--\eqref{eqn:condition-derivative-measure} hold in the exact same way also in the higher-dimensional case. Moreover, we also note that
\begin{align*}
s &= \mu_\e(u_\e^*(s)) \cr 
  &= \mu_0(u_\e^*(s)) + O(\e^2) \cr 
  &= s - (u_\e^*(s) - u_0^*(s))  \frac{d \cdot s^{1-1/d}}{(d!)^{1/d} u_0^*(s)} + O(\e^2), 
\end{align*}
where the last equality simply follows by differentiating \eqref{defmuspec} with respect to $t$ and evaluating at $t = u_0^*(s)=v^*(s).$ We hence conclude once again that 
\begin{align*}
u_\e(\Phi_\e(z)) & = (1 + 2 \e \Re G(z)- 2\pi \e \langle X_0(z),z\rangle + O(\e^2) )e^{-\pi |z|^2},\\
 u_\e^*(s) & = \left(1 + \frac{(d!)^{1/d}}{d \cdot s^{1-1/d}} O(\e^2)\right) u_0^*(s).
\end{align*}
Again, since $\Phi_\e (\{u_0 = u_0^*(s)\}) = \{u_\e = u_\e^*(s)\}$ by definition, if one looks at $z \in \{u_0 = u_0^*(s)\}$ and compares the expansion in $\e$ of $u_\e(\Phi_\e(z)) = u_\e^*(s),$ one arrives at 
\begin{equation}\label{eqn:boundary-high-d-vector}
\text{Re}(G) = \pi \langle X_0, z \rangle = \pi \cdot r(\Omega_0)  \langle X_0, \nu \rangle \text{ on } \partial \Omega_0,
\end{equation}
where $r(\Omega_0)$ denotes the radius of $\Omega_0$, which by the definition of $v^*$ in \eqref{defe*}, is given by
\begin{equation*}
    \label{eq:defrOmega0}
    r(\Omega_0) = \left( \frac{(d!s)^{1/d}}{\pi} \right)^{1/2}.
\end{equation*}
Hence, 
\begin{align*}
\label{eq:boundary-term-variation}
\begin{split}
\int_{\partial \Omega_0} \text{Re}(G) \langle X_0, \nu \rangle e^{-\pi |z|^2} \diff \mc H^{2d-1} &= \frac{\mathcal{H}^{2d-1}(\partial \Omega_0) \cdot e^{-\pi r( \Omega_0)^2}}{\pi \cdot r(\partial \Omega_0)} \fint_{\partial \Omega_0} \text{Re}(G)^2 \diff \mc H^{2d-1} \, \cr 
    & = \frac{\mathcal{H}^{2d-1}(\partial \Omega_0) \cdot e^{-\pi r( \Omega_0)^2}}{2 \pi \cdot r( \Omega_0)} \fint_{\partial \Omega_0} |G|^2 \diff \mc H^{2d-1}\\
    &= e^{-(d!s)^{1/d}}  \frac{d \cdot s^{1-1/d}}{(d!)^{1/d}}  \fint_{\partial \Omega_0} |G|^2 \diff \mc H^{2d-1},
    \end{split}
\end{align*}
where the second identity may be justified by the fact that $\Omega_0$ is a ball centered at 0 and $G,\, \overline{G}$ are harmonic functions with $G(0) = 0.$ Thus, as in Section \ref{sec:alternative-proof}, this shows that 
\begin{equation}\label{eq:variational-high-dim}
\frac 1 2 \frac{\diff^{\,2}}{\diff \e^{\,2}}\mc K[1+\e G]\Bigg|_{\e = 0 } 
= \int_{\Omega_0} |G|^2 e^{-\pi |z|^2} - \|G\|_{\mc F^2}^2\int_{\Omega_0} e^{-\pi |z|^2} +  e^{-(d!s)^{1/d}}  \frac{d \cdot s^{1-1/d}}{(d!)^{1/d}} \fint_{\partial \Omega_0} |G|^2 \,.
\end{equation}
We now wish to compute the right-hand side of \eqref{eq:variational-high-dim} for 
$$G(z) = \sum_{i=1}^d z_i^2 +   \sum_{1\le i < j \le d} \sqrt{2} z_i z_j .$$ In order to do so, we note that each monomial in the definition of $G$ is orthogonal to every other monomial not just over $\C^d$ but in fact over any ball centered at the origin. Thus we have 
\begin{equation}\label{eq:second-variation-higher}
 \frac 1 2\frac{\diff^{\,2}}{\diff \e^{\,2}}\mc K[1+\e G]\Bigg|_{\e = 0 }  = \int_{\Omega_0} |z|^4 e^{-\pi |z|^2} - \|G\|_{\mc F^2}^2\int_{\Omega_0} e^{-\pi |z|^2} +  e^{-(d!s)^{1/d}} \cdot \frac{d \cdot s^{1-1/d}}{(d!)^{1/d}}  \fint_{\partial \Omega_0} |z|^4 \,. 
\end{equation}
In order to further analyze \eqref{eq:second-variation-higher}, note the identity 
\[
\int_{\Omega_0} |G|^2 e^{-\pi|z|^2} - \|G\|_{\mc F^2}^2 \int_{\Omega_0} e^{-\pi|z|^2} = - \int_{\C^d \setminus \Omega_0} |G|^2 e^{-\pi|z|^2} + \|G\|_{\mc F^2}^2 \int_{\C^d \setminus \Omega_0} e^{-\pi|z|^2}. 
\]
For our choice of $G$, we can explicitly evaluate these integrals: indeed, a routine computation implies that
\[
\int_{\C^d \setminus \Omega_0} |z|^4 e^{-\pi|z|^2} \, \diff z =  \frac{\Gamma(d+2,\pi r(\Omega_0)^2)}{\pi^2 \Gamma(d)},
\]
while $\|G\|_{\mc F^2}^2 = \frac{d(d+1)}{\pi^2},$ which implies that 
\[
\|G\|_{\mc F^2}^2 \int_{\C\setminus \Omega_0} e^{-\pi|z|^2} \diff z = \frac{ d(d+1)\Gamma(d,\pi r( \Omega_0)^2) }{\pi^2 \Gamma(d)}.
\]
Hence, by using that $\Gamma(k,x) = (k-1)! e^{-x} \left( \sum_{j=0}^{k-1} \frac{x^j}{j!} \right),$ we conclude that 
\[
 \int_{\Omega_0} |z|^4 e^{-\pi |z|^2} \diff z- \|G\|_{\mc F^2}^2\int_{\Omega_0} e^{-\pi |z|^2}  \diff z = -  \frac{d s \cdot e^{-(d!s)^{1/d}}}{\pi^2} \left(  1+d+ (d! s)^{1/d} \right).
\]
Finally, the last term in \eqref{eq:second-variation-higher} may be explicitly computed to be $\frac{d \cdot s}{\pi^2} e^{-(d!s)^{1/d}}(d! s)^{1/d}.$ Thus, plugging these into \eqref{eq:second-variation-higher}, we obtain 
\[
\frac 1 2 \frac{\diff^{\,2}}{\diff \e^{\,2}}\mc K[1+\e G]\Bigg|_{\e = 0 } = - e^{-(d!s)^{1/d}} \cdot d(d+1) \frac{s}{\pi^2}. 
\]
A straightforward adaptation of the arguments from Section \ref{sec:sharpness} shows the desired sharpness of the exponent, as well as the stability of the order of growth of the constant in \eqref{eqn:StabilityFunctionFockDimensiond}. That is, we are able to obtain the following result: 

\begin{numcor}\label{cor:sharpnesshighd} The following assertions hold:
\begin{enumerate}
    \item\label{it:optimaldeltahighd} The factor $\delta(f;\Omega)^{1/2}$ cannot be replaced by $\delta(f;\Om)^{\beta},$ for any $\beta > 1/2,$ in \eqref{eqn:StabilityFunctionDimensiond};
    \item\label{it:optimalconstanthighd} There is no $c\in(0,(d!)^{1/d})$ such that, for all measurable sets $\Omega \subset \C^d$ of finite measure, we have
    $$    \min_{z_0\in \C^d, |c|=\|f\|_{2}} \frac{\|f - c\, \varphi_{z_0} \|_2}{\|f\|_2} \leq C \Big(e^{c |\Omega|^{1/d}} \delta(f;\Om)\Big)^{1/2}. $$
\end{enumerate}
\end{numcor}

Elementary computations reveal that the denominator on the right-hand side of  \eqref{eqn:StabilityFunctionFockDimensiond} behaves as
$$\int_{|\Om|}^\infty e^{-(d!s)^{1/d}} \diff s \approx C_d |\Om|^{\frac{d-1}{d}} e^{-(d! |\Om| )^{1/d}} \quad \text{ as } |\Om|\to \infty,$$
for some explicitly computable constant $C_d>0$. Thus, Corollary \ref{cor:sharpnesshighd}\ref{it:optimalconstanthighd} yields, indeed, the desired optimal dependence of \eqref{eqn:StabilityFunctionFockDimensiond} on $|\Om|$.

\appendix
\section{Proof of Lemma \ref{lemma:control3rdder}}
\label{sec:appendix}

This appendix is dedicated to the proof of the technical Lemma \ref{lemma:control3rdder}. 
Before proceeding with the main part of the proof it is convenient to establish some auxiliary estimates for the vector fields $X_\e$, with flows $\Phi_\e$, which were constructed in Lemma \ref{lemma:flows-specific}.

\begin{numlemma}\label{lemma:flows-specific-estimate}
Let $G\in \mc F^2(\C)$ satisfy \eqref{eq:normG}. Let $\Phi_\e, \Psi_\e$ be as in Lemma \ref{lemma:flows-specific}.
There is $\e_0 =\e_0(s,\|G\|_{\mc F^2})> 0$ and a modulus of continuity $\eta\colon [0,\infty)\to [0,\infty)$, depending on $s$, such that, when $\e\leq \e_0$, 
\begin{align}
    \label{eq:estimateXeps}
    \|X_\e - X_0\|_{C^1} & \leq \eta(\e)\|G\|_{\mc F^2},\\
    \label{eq:estimateflow}
    \|\Phi_\e-\Phi_0\|_{C^1}& \leq \eta(\e) \|G\|_{\mc F^2},\\
    \label{eq:H1estimate}
    |\mc H^1(\p \Omega_\e)- \mc H^1(\p \Omega_0)|& \leq \eta(\e) \|G\|_{\mc F^2},\\
    \label{eq:measureestimate1}
    |\Omega_\e\triangle \Omega_0| & \leq \eta(\e) \|G\|_{\mathcal F^2},\\
    \label{eq:measureestimate2}
    \sigma(\Omega_\e\triangle \Omega_0) & \leq \eta(\e) \|G\|_{\mathcal F^2}.
\end{align}
\end{numlemma}

We note that in \eqref{eq:measureestimate2}, as in the proof of Lemma \ref{lemma:second-variation-K}, we denote for simplicity $\diff \sigma(z) \coloneqq e^{-\pi |z|^2} \diff z.$

\begin{proof}
Let $f_\e$ be as in \eqref{eq:deffeps}. We begin by proving  that
\begin{equation}
    \label{eq:claimappendix}
\|f_\e\|_{C^2(\partial \Omega_0)} \le \eta(\e) \|G\|_{\mc F^2},
\end{equation}
whenever $\e$ is sufficiently small. 

We first argue exactly as in Proposition \ref{prop:convexity}. More precisely, we write $u_0=e^{-\pi |\cdot|^2},$ and so, as in \eqref{eq:first-set-comparison}, we have
\begin{equation}
    \label{eq:setcomparisonappendix}
    \{u_0 > u^*_\e(s) + 2 \e  \} \subset \{  u_\e  > u^*_\e(s) \} \subset \{  u_0 > u^*_\e(s) - 2 \e  \}.
\end{equation}
Thus we see that
\begin{equation}
    \label{eq:uepsstar}
    |u_\e^*(s) - e^{-s} | \leq 2\e 
\end{equation}
and, if $u_\e(z)\geq u_\e^*(s)$, we have
\begin{equation}
    \label{eq:radiicomparisonappendix}
    s\leq \pi |z|^2 \leq - \log(u_\e^*(s)-2 \e)\leq s+8 e^s \e 
\end{equation}
for all $\e$ sufficiently small (depending on $s$). In particular, we may take $z\in \p\Omega_\e$ and so $z=(1+f_\e(\omega))\omega$ for some $\omega \in \p \Omega_0$, thus $\pi |z|^2 = \pi |\omega|^2 (1+f_\e(\omega))^2 = s(1+f_\e(\omega))^2$, and the last inequalities yield
$$1 \leq (1+f_\e(\omega))^2 \leq 1+ 8 e^s s^{-1} \e.$$
Since $1+f_\e(\omega)\geq 0$, this easily implies the $C^0$-estimate
\begin{equation}\label{eq:upper-bound-f_eps} 
|f_\e(\omega)| \leq 4 e^s s^{-1} \e .
\end{equation} 

We  now prove the $C^1$-estimate for $f_\e$ on $\p \Omega_0$; in order to simplify a bit the notation, we assume that $\p \Omega_0=\mb S^1$ and so we write $\om = e^{i \theta}.$ Then $u_\e(\Phi_\e(\om)) = u_\e^*(s)$ implies  
\begin{align}
\label{eq:derivative-f_eps}
\begin{split}
0   & = \partial_\theta (u_\e^*(s)) = \partial_\theta \left(u_\e(\Phi_\e(e^{i\theta})) \right) \cr 
    & = \left( 2 \e \Re[G'(\Phi_\e) \partial_\theta (\Phi_\e(e^{i \theta}))] + 2\e^2 \Re[ G'(\Phi_\e) \partial_\theta (\Phi_\e(e^{i\theta}))\,
    \overline{G(\Phi_\e)}] \right) e^{-\pi|\Phi_\e|^2} \cr 
    & \quad + \left(1 + 2\e \Re(G(\Phi_\e)) + \e^2 |G(\Phi_\e)|^2 \right)(- \pi \partial_\theta|\Phi_\e(e^{i\theta})|^2) e^{-\pi |\Phi_\e|^2} ,
    \end{split}
\end{align}
which can be rewritten as 
\begin{equation}
    \label{eq:rewrite-derivative-f_eps}
    |1+\e G(\Phi_\e)|^2 \p_\theta|\Phi_\e(e^{i\theta})|^2 
    = \frac{2 \e}{\pi}\Re \left(G'(\Phi_\e) \partial_\theta (\Phi_\e(e^{i\theta})) \left(1+ \e \overline{G(\Phi_\e)} \right) \right).
\end{equation}
By the explicit formula for $\Phi_\e$, cf.\ \eqref{eq:defPsi} and \eqref{eq:defXeps}, we have
\begin{align}
\label{eq:auxidentitiesfeps}
\begin{split}
\partial_\theta (\Phi_\e(e^{i\theta}))& = ie^{i\theta}(1+f_\e(e^{i\theta})) + ie^{2i\theta} \partial_\theta (f_\e(e^{i\theta})),\\ \partial_\theta |\Phi_\e(e^{i\theta})|^2 & = 2ie^{i\theta}(1+f_\e(e^{i\theta}))\partial_\theta (f_\e(e^{i\theta})).    
\end{split}
\end{align}
We now argue essentially as in Lemma \ref{lemma:Level-Set-Regularity}. More precisely,  we note that if $\e$ is sufficiently small (depending on $\|G\|_{\Fock}$) then 
$$\frac 1 2 \leq |1+\e G(\Phi_\e)| \leq \frac 3 2, \qquad \frac 1 2 \leq |1+f_\e(e^{i \theta})|\leq \frac 3 2, $$
where the last bounds follow from \eqref{eq:upper-bound-f_eps}. Thus, taking absolute values in \eqref{eq:rewrite-derivative-f_eps} and then inserting identities \eqref{eq:auxidentitiesfeps}, we obtain 
\begin{align}
    \label{eq:bound-first-derivative-f_eps-aux}
     \begin{split}
        |\p_\theta(f_\e(e^{i \theta}))| & 
        \leq C \e |\p_\theta(\Phi_\e(e^{i \theta}))| |G'(\Phi_\e)|\left(1+ \e |G(\Phi_\e)|\right)\\
        & \leq C \e \|G\|_{\Fock}\left(1+ |\p_\theta(f_\e(e^{i \theta}))| + \e \|G\|_{\Fock}\right).
    \end{split}
\end{align}
Here, to pass to the last line, we also used the estimate 
$|G'(\Phi_\e)| \leq C \|G\|_{\Fock}$
which follows from the Cauchy integral formula and the $C^0$-estimate \eqref{eq:upper-bound-f_eps}, exactly as in \eqref{eq:bound-derivative-G}.
By choosing $\e$ sufficiently small, we can absorb the term $|\p_\theta(f_\e(e^{i \theta}))|$ on the right-hand side of \eqref{eq:bound-first-derivative-f_eps-aux} into its left-hand side, and so we obtain
\begin{equation}\label{eq:bound-first-derivative-f_eps}
|\partial_\theta (f_\e(e^{i\theta}))| \le \eta(\e) \|G\|_{\mc F^2}.
\end{equation}
By differentiating \eqref{eq:derivative-f_eps} once more with respect to $s$, repeating the argument above and using the bound \eqref{eq:bound-first-derivative-f_eps}, one likewise obtains $|\partial_\theta^2 f_\e(e^{i\theta})| \le  \eta(\e) \|G\|_{\mc F^2}$, as claimed in \eqref{eq:claimappendix}. Since the details are essentially the same and bear no real insight, we omit them.

Having \eqref{eq:claimappendix} at our disposal, estimate \eqref{eq:estimateXeps} follows easily from the explicit formula for the fields $X_\e$ in \eqref{eq:defXeps} and   \eqref{eq:estimateflow} follows immediately from \eqref{eq:estimateXeps} by ODE theory; alternatively, one may also argue this directly from the explicit form of the underlying vector fields in Lemma \ref{lemma:flows}. 

To prove the remaining estimates we return to \eqref{eq:setcomparisonappendix} and \eqref{eq:uepsstar} to see that $\Omega_\e=\{u_\e>u_\e^*(s)\}$ is nested between two balls:
$$B_\e^1 := \{ u_0 > e^{-s} + 4\e \} \subset \Omega_\e 
\subset \{ u_0 > e^{-s} - 4\e   \}=: B_\e^2. 
$$
Proposition \ref{prop:convexity} ensures that, by choosing $\e_0$ sufficiently small, $\Omega_\e$ is convex, hence
$$\mc H^1(\partial B_\e^1) \leq \mc H^1(\p \Omega_\e) \leq \mc H^1(\p B_\e^2).$$
Let us denote by $R_\e^1, R_\e^2$ the radii of $B_\e^1,B_\e^2$ respectively. 
If we choose $\e_0$ sufficiently small, depending on $s$, then it is easy to see that
\begin{equation}
\label{eq:estimateradii}
s- 8 e^s \e  \leq \pi(R_\e^1)^2 \leq \pi(R_\e^2)^2 \leq s+8  e^s \e,
\end{equation}
similarly to \eqref{eq:radiicomparisonappendix}.
This immediately yields 
\begin{align*}
4 \pi(s - 8 e^s \e  ) \leq \mc H^1(\p B_\e^1)^2 \leq \mc H^1(\p B_\e^2)^2 & \leq 4 \pi(s + 8 e^s\e),
\end{align*}
and, as $\mc H^1(\p \Omega_0)^2 = 4 \pi s$, \eqref{eq:H1estimate} follows. Similarly, for \eqref{eq:measureestimate1}, we have
\begin{align*} 
|\Omega_\e \triangle \Omega_0| & \le |B_\e^2 \setminus B_\e^1|  \le \pi \left( (R_\e^2)^2-  (R_\e^1)^2\right) \leq 16 e^s \e .
\end{align*} 
Finally, \eqref{eq:measureestimate2} follows from \eqref{eq:measureestimate1} through $\sigma(\Omega_\e\Delta \Omega_0)\leq |\Omega_\e \Delta \Omega_0|$.
\end{proof}

\begin{proof}[Proof of Lemma \ref{lemma:control3rdder}]
Returning to the beginning of the proof of Lemma \ref{lemma:second-variation-K}, and in particular to \eqref{eq:derK0}, we can write, using Taylor's theorem, 
\begin{equation*}
    K_\e = K_0 + \frac{\e^2}{2} K''_0 + \int_0^\e (\e-s)(K''_s-K''_0) \diff s, 
\end{equation*}
and thus our task is to show that, for all $\e\in[0,\e_0],$ we have
\begin{equation}
    \label{eq:decay}
    |K''_\e-K''_0| \leq \eta(\e)
\end{equation}
for a suitable modulus of continuity $\eta$. We recall that, as in the proof of Lemma \ref{lemma:second-variation-K}, primes denote derivatives with respect to $\e$, and we also recall the definition of $I_\e$ and $J_\e$ from \eqref{eq:defIJ}.

Although in the statement of the lemma we assumed that $\|G\|_{\mc F^2}=1$, for the sake of clarity we will still write $\|G\|_{\mc F^2}$ explicitly in our estimates. Since $\langle G, 1\rangle_{\mathcal F^2}=0$ by \eqref{eq:normG}, we have
\begin{equation}
\label{eq:4thterm}
J_\e - J_0 = \|1+\e G\|_{\mathcal F^2}^2 - \|1\|_{\mathcal F^2}^2 = \e^2 \|G\|^2_{\mc F^2} = \e^2,
\end{equation}
and so $J_\e \geq J_0 = 1$.  By \eqref{eq:derK}, the function $K_\e ''$, seen as a function of $(I_\e, I_\e',I_\e'', J_\e, J_\e',J_\e'')$,  is smooth in the set $\{J_\e >0\}$ and therefore there is a constant $C$ such that
\begin{equation}
\label{eq:lipestimate}
\left|K_\e'' - K_0''\right| \leq C
(|I_\e - I_0| + |I'_\e - I'_0| + |I''_\e- I''_0| + |J_\e-J_0| + |J'_\e-J'_0|).
\end{equation}
since $J''_\e = J''_0$ is independent of $\e$, cf.\ \eqref{eq:derIJsecond}.  It now suffices to estimate each of the terms on the right-hand side of \eqref{eq:lipestimate}. In \eqref{eq:4thterm} have already estimated $|J_\e-J_0|$ and, similarly, we have
\begin{equation}
\label{eq:5thterm}
|J'_\e - J'_0| = 2 \int_\C \langle G, \e G\rangle \diff \sigma \leq 2 \e \|G\|^2_{\mathcal F^2},
\end{equation}
where we used the identities in \eqref{eq:derIJ}.
Thus it remains to estimate the first three terms in \eqref{eq:lipestimate}.
 
 For the first term, again since $\langle G, 1\rangle_{\mathcal F^2}=0$, we estimate using the fundamental theorem of calculus:
\begin{align}
\label{eq:1stterm}
\begin{split}
|I_\e-I_0| & \leq \e \sup_\e I'_\e \leq 2\e |\langle G,1+ \e G\rangle_{\mc F^2} |= 2 \e^2 \|G\|_{\mc F^2}^2.
\end{split}
\end{align}
For the second term,  with the help of \eqref{eq:measureestimate2},  we estimate
\begin{align}
\label{eq:2ndterm}
\begin{split}
|I'_\e-I'_0| & = 2\left|\langle G, (1+\e G)1_{\Omega_\e} - 1_{\Omega_0}\rangle_{\mc F^2}\right|\\
& \leq 2 \|G \|_{\mc F^2} \sigma(\Omega_\e \Delta \Omega_0) + 2 \e \|G\|_{\mc F^2}^2\\
& \leq \eta(\e) \|G\|_{\mc F^2}^2.
\end{split}
\end{align}
Finally, we arrive at the third term:
\begin{align*}
    \frac 1 2(I''_\e - I''_0) & = \int_\C |G|^2 (1_{\Omega_\e}- 1_{\Omega_0}) \diff \sigma  + \int_{\p \Omega_\e} \langle G, 1+\e G\rangle \langle X_\e, \nu_\e\rangle e^{-\pi |z|^2} \diff \mc H^1(z)\\
& \quad - \int_{\p \Omega_0} \langle G, 1\rangle \langle X_0, \nu_0\rangle e^{-\pi |z|^2}\diff \mc H^1(z).
\end{align*}
The first term on the right-hand side is easily estimated  using \eqref{eq:measureestimate1}:
$$
\int_\C |G|^2(1_{\Omega_\e}- 1_{\Omega_0}) \diff \sigma 
\leq \|G\|_{\mc F^2}^2 |\Omega_\e \Delta \Omega_0| \leq \eta(\e) \|G\|_{\mc F^2}^3.
$$
For the last two terms, we write
\begin{align*}
\int_{\p \Omega_\e} \langle G, 1+\e G\rangle \langle X_\e, \nu_\e\rangle e^{-\pi |z|^2} \diff \mc H^1(z)
- \int_{\p \Omega_0} \langle G, 1\rangle \langle X_0, \nu_0\rangle e^{-\pi |z|^2} \diff \mc H^1(z)& = A_1 + A_2 + A_3, 
\end{align*}
where
\begin{align*}
    A_1 & := \int_{\p \Omega_\e} \langle G, 1+\e G\rangle e^{-\pi |z|^2} (\langle X_\e, \nu_\e\rangle - \langle X_0, \nu_0\rangle )\diff \mc H^1(z),\\
    A_2 & := \int_{\p \Omega_\e} \e |G|^2 \langle X_0,\nu_0\rangle e^{-\pi |z|^2}\diff \mc H^1(z),\\
    A_3 & := \int_{\p \Omega_\e} \langle G, 1\rangle e^{-\pi |z|^2} \langle X_0, \nu_0\rangle -\int_{\p \Omega_0} \langle G, 1\rangle e^{-\pi |z|^2} \langle X_0, \nu_0\rangle\diff \mc H^1(z) .
\end{align*}
Here, we have extended $\nu_\e$ to a neighbourhood of $\partial \Om_0$ as $\nu_\e = - \frac{\nabla u_\e}{|\nabla u_\e|}.$ With this definition, it follows from a computation similar to that of Lemma \ref{lemma:Level-Set-Regularity} that $\|\nu_\e - \nu_0\|_{L^{\infty}} \le \tilde{C}(s) \e \|G\|_{\mc F^2}$ in a neighbourhood of $\partial \Omega_0.$ Thus, on $\p \Omega_\e$, 
$$|\langle X_\e,\nu_\e\rangle - \langle X_0,\nu_0\rangle | \leq |\langle X_\e, \nu_\e-\nu_0\rangle| + | \langle X_\e-X_0, \nu_0\rangle|\leq C(s)\e \|G\|_{\Fock} + \eta(\e) \|G\|_{\Fock} $$
and so, using \eqref{eq:estimateXeps} and \eqref{eq:H1estimate}, we estimate
\begin{gather*}
    |A_1|\leq C(s)\e \|G\|_{\Fock} \int_{\p\Omega_\e} |G|(1+\e |G|) e^{- \pi|z|^2} \diff \mc H^1(z) \leq C(s) \eta(\e) (\|G\|_{\mc F^2}^2+ \|G\|_{\mc F^2}^3).
\end{gather*}
Similarly, by \eqref{eq:estimateXeps} and \eqref{eq:vector-field-X_0} we have
$$|A_2|\leq \e \|G\|^2_{\Fock} \int_{\p \Omega_\e} \langle X_0,\nu_0\rangle \diff \mc H^1 \leq  C(s)\e \|G\|_{\mc F^2}^2.$$
Finally, we note that
\begin{align*}
    & \int_{\p \Omega_\e} \langle G, 1\rangle e^{-\pi |z|^2} \langle X_0, \nu_0\rangle \diff \mc H^1(z)\\
    & \qquad = \int_{\p \Omega_0} |\nabla \Phi_\e(\omega)| \langle G(\Phi_\e(\omega)),1\rangle \langle X_0(\Phi_\e(\omega)),\nu_0(\Phi_\e(\omega))\rangle e^{-\pi |\Phi_\e(\omega)|^2} \diff \mc H^1(\omega),
\end{align*}
and so, using \eqref{eq:vector-field-X_0}, we have
$$A_3= \frac{2}{\mc H^1(\p \Omega_0)} \int_{\p \Omega_0} |\nabla \Phi_\e| (\Re G(\Phi_\e))^2 e^{-\pi |\Phi_\e|^2} - (\Re G)^2 e^{-\pi |\cdot|^2} \diff \mc H^1.$$
By \eqref{eq:estimateflow} we have $\left||\nabla \Phi_\e|-1\right| \leq \eta(\e)\|G\|_{\Fock}$, and so it suffices to estimate the function
$$g(z) \coloneqq (\Re G(z))^2 e^{-\pi |z|^2}.$$
For $z\in \p \Omega_0$, we have
\begin{align*}
    |\nabla g(z)| & \leq (2 |G(z)| |G'(z)| + 2\pi |z| |G(z)|^2 )e^{-\pi |z|^2} \\
    & \leq ( C(s) |G(z)|^2 + 2\pi |z| |G(z)|^2)e^{-\pi |z|^2}  \leq C(s) \|G\|_{\Fock}^2,
\end{align*}
where to pass to the second line we used the fact that $G$ is holomorphic together with Cauchy's integral formula, as in \eqref{eq:bound-derivative-G}. Hence, using again \eqref{eq:estimateflow}, and as $\Phi_0(\omega)=\omega$, we have
$$|g(\Phi_\e(\omega)) - g(\omega)|\leq \|\nabla g\|_{L^\infty} |\Phi_\e(\omega) - \omega|\leq C(s) \eta(\e) \|G\|^3_{\Fock},$$
which yields immediately
$$|A_3|\leq \eta(\e)\|G\|_{\mc F^2}^2.$$
Hence, combining the last estimates, and up to replacing $\eta$ with a new modulus of continuity, we have
\begin{equation}
    \label{eq:3rdterm}
    |I''_\e - I''_0| \leq \eta(\e) \|G\|^2_{\mc F^2}.
\end{equation}
The desired decay \eqref{eq:decay} now follows by combining \eqref{eq:4thterm}--\eqref{eq:3rdterm}.
\end{proof}

\nocite{*}                                  
\printbibliography                          
\end{document}